\documentclass[12pt]{amsart}
\usepackage[dvipsnames]{xcolor}
\usepackage{fullpage,graphicx,psfrag,amsmath,amsfonts, amssymb,enumitem, hyperref, mathrsfs}

\usepackage{forest}
\usetikzlibrary{decorations.pathreplacing, positioning, arrows}

\numberwithin{equation}{section}

\DeclareMathOperator{\R}{\mathbb{R}}
\DeclareMathOperator{\C}{\mathbb{C}}
\DeclareMathOperator{\Z}{\mathbb{Z}}
\DeclareMathOperator{\N}{\mathbb{N}}

\DeclareMathOperator{\symm}{{\mbox{\bf S}}}  
\DeclareMathOperator{\diam}{{\mbox{diam}}}

\DeclareMathOperator{\Tr}{\mathop{\bf Tr}}

\DeclareMathOperator{\dist}{\mathop{\bf dist{}}}

\newcommand{\eg}{{\it e.g.}}
\newcommand{\ie}{{\it i.e.}}

\newtheorem{theorem}{Theorem}[section]
\newtheorem{prop}[theorem]{Proposition}

\newtheorem{lemma}[theorem]{Lemma}

\newtheorem{remark}[theorem]{Remark}
\newtheorem{definition}[theorem]{Definition}
\newtheorem{problem}[theorem]{Problem}
\newtheorem{conj}[theorem]{Conjecture}

\newcommand{\cl}{\mathbf{cl}}
\DeclareMathOperator{\inte}{\mathbf{int}}

\renewcommand*{\P}{\mathbb P}
\DeclareMathOperator{\E}{\mathop{\mathbb E{}}}

\renewcommand*{\S}{\mathbb S}
\DeclareMathOperator*{\osc}{osc}

\newcommand{\indc}{{\boldsymbol{1}}}
\newcommand*{\Zd}{\ensuremath{\mathbb{Z}^d}}
\newcommand*{\Rd}{\ensuremath{\mathbb{R}^d}}
\newcommand*{\supp}{\ensuremath{\mathrm{supp\,}}}
\newcommand{\cu}{\square}
\newcommand{\size}{\mathrm{size}}

\newcommand*\xxbar[1]{%
	\hbox{%
		\vbox{%
			\hrule height 0.7pt % The actual bar
			\kern0.25ex%         % Distance between bar and symbol
			\hbox{%
				\kern+0em%      % Shortening on the left side
				\ensuremath{#1}%
				\kern-0.1em%      % Shortening on the right side
			}%
		}%
	}%
}

\hypersetup{
	colorlinks=true,
	linkcolor=NavyBlue,
	urlcolor = RoyalBlue,
	citecolor = PineGreen
}

\begin{document}
	\title{Rigidity of harmonic functions on the supercritical percolation cluster}
	\author{Ahmed Bou-Rabee}
	\author{William Cooperman}
	\author{Paul Dario}
	
	\begin{abstract}
		We use ideas from quantitative homogenization to show that nonconstant harmonic functions on the percolation cluster cannot satisfy certain structural constraints, for example, a Lipschitz bound. These unique-continuation-type results are false on the full lattice and hence the disorder is utilized in an essential way. 
	\end{abstract}
	
	\maketitle
	\begin{figure}
		\begin{center}
			\includegraphics[width = 0.25\textheight]{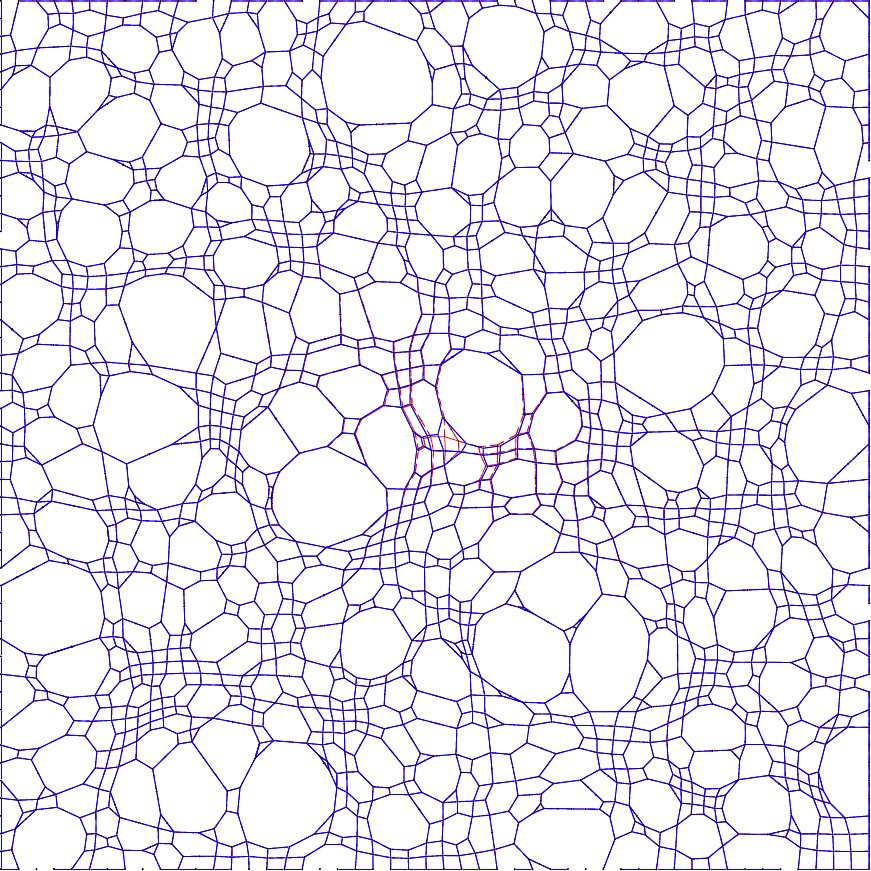} \qquad
			\includegraphics[height = 0.25\textheight]{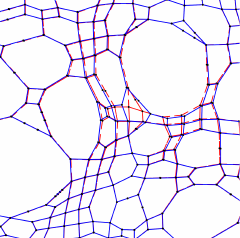}
		\end{center} \caption{The effect of changing the value of an edge. On the left is the harmonic (barycentric) embedding of the cluster for $\mathfrak{p} = 0.8$ on a cube of side length $N=50$ superimposed with the embedding after one edge is re-sampled. The original embedding is in blue and the re-sampled embedding is shown in green. On the right is a zoomed in piece of the re-sampled embedding.}\label{fig:flip-edge}
	\end{figure} 
	
	\section{Introduction}
	\subsection{Overview}
	We consider harmonic functions on the supercritical bond percolation cluster of the integer lattice.  Much of the study of these functions has been devoted to showing that they share large scale features, \eg, a Harnack inequality \cite{barlow-random-walk-percolation} and polynomial approximation \cite{armstrong-dario-2018}, with their counterparts on $\Z^d$ and $\R^d$. Such results are often viewed probabilistically; the convergence of random walk on the cluster to Brownian motion, as established in  \cite{berger-rw-percolation, mathieu-rw-percolation} is closely related to the fact that harmonic functions on the cluster approximate continuum harmonic functions.
	
	In this article, we focus on the differences between the percolation cluster and $\Z^d$ or $\R^d$ by proving three unique-continuation-type theorems on the cluster which are false on both the full lattice and the continuum. We are guided by the principle that harmonic functions on the percolation cluster should appear fairly generic at small scales. In Figure \ref{fig:flip-edge}, we illustrate the rippling effect of closing one bond on the family of linear harmonic functions. Closing or opening a bond will typically have a small but nonzero effect, preventing many fine properties (for example, linear relations between values of a harmonic function at nearby sites) from holding. 
	
	Our work, although self-contained, is motivated by the Abelian sandpile \cite{levine-what-is, levine-peres-survey}  and we discuss applications to the model briefly in Section \ref{subsec:sandpile-intro} and in detail in Section \ref{sec:abelian-sandpile} below. 
	
	\subsection{Main results}
	Let $E\left( \Z^d \right)$ denote the set of {\it edges}, \ie, unordered pairs $\{x,y\} \in \Z^d$ satisfying $|x-y| =1$ and associate to each edge
	an i.i.d\ Bernoulli random variable 
	\[
	\mathbf{a}: E\left( \Z^d \right) \to \{0, 1\}, \quad \mbox{with } \P[\mathbf{a}(e) = 1] = \mathfrak{p}
	\]
	where $\mathfrak{p} \in (\mathfrak{p}_c(d), 1)$, and $\mathfrak{p}_c(d) \in (0,1)$ is the bond percolation threshold for the lattice $\Z^d$. 
	Consequently, there is a unique infinite connected component denoted by $\mathscr{C}_{\infty}$.

	A function on the cluster $u: \mathscr{C}_{\infty} \to \R$ is {\it harmonic} if 
	\begin{equation}
		\Delta_{{\mathscr{C}_{\infty}}} u = 0 \quad \mbox{in $\mathscr{C}_{\infty}$} 
	\end{equation}
	where the {\it graph Laplacian} $\Delta_{\mathscr{C}_{\infty}}$ is defined as
	\begin{equation}
		\Delta_{{\mathscr{C}_{\infty}}} u(x) := \sum_{y \sim x} (u(y) - u(x)), 
	\end{equation}
	and the sum $y \sim x$ is over the edges $(x,y) \in E \left( \mathscr{C}_\infty \right)$.

	For our first theorem, we say that a function $u: \mathscr{C}_{\infty} \to \R$ is {\it Lipschitz} if there exists a constant $C$ such that
	\begin{equation} \label{eq:lipschitz-def}
		|u(x) - u(y)| \leq C \dist_{\mathscr{C}_{\infty}}(x,y), \quad \forall x,y \in \mathscr{C}_{\infty},
	\end{equation}
	where $\dist_{\mathscr{C}_{\infty}}$ refers to the graph distance on the cluster. 
	On the full lattice, $\Z^d$, and the continuum, $\R^d$, many harmonic functions are Lipschitz: take, for instance, $x \to (x_1 - x_2)$. However, harmonic functions on $\Z^d$ are generally not harmonic
	on the cluster; in fact, it is not obvious how to construct any nontrivial harmonic function on the cluster.  
	
	It was shown by Benjamini, Duminil-Copin, Kozma, and Yadin in \cite{disorder-entropy-2015} that the space of linearly growing harmonic functions on the cluster has the same dimension as that of $\R^d$. Their construction was not explicit and they asked \cite{disorder-entropy-2015-question} whether there are nontrivial {\it Lipschitz} harmonic functions on the cluster.  We show that the only Lipschitz harmonic functions on the cluster are constant. 
	
	Our proof relies on the large-scale regularity of harmonic functions on the percolation cluster established in~\cite{armstrong-dario-2018}. This regularity is used to measure how changing a bond's value (as in Figure~\ref{fig:flip-edge}) affects harmonic functions of linear growth.
	\begin{theorem} \label{theorem:lipschitz}
		Almost surely, if $u: \mathscr{C}_{\infty} \to \R$ is Lipschitz and harmonic then 
		$u \equiv c$ for some $c \in \R$. 
	\end{theorem}
	
	The Lipschitz function $x \to (x_1 - x_2)$ is also integer-valued. The delicate dependence of harmonic functions on the structure of open and closed bonds suggests that there are also no nontrivial integer-valued harmonic functions 
	on the cluster. We prove this under the additional assumption of linear growth, a technical condition which we conjecture can be replaced by any deterministic growth, see Problem \ref{problem:all-growth} below.
	\begin{theorem} \label{theorem:integer-valued-linear-growth}
		Almost surely, if $u: \mathscr{C}_{\infty} \to \R$ is harmonic, integer-valued, and grows at most linearly, \ie,
		$u(x) = O(|x|)$,  then $u \equiv a$ for some $a \in \Z$. 
	\end{theorem}

	To motivate our next theorem, recall that any bounded harmonic function on $\Z^d$, $\R^d$, or $\mathscr{C}_{\infty}$ is constant. This is, for instance, a consequence 
	of the elliptic Harnack inequality. However, this Liouville theorem is false on $\Z^d$ and $\R^d$ if one relaxes the harmonic constraint 
	to harmonic outside of a finite set. For example, the discrete derivative of the elliptic Green's function on $\Z^d$ decays to zero at infinity. 
	In fact, by iteratively taking finite differences of the elliptic Green's function,  \cite{schmidt-verbitskiy-2009,sandpiles-square-lattice}, 
	one can construct a function on $\Z^d$ that is harmonic outside a finite set and decays faster than any given inverse polynomial,
	but is not compactly supported.

	Our next theorem states that this construction fails on the cluster. Specifically, we show on the two-dimensional supercritical percolation cluster that any function with an integer-valued Laplacian that decays faster than $|x|^{-1}$ is zero outside a finite set. 
	\begin{theorem} \label{theorem:fast-decay}
		The following holds for $d = 2$. Almost surely and for all functions with finite support, $f: \Z^d \to \Z$, if $u: \mathscr{C}_{\infty} \to \R$ satisfies 
		\[
		\begin{aligned}
			-\Delta_{{\mathscr{C}_{\infty}}} u &= f \\ 
			\limsup_{R \to \infty} \, R &\left (\sup_{B_R \setminus B_{R/2}} |u|  \right) = 0,
		\end{aligned}
		\]
		then the support of $u$ is finite.  
	\end{theorem}
	In fact, we characterize the possible decay rates of functions with integer-valued Laplacians on the cluster; see Theorem \ref{theorem:fast-decay-stronger} below. We conjecture that Theorem \ref{theorem:integer-valued-linear-growth} holds in all dimensions and discuss a possible route to proving this under Problem \ref{problem:all-dimensions} below.

	\subsection{Abelian sandpile} \label{subsec:sandpile-intro}
	As mentioned earlier, this investigation is partly motivated by the desire to understand the Abelian sandpile model \cite{levine-what-is, levine-peres-survey}
	on the supercritical percolation cluster. Significant differences between the model on the full lattice and on the cluster can be observed in numerical experiments --- this has been documented in the physics literature by Sadhu and Dhar \cite{sadhu2011effect}.  Until this article, nothing to this effect had been rigorously demonstrated. 
	
	We defer a definition of the model to Section \ref{sec:abelian-sandpile} below. For now, we note that Theorem \ref{theorem:fast-decay-stronger} gives a complete description of the so-called toppling invariants \cite{dhar-algebraic-aspects} of the model on the cluster
	which is considerably different from that of $\Z^2$. Moreover, integer-valued harmonic functions on periodic graphs lay the foundations for the fractal patterns appearing in the Abelian sandpile \cite{pegden-smart-sandpile-2013, bourabee-sandpile-2021, apollonian-matrix, apollonian-sandpile, bou2021integer}. The absence of such functions, Theorem \ref{theorem:integer-valued-linear-growth}, provides a partial explanation for why such patterns do not appear on the cluster.
	
	In Section \ref{sec:abelian-sandpile} we also prove a new `slow mixing' theorem for the sandpile Markov chain on the cluster.
	\subsection{Open questions}
	As mentioned previously, we expect the assumption of linear growth in Theorem \ref{theorem:integer-valued-linear-growth} 
	can be replaced by that of any deterministic growth. 
	\begin{problem} \label{problem:all-growth}
		Show that there are no non-constant integer-valued harmonic functions on the cluster (of any growth).
	\end{problem}
	In fact, we prove a weaker statement for integer-valued functions of quadratic growth; see Section \ref{subsec:polynomial-growth} below. 
	
	Similarly, we expect that Theorem \ref{theorem:fast-decay} is true in all dimensions $d \geq 2$
	\begin{problem} \label{problem:all-dimensions}
		Prove Theorem \ref{theorem:fast-decay} in dimensions $d > 2$.
	\end{problem}
	A large part of the proof of Theorem \ref{theorem:fast-decay} carries over verbatim when the dimension $d$ is arbitrary, except for one step which contains arguments that depend crucially on the planarity of the lattice, Proposition~\ref{prop:level-set-of-harmonic-function}.  We present a more general version of Proposition~ \ref{prop:level-set-of-harmonic-function} explicitly as a problem, which we believe to be independently interesting. For the statement of this problem and the next, 
	we say that a set $A \subset \mathscr{C}_{\infty}$ on the cluster has {\it positive density} if 
	\begin{equation} \label{eq:has-density}
		\liminf_{N \to \infty} |B_N|^{-1}  |A \cap B_N| > 0,
	\end{equation}
	where $B_N$ is the ball of radius $N$ centered at the origin. 
	\begin{problem} \label{problem:level-set}
		Consider the cluster in dimensions $d > 2$ and let $u: \mathscr{C}_{\infty} \to \R$ be a function which is harmonic outside of a finite set and which decays to zero at infinity, $\lim_{|x| \to \infty} u(x) = 0$.  Show that if the set of edges $e$ such that $\{ \nabla u(e) = 0\}$ has positive density, then $u$ must vanish outside  of a finite set. 
	\end{problem}
	This distinction between two and higher dimensions has appeared previously in the study of singular sets and unique continuation of elliptic equations \cite{logunov-malinnikova-survey, han-singular-sets}. For instance, the solution of a divergence-form elliptic equation with H\"older continuous coefficients which vanishes on an open set is identically zero 
	in two dimensions. However, this property may fail in dimensions higher than two \cite{miller-unique-continuation}. 
	
	The following question seems to be closely related. 
	\begin{problem} \label{problem:unique-continuation}
		Let $h$ be a harmonic function on $\mathscr{C}_{\infty}$. Show that if the set of sites in the cluster where $h$ is bounded has positive density, then $h$ is a constant. 
	\end{problem}
	On $\Z^2$, it was shown in a spectacular work by Buhovsky, Logunov, Malinnikova, and Sodin that if the set of sites where $h$ is bounded has a sufficiently 
	high density, then $h$ is constant \cite{blms-harmonic}. This was then applied to Anderson localization in the breakthrough papers 
	\cite{ding-smart-localization, li-zhang-localization}. 
	
	Buhovsky, Logunov, Malinnikova, and Sodin demonstrate via counterexamples that their result fails for a low density on $\Z^2$ and for any density on $\Z^d$ for $d \geq 3$.  Their arguments utilize the exact symmetries of $\Z^d$. In particular, the counterexamples in \cite{blms-harmonic} show that a solution to Problem \ref{problem:unique-continuation} must use the randomness of the cluster.

	\subsection{Method and paper outline}
	Each of our results may be thought of as a consequence of the following (intentionally ambiguous) claim: almost surely, the supercritical percolation cluster has no symmetries. Our proof strategy is thus, roughly, to proceed by contradiction. We suppose that a function with some structure on the percolation cluster exists and then identify a specific asymmetry which this contradicts. 
	
	An early example of this technique appears in a paper of Chayes, Chayes, Franz, Sethna, and Trugman on the quantum percolation problem \cite{chayes-quantum-percolation}.
	There they used the fact that every finite subgraph of $\Z^d$ appears in the percolation cluster with positive density to construct finitely supported eigenfunctions of the Laplacian. 
	Other works which establish qualitative differences between the cluster and the lattice include  \cite{fastmixing-ising} and \cite{effective-resistance-percolation-cluster}.

	Our main technical innovation in this paper is in the use of quantitative stochastic homogenization. We use these ideas together with ergodicity arguments to identify specific events which preclude harmonic functions with structure from existing on the cluster. The proofs of each of our theorems, while sharing the common thread of quantitative homogenization, can be read independently.
	
	We collect several common preliminary results in Section \ref{sec:preliminaries} on percolation, homogenization, and topology. 
	In Section \ref{sec:lipschitz} we prove Theorem \ref{theorem:lipschitz}, in Section \ref{sec:integer-harmonic}, Theorem \ref{theorem:integer-valued-linear-growth},
	and in Section \ref{sec:toppling-invariants}, Theorem \ref{theorem:fast-decay}. The beginning of each section contains an outline of each proof.  
	This arrangement is roughly in the order of technicality; in particular, the proof of Theorem \ref{theorem:fast-decay} is the most involved and occupies the bulk of this article. 
	
	We conclude in Section \ref{sec:abelian-sandpile} with a discussion of how our results relate to the Abelian sandpile. This section is mostly expository and contains several open questions which stem from this work. The main contribution of this section is a `slow mixing' result for the sandpile Markov chain stated in Theorem \ref{theorem:slow-mixing-percolation} below. 
	
	\subsection*{Acknowledgments}
	We thank an anonymous referee for helpful comments on an earlier version of this article.
	We also thank Scott Armstrong, Ewain Gwynne, Lionel Levine, Charles Smart, Philippe Sosoe, and Ariel Yadin for useful discussions.
	A.B. was partially supported by NSF grant DMS-2202940 and a Stevanovich fellowship.

	\section{Preliminaries} \label{sec:preliminaries}
	This section, which consists of preliminary results, may be skipped on a first read. The reader may read the subsequent sections 
	and refer back to the relevant results here when necessary.

	\subsection{Basic notation and assumptions} \label{subsec:notation}
	\begin{itemize}
		\item General notation.
		\begin{itemize} 
			\item Inequalities/equalities between functions/scalars are interpreted pointwise.
			\item Constants $C, c$ are finite and positive and may change from line to line. Dependence on other constants is indicated by, \eg, $C(d)$
			and when constants need to be distinguished we write, \eg, $c_1$ and $C_2$. 
			\item For $x \in \R^d$, $x_i$ denotes the $i$-th coordinate of $x$ and $e_1, \ldots, e_d$ the standard basis of $\R^d$.
			\item $|x|$, $|x|_2$, and $|x|_{\infty}$ respectively denote the $\ell^1$ norm, $\ell^2$ norm and $\ell^{\infty}$ norm respectively: $\sum |x_i|$, $\sqrt{\sum (x_i)^2}$, and $\max |x_i|$.
			\item A \emph{domain} or \emph{region} of $\Rd$ is a nonempty connected open subset of $\Rd$.
			\item Given a subset $A \subseteq \Rd$ (resp. $A \subseteq \Zd$), we denote by $\left| A \right|$ the volume (resp. the cardinality) of $A$.
			\item For a set $D \subset \C$, $\partial D$ denotes its topological boundary, $\cl(D) = D \cup \partial D$ its closure, 
			and $\inte(D)$ its interior. 
			\item $B_r(x)$ denotes the open ball of Euclidean radius $r > 0$ centered at $x \in \R^d$, when $x$ is omitted, the ball is centered at 0. $Q_r$ similarly 
			denotes cube (also referred to as box) of radius $r$. 
			\item We write $\size(Q_r) := r$ for the side length of the cube.
			\item For $L \in \N$ and in dimension $d = 2$, we define the triangle $T_L := \{ x = (x_1 , x_2) \in \Z^2 \, : \, |x| \leq L, x_1 \geq |x_2|\}$.
			\item Given two set $A,B \subseteq \R^d$, we denote by $A + B := \{ z = x + y \, : \, x \in A , \, y \in B \}.$
			\item The set of $d \times d$ symmetric matrices is denoted by $\symm^d$.
		\end{itemize}
		
		\item Percolation.
		\begin{itemize}
			\item The set of edges on $\Z^d$ is denoted by $E\left( \Z^d \right)$ , \ie, the set of unordered pairs $\{x,y\} \in \Z^d$ satisfying $|x-y| =1$.
			\item An edge $e \in E\left( \Z^d \right)$ is open if $\mathbf{a}(e) = 1$ and if $\mathbf{a}(e) = 0$ it is closed.
			\item Each edge is open according to an independent Bernoulli random variable with success probability 
			$\mathfrak{p} \in (\mathfrak{p}_c(d), 1)$ where $\mathfrak{p}_c(d) \in (0,1)$ is the bond percolation threshold for the lattice $\Z^d$. 
			\item The set of open edges of the infinite cluster is denoted by $E \left( \mathscr{C}_\infty\right)$.
			\item We write $x \sim y$ if $x$ and $y$ are connected by an edge in $E \left( \mathscr{C}_\infty\right)$.
			\item The infinite cluster is $\mathscr{C}_{\infty}$, the set of its unordered edges is $E(\mathscr{C}_{\infty})$
			and the cluster is equipped with the graph Laplacian defined as 
			\[
			\Delta_{{\mathscr{C}_{\infty}}} u(x) := \sum_{y \sim x} (u(y) - u(x)). 
			\]
			\item The degree of a vertex in the cluster is $\deg_{\mathscr{C}_{\infty}}$.
			\item For a set $A \subset \mathscr{C}_{\infty}$, by $\partial A$ the set of vertices in $\mathscr{C}_{\infty} \setminus A$ which are joined by an edge to
			a vertex in $A$ and by $\cl(A) = A \cup \partial A$ is its closure.  We write $\partial^-$ for the inner boundary, \ie,  the set of vertices in $A$ which are connected by an edge to a vertex not in $A$. 
			The set $\inte(A)$ are the vertices in $A$ which do not share an edge with a vertex in $A^c$. 
			We write $\partial_e A$ for the set of edges in $\mathscr{C}_{\infty}$ with one end in $A$ and one end not in $A$. 
		\end{itemize}
		\item Paths on $\Z^d$ and $\R^d$.
		\begin{itemize}
			\item A {\it path} on $\Z^d$ is an injective function $\gamma: [1, \ldots, \mathrm{end}] \to  \Z^2$ with $|\gamma(i+1) - \gamma(i)| = 1$ (with $\mathrm{end} \in \N \cup \{ \infty\}$). 
			\item The path is infinite if $\mathrm{end} = \infty$ and otherwise is finite. We denote the length of the path $\gamma $ by $\left| \gamma \right| \in \N \cup \{ \infty \}.$
			\item A {\it bi-infinite path} is an injective function $\gamma: \Z \to \Z^2$  with $|\gamma(i+1) - \gamma(i)| = 1$.
			\item A {\it loop} is a finite path such that $|\gamma(1) - \gamma(\mathrm{end})| = 1$ and we define $\gamma(\mathrm{end}+1) = \gamma(1)$. We overload notation and consider paths and loops both as functions and as subsets of $\Z^d$.
			\item Given two vertices $x , y \in \Zd$ (resp. $x , y \in \mathscr{C}_\infty$), we define the distance $\dist(x , y)$ (resp. $\dist_{\mathscr{C}_\infty}(x , y)$) to be the length of the shortest path connecting $x$ and $y$ in $\Zd$ (resp. in $\mathscr{C}_\infty$). We similarly define the distance between two edges of $E(\Zd)$ (resp. $E(\mathscr{C}_\infty)$) and the distance between a vertex and an edge.
			\item In Section~\ref{sec:toppling-invariants}, we specifically work in two dimensions, and define each of the prior objects on the plane $\R^2$ and the Euclidean sphere $\mathbb{S}^2 \subseteq \R^3$ in the same way, by simply defining a path on~$\R^2$ to be the linear interpolation of a path on~$\Z^2$ (and mapping $\R^2$ to $\mathbb{S}^2$ through the stereographic projection).
			\item A Jordan curve in $\R^2$ or $\mathbb{S}^2$ is a set which is homeomorphic to the circle $\mathbb{S}^1 \subseteq \R^2$.
		\end{itemize}
		\item Stochastic integrability.
		\begin{itemize}
			\item  Given a random variable $X$, we write $X \leq \mathcal{O}_s(\theta)$ to mean \[ \E[ \exp( (\theta^{-1} X \vee 0)^s)] \leq 2. \]
			\item By Markov's inequality, if $X \leq \mathcal{O}_s(\theta)$, then $\P[X \geq \theta t] \leq 2 \exp(-t^s)$.
		\end{itemize}
		\item Function spaces.
		\begin{itemize}
			\item In most instances we use the same notation for the lattice as the continuum. 
			\item If $A$ is a finite subset of $\Z^d$, $|A|$ denotes its cardinality. For a subset of the plane,  
			$V \subset \R^d$ we also use $|V|$ to denote Lebesgue measure. 
			\item The sum over a finite subset of $\Z^d$ will sometimes be written as an integral.
			\item For $p \in [1, \infty)$ we denote the $L^p$ and normalized $L^p$ norms of $w$ by 
			\[
			\|w\|_{L^p(U)} :=  \left( \sum_{x \in U} |w(x)|^p dx \right)^{1/p} \quad \mbox{and} \quad 
			\|w\|_{\underline{L}^p(U)} :=  \left( \frac{1}{|U|} \sum_{x \in U} |w(x)|^p dx \right)^{1/p}
			\]
			and $\|w\|_{L^{\infty}(U)} = \sup_{x \in U} |w(x)|$.
			\item The oscillation of a function $w$ over a finite subset $U$ is denoted by 
			\[
			\osc_{U} w := \sup_{x \in U} w(x) - \inf_{x \in U} w(x)
			\]
			and its average is  
			\[
			(w)_{U} := \frac{1}{|U|} \sum_{x \in U} w(x) dx.  
			\]
		\end{itemize}
		\item Vector fields.
		\begin{itemize}
			\item Let $E_d := \{ (x,y) \in \Z^d : |x-y| = 1 \}$ denote the set of oriented nearest-neighbor edges on $\Z^d$ 
			\item For a subset $U$ of $\Z^d$, we write $E_d(U)$ as the set of oriented edges with both ends in $U$. 
			\item A vector field $v$ on $U$ is a function $v: E_d(U) \to \R$ which is antisymmetric, $u(x,y) = -u(y,x)$ 
			for every $(x,y) \in E_d(U)$.
			\item For a function $u: U \to \R$, $\nabla u$ is the vector field defined by, for any edge $e = (x , y) \in E_d(U)$,
			\begin{equation} \label{id:vectorfiedledges}
				\nabla u(e) := u(x) - u(y),
			\end{equation}
			In Proposition~\ref{prop:homogenizationellipticgreen} and Proposition~\ref{prop:green-mixed-derivative}, and mimicking the notation of the continuum, we will write, for $x \in \mathscr{C}_\infty$, 
			\begin{align*}
				&\nabla u(x)  \\
				&=(\nabla u((x, x + e_1)) \indc_{\{ (x, x+e_1) \in E(\mathscr{C}_\infty) \}}, \ldots, \nabla u((x, x + e_d)) \indc_{\{ (x, x+e_d) \in E(\mathscr{C}_\infty) \}} ) \\
				&\in \R^d.
			\end{align*}
			\item For a function $u$ defined on a subset of the cluster, we similarly define $\nabla u(e)$
			where in this case a vector field has a domain given by the set of oriented edges in $\mathscr{C}_{\infty}$.
			\item For a function $v$ on $\R^d$, $\nabla v$ denotes the usual gradient, $\nabla^2$ the Hessian, 
			and $\partial_{i}$ the partial derivative with respect to the $i$-th coordinate.
			\item For a function of two arguments, \eg, the Green's function $G(x,y)$, we write $\nabla_x$ for the gradient with respect
			to the first variable and $\nabla_y$ for the second variable.
		\end{itemize}    
	\end{itemize}

	\subsection{An elementary property of sets of full measure}
	
	We collect the following elementary result which will be used repeatedly in the proofs below. Let us denote by $\left( \Omega , \mathcal{F} , \P\right)$ the probability space where $\Omega := \left\{ 0, 1 \right\}^{E ( \mathbb{Z}^d )}$ is the set of all the percolation configurations and $\mathcal{F}$ is the product $\sigma$-algebra. Given a measurable set~$E \in \mathcal{F}$ and a finitely supported function~$f: E(\Zd) \to \{0, 1\}$ we denote by~$E_f$ the event which is~$E$ except the values at edges~$e$ with~$f(e)=1$
	are flipped:
	\[
	E_f := \left\{ \mathbf{a} \in \Omega \, : \, {\mathbf{a}}' \in E ~ \mbox{with} ~  {\mathbf{a}}'(e) = (1 - f(e)) \mathbf{a}(e) + f(e)(1 - \mathbf{a}(e)) \, \, \forall \, e \in E(\Zd) \right\}.
	\]
	We denote by~$E_{\infty}$ the intersection of~$E_f$ over all such finitely supported~$f$. 	The following (trivial) implication holds, $E_\infty \subseteq E$. We observe that if $E$ is an event of full probability,  \ie, satisfying $\P(E) = 1$, then $E_\infty$ is also an event of probability $1$.
	
	\begin{lemma} \label{lemma:finite-energy}
		For any measurable event $E \in \mathcal{F}$ of full probability, one has the identity
		$\P (E_\infty) = 1.$
	\end{lemma}
	
	\begin{proof}
		This is immediate from the fact that for every finitely supported~$f: E(\Zd)  \to \{0,1\}$, we have that $\P[E_f] = 1$ and the set of such~$f$ is countable. 
	\end{proof}
	
	\begin{remark}
		The prior result is a consequence of the finite energy of Bernoulli percolation, see, \eg, \cite[Section 7.2]{lyons-peres-book}.
	\end{remark}

	\subsection{Connectivity properties of the cluster} \label{subsec:percolation-properties}
	Recall the notation for the cluster from Section \ref{subsec:notation}. Given a cube $Q_R$, we write $\mathscr{C}_*(Q_R)$ for the largest connected component of open edges contained in the cube, 
	breaking ties in a deterministic fashion. We say that a cube $Q_R$ is {\it well-connected} if the following properties hold:
	\begin{itemize}
		\item $R$ is large enough so that $R^{1/(10 d)^2} \leq  R/100$.
		\item $Q_R$ is {\it crossing}: each of the $d$ pairs of opposite $(d-1)$ dimensional faces of $Q_R$ is joined by a path in $\mathscr{C}_*(Q_R)$.
		\item Every cube $Q_M$ of side length $M$ for $M \in [R^{1/(10 d)^2}/2, R/10]$ which intersects $Q_R$ is crossing. Moreover, for every such cube, every open path of length at least $\frac{M}{100}$ in $Q_M$ is connected, within $Q_M$, to $\mathscr{C}_*(Q_M)$.
	\end{itemize}
	We remark that the exponent $1/(10 d)^2$ could be made smaller, at the cost of changing the constants and exponents below.
	Moreover, for $R$ sufficiently large, $\mathscr{C}_*(Q_R)$ is comparable to $\mathscr{C}_{\infty} \cap Q_R$.
	\begin{prop} \label{prop:well-connected}
		There exist constants $c(d, \mathfrak{p}) > 0,C(d, \mathfrak{p}) < \infty$ and an exponent $s(d) > 0$ such that the probability that $Q_N$ is well-connected
		is at least $1 - C \exp(-c N^s)$. 
	\end{prop}
	\begin{proof}
		This follows from \cite[Theorem 3.2]{pisztora-percolation} and \cite[Theorem 5]{penrose-pisztora-1996} as recalled in \cite[Equation (2.24)]{antal-pisztora-chemical}.
	\end{proof}

	\subsection{Homogenization of harmonic polynomials}
	In this section we record the preliminary theorems we use from \cite{dario-corrector, dario-gu-green,armstrong-dario-2018}, introduce the corrector, corrected planes, and spaces of harmonic functions. 
	Start by denoting by $\mathcal{A}(\mathscr{C}_{\infty})$ the space of harmonic functions on the cluster. For each $k \in \N$, we let $\mathcal{A}_k(\mathscr{C}_\infty)$ denote the subspace 
	of $\mathcal{A}(\mathscr{C}_{\infty})$ consisting of functions growing more slowly at infinity than a polynomial of degree $(k+1)$: 
	\[
	\mathcal{A}_k(\mathscr{C}_{\infty}) := \left\{ u \in \mathcal{A}(\mathscr{C}_{\infty}) : \limsup_{R \to \infty} R^{-(k+1)} \|u\|_{\underline{L}^2(\mathscr{C}_{\infty} \cap B_R)} = 0 \right \},
	\]
	and the space of harmonic polynomials of degree at most $k$ by 
	\[
	\overline{\mathcal{A}}_k := \left\{ u : \R^d \to \R : \Delta u = 0 ~\mbox{in} ~\Rd ~\mbox{and}~ \limsup_{R \to \infty} R^{-(k+1)} \|u\|_{\underline{L}^2(B_R)} = 0 \right \}.
	\]
	Recall the classical fact that harmonic functions of polynomial growth are polynomials, \ie, the space ~$\overline{\mathcal{A}}_k$ consists of harmonic polynomials of degree at most~$k$. Below we record the analogue of this fact, proved in \cite{armstrong-dario-2018}, for the percolation cluster.

	\subsubsection{Large-scale regularity}
	
	We recall below the large-scale regularity for harmonic functions on the percolation cluster. The result was originally established in~\cite[Theorem 1.2]{armstrong-dario-2018}.

	\begin{theorem}[Theorem 1.2 in \cite{armstrong-dario-2018}] \label{theorem:large scale-regularity}
		There exist exponents $s(d, \mathfrak{p}) > 0$ and $\delta(d, \mathfrak{p}) > 0$, a constant $C := C(d , \mathfrak{p}) < \infty$ such that for any $x \in \Zd$, there exists a nonnegative random variable  $\mathcal{M}_{\mathrm{reg}}(x)$  satisfying the stochastic integrability estimate, 
		\begin{equation}
			\mathcal{M}_{\mathrm{reg}}(x) \leq \mathcal{O}_s(C)
		\end{equation}
		such that the following hold:
		\begin{enumerate}[label=(\roman*)]
			\item For each $k \in \N$, there exists a constant $C_k(k, d, \mathfrak{p}) < \infty$ such that, for every $u \in \mathcal{A}_k(\mathscr{C}_{\infty})$, 
			there exists $p \in \overline{\mathcal{A}}_k$ such that, for every $r \geq   \mathcal{M}_{\mathrm{reg}} (x)$, 
			\begin{equation} \label{eq:harmonic-approximation}
				\| u - p\|_{\underline{L}^2(\mathscr{C}_{\infty} \cap B_r(x))} \leq C_k r^{-\delta} \|p\|_{\underline{L}^2(B_r(x))}.
			\end{equation}
			\item For every $k \in \N$ and $p \in \overline{\mathcal{A}}_k$, there exists $u \in \mathcal{A}_k$ such that, for every $r \geq   \mathcal{M}_{\mathrm{reg}}(x)$, 
			the inequality \eqref{eq:harmonic-approximation} holds. 
		\end{enumerate}
	\end{theorem}

	\subsubsection{First order corrector}
	In the case $k = 1$,  functions which lie in $\mathcal{A}_1(\mathscr{C}_{\infty})$ satisfy stronger properties than that stated in Theorem \ref{theorem:large scale-regularity}. 
	
	\begin{theorem}[(1.22) in \cite{armstrong-dario-2018} and Proposition 2.12 in \cite{dario-gu-green}] \label{theorem:first-order-corrector}
		Every function $u \in \mathcal{A}_1(\mathscr{C}_{\infty})$ can be uniquely written as 
		\[
		u(x) = a + p \cdot x + \chi_p(x)
		\]
		where $a \in \R$, $p \in \R^d$, and $\chi_p$ is called the corrector and is defined up to additive constant. 
		The corrector satisfies the following properties:
		\begin{itemize}
			\item Linearity with respect to the variable $p$. The map $p \to \chi_p$ is linear.
			\item Quantitative sublinearity and Lipschitz bound. For any exponent $\alpha > 0$, there exist an exponent $s(d, \mathfrak{p}, \alpha) > 0$ and a constant $C(d, \mathfrak{p}, \alpha) < \infty$
			such that for any vertex $x \in \Z^d$, there exists a non-negative random variable $ \mathcal{M}_{\mathrm{corr}, \alpha}$
			satisfying the stochastic integrability estimate, 
			\begin{equation}
				\mathcal{M}_{\mathrm{corr}, \alpha}(x) \leq \mathcal{O}_s(C)
			\end{equation}
			such that for every radius $r \geq \mathcal{M}_{\mathrm{corr}, \alpha}(x)$ and every $p \in \R^d$,
			\begin{equation} \label{eq:corrector-sublinearity-and-lipschitz}
				\osc_{\mathscr{C}_{\infty} \cap B_r(x)} \chi_p + \| \nabla \chi_p \|_{L^{\infty}(\mathscr{C}_{\infty} \cap B_r(x))} \leq C |p| r^{\alpha}. 
			\end{equation}
		\end{itemize}
	\end{theorem}

	%% command for corrected plane
	\newcommand{\ellp}[1]{\hyperref[eq:corrected-plane]{\ell_{#1}}}
	
	In the rest of this article, we write 
	\begin{equation} \label{eq:corrected-plane}
		\ellp{p} = \mbox{unique, modulo additive constant harmonic function which grows like $p \cdot x$ at infinity} \\
	\end{equation}
	and refer to this function as the {\it corrected plane}. 
	
	\subsection{Homogenization of the Green's function}
	Given the results of \cite{armstrong-dario-2018} and \cite{dario-gu-green}, the following results are a straightforward adaptation of the proofs of \cite[Theorem 8.14 and Theorem 8.20]{AKM19}. 
	For the statement, denote by $G(\cdot, \cdot)$ the elliptic Green's function on the percolation cluster and $\bar{G}(\cdot)$ its continuum homogenized version. The function $\bar{G}(\cdot)$ is a multiple of the standard elliptic Green's function on $\R^d$, and the multiplicative coefficient is explicit in terms of the diffusivity of the random walk on the percolation cluster and of the density of the infinite cluster (see e.g., \cite[(1.9)]{BH09} or \cite[(1.9)]{dario-gu-green})
	\[
	\bar{G}(x) := 
	\begin{cases}
		- \frac{1}{2\pi \bar{\mathbf{a}}(\mathfrak p)} \ln |x| \quad \mbox{if~$d = 2$}  \\
		\frac{\Gamma(d/2 -1)}{2 \pi^{d/2} \bar{\mathbf{a}}(\mathfrak p)} \frac{1}{|x|^{d-2}} \quad \mbox{if~$d \geq 3$}, 
	\end{cases}
	\]
	where~$\bar{\mathbf{a}}( \mathfrak{p})$ is a deterministic function of the percolation probability~$ \mathfrak{p}$ (the homogenized coefficient associated with the infinite cluster).
	
	In dimension $d \geq 3$, given a vertex $y \in \mathscr{C}_{\infty}$,
	we define $G(\cdot, y)$ as the solution to
	\begin{equation} \label{eq:green-function-3D}
		\begin{cases}
			&\Delta_{\mathscr{C}_{\infty}} G(\cdot, y) = -\delta_y \quad \mbox{in $\mathscr{C}_{\infty}$} \\
			&\lim_{|x| \to \infty} G(x, y) =  0
		\end{cases}
	\end{equation}
	and in dimension $d = 2$, 
	\begin{equation} \label{eq:green-function-2D}
		\begin{cases}
			&\Delta_{\mathscr{C}_{\infty}} G(\cdot, y) = -\delta_y \quad \mbox{in $\mathscr{C}_{\infty}$} \\
			&\lim_{|x| \to \infty} \frac{1}{|x|} G(x, y) =  0 \\
			& G(y,y) = 0.
		\end{cases}
	\end{equation}
	
	The first result we present is a homogenization theorem for the gradient of the Green's function. This result can be fairly easily deduced from~\cite[Theorem 2]{dario-gu-green} and the large-scale regularity on the infinite cluster stated in Theorem~\ref{theorem:large scale-regularity}. The proof is written below for completeness. 
	
	\begin{prop}[Homogenization of the gradient of the Green's function] \label{prop:homogenizationellipticgreen}
		For any $\delta > 0$, there exist constants $s(d , \mathfrak{p}, \delta)> 0$, $C (d , \mathfrak{p}, \delta) < \infty$, and, for any $y \in \mathbb{Z}^d$, there exists a minimal scale $\mathcal{M}_{\nabla-\mathrm{Homog}, \delta}(y)$
		satisfying 
		\begin{equation*}
			\mathcal{M}_{\nabla-\mathrm{Homog}, \delta}(y) \leq \mathcal{O}_s(C)
		\end{equation*}
		such that, for every $x, y \in \mathscr{C}_\infty$ with $|x - y| \geq \mathcal{M}_{\nabla-\mathrm{Homog}}(y)$,
		\begin{equation} \label{eq:homogenizationellipticgreen}
			\left| \nabla_x G(x , y) - \nabla \bar G(x - y) - \sum_{i = 1}^d \nabla \chi_{e_i}(x) \partial_i \bar G(x - y) \right| \leq C |x-y|^{-\frac{1}{2}+\delta} |x - y|^{1-d}.
		\end{equation}	
	\end{prop}
	
	\begin{remark}
		The exponent $1/2$ in the right-hand side of~\eqref{eq:homogenizationellipticgreen} is not optimal but is sufficient for our purposes and can be obtained with a relatively short argument.
	\end{remark}

	The following proposition provides decay estimates on the mixed derivative of the Green's function. Once again, the proof can be fairly easily obtained by combining the results of \cite{dario-gu-green}
	and \cite{armstrong-dario-2018}.

	\begin{prop}[Decay estimate for the mixed derivative of the Green's function] \label{prop:green-mixed-derivative}
		For any $\delta > 0$, there exist constants $s(d , \mathfrak{p}, \delta)> 0$, $C (d , \mathfrak{p}, \delta) < \infty$, and, for any $y \in \mathbb{Z}^d$, there exists a minimal scale $\mathcal{M}_{\nabla\nabla-\mathrm{Decay}}(y)$
		satisfying 
		\begin{equation} \label{eq:stoch.estMdecay}
			\mathcal{M}_{\nabla\nabla-\mathrm{Decay}, \delta}(y) \leq \mathcal{O}_s(C)
		\end{equation}
		such that, for every $x, y \in \mathscr{C}_\infty$ with $|x - y| \geq \mathcal{M}_{\nabla\nabla-\mathrm{Decay}, \delta}(y)$, one has the estimate
		\begin{equation*}
			|\nabla_x \nabla_y G(x, y)| \leq \frac{C}{|x - y|^{d - \delta}}.
		\end{equation*}
	\end{prop}
	
	\begin{proof}[Proof of Proposition~\ref{prop:homogenizationellipticgreen}]
		We present the proof in the case $d \geq 3$, the only change for $d = 2$ is to replace the bound in \eqref{eq:green-homogenziation} below
		with the $d=2$ bound of \cite{dario-gu-green}. Fix an $x, y \in \Z^d$ and restrict to the event that $y \in \mathscr{C}_{\infty}$. 
		
		{\it Step 1: Preliminary results.} \\ 
		We start by recalling two results. The first one is a main result of \cite{dario-gu-green} and establishes a quantitative homogenization theorem for the Green's function on the percolation cluster. 
		Fix a small exponent $\delta > 0$. 
		By \cite[Theorem 2]{dario-gu-green}, there exists a nonnegative random variable $\mathcal{M}_{\mathrm{Homog}, \delta}(y) \leq \mathcal{O}_s(C)$ such that 
		\begin{equation} \label{eq:green-homogenziation}
			|G(z,y) - \bar{G}(z-y)|  \leq \frac{1}{|z-y|^{1-\delta}}  \frac{C}{|z-y|^{d-2}}, \quad \forall z \in \mathscr{C}_\infty, \, |z-y| \geq \mathcal{M}_{\mathrm{Homog}, \delta}(y).
		\end{equation}
		The second one is a density estimate for the infinite cluster due to~\cite[Theorem 1]{penrose-pisztora-1996}: 
		there exists a constant $\mathfrak{c}(\mathfrak{p}) > 0$ and a minimal scale $\mathcal{M}_{\mathrm{dense}}(x) \leq \mathcal{O}_s(C)$ 
		such that for all  $r \geq \mathcal{M}_{\mathrm{dense}}(x)$ any ball of radius $\frac{r}{100}$
		containing the vertex $x$ contains at least $\mathfrak{c} r^d$ vertices in the cluster. 
		
		We next note that Theorem~\ref{theorem:large scale-regularity} and Theorem~\ref{theorem:first-order-corrector} can be combined so as to establish the following statement. For every $R \geq 2 \mathcal{M}_{\mathrm{reg}}(x)$ and every function $u \in \mathcal{A}(\mathscr{C}_{\infty} \cap B_{2 R}(x))$, 
		there exists a slope $\xi \in \R^d$ such that 
		\begin{equation} \label{eq:large scale-regularity}
			\inf_{a \in \R} \| u - \ell_\xi  - a \|_{\underline{L}^2(\mathscr{C}_{\infty} \cap B_r(x))} \leq C \left( \frac{r}{R} \right)^2 \|u\|_{\underline{L}^2( \mathscr{C}_{\infty} \cap B_R(x))}, \quad \forall 
			r \in \left[\mathcal{M}_{\mathrm{reg}}(x), \frac{1}{2} R\right].
		\end{equation}
		
		Finally, we require the following (standard) deterministic estimates on the homogenized Green's function,
		which hold for all $r \geq 1$:
		\begin{equation} \label{eq:decay-estimate}
			\| \bar{G}\|_{L^\infty(B_{2r} \setminus B_r)} \leq C r^{2-d},
		\end{equation}
		\begin{equation} \label{eq:green-gradient-estimate}
			\| \nabla \bar{G}\|_{L^\infty(B_{2r} \setminus B_r)} \leq C r^{1-d},
		\end{equation}
		\begin{equation} \label{eq:green-hessian-estimate}
			\| \nabla^2 \bar{G}\|_{L^\infty(B_{2r} \setminus B_r)} \leq C r^{-d}.
		\end{equation}

		\medskip
		{\it Step 2: Reduction to a weaker statement.} \\
		Let 
		\begin{equation} \label{eq:choice-of-minimal-scale}
			\mathcal{X}(x,y) = \mathcal{M}_{\mathrm{Homog}, \delta}(y) \vee \mathcal{M}_{\mathrm{reg}}(x) \vee \mathcal{M}_{\mathrm{corr}, \delta}(x) \vee \mathcal{M}_{\mathrm{dense}}(x) \vee R_0
		\end{equation}
		where $R_0$ is a deterministic constant depending only on $d$, $\mathfrak{p}$, and $\delta$ to be determined in Step~3 below
		and write for notational simplicity, 
		\[
		g(\cdot) := G(\cdot, y) \quad \bar{g}(\cdot) := G(\cdot-y) \quad \mbox{and} \quad R = \frac{|x-y|}{2}. 
		\]
		
		We explain in this step that it suffices to prove the following weaker statement:  if $R^{\frac{1}{2}} \geq \mathcal{X}(x,y)$, then
		\begin{equation} \label{eq:two-scale-gradient-bound-l2}
			\| \nabla g - (\nabla \bar{g} + \nabla \chi_{\nabla \bar{g}(\cdot) }) \|_{\underline{L}^2(\mathscr{C}_{\infty} \cap B_r(x))}
			\leq C R^{ \frac{1}{2}- d + \frac{3\delta}{2}}, \quad \forall r \in [\mathcal{X}(x,y), R^{\frac{1}{2}}].
		\end{equation}
		where, for any $z \in \mathscr{C}_\infty$,
		\[
		\nabla \chi_{\nabla \bar{g}(\cdot) }(z) :=  \sum_{i = 1}^d \nabla \chi_{e_i}(z) \partial_i \bar g(z) .
		\]
		First, assume that $x \in \mathscr{C}_\infty$ and observe that \eqref{eq:two-scale-gradient-bound-l2} implies the following (trivial) bound: 
		\[
		\| \nabla g - (\nabla \bar{g} + \nabla \chi_{\nabla \bar{g}(\cdot) }) \|_{{L}^\infty(\mathscr{C}_{\infty} \cap B_r(x))}\leq C r^{\frac{d}{2}} R^{ \frac{1}{2}- d + \frac{3\delta}{2}}, \quad \forall r \in [\mathcal{X}(x,y), R^{\frac{1}{2}}]
		\]
		and therefore we may take $r = \mathcal{X}(x,y)$ in the above to get
		\begin{equation} \label{eq:two-scale-trivial-bound}
			| \nabla g(x) - (\nabla \bar{g}(x) + \nabla \chi_{\nabla \bar{g}(\cdot) }(x))| \leq C \mathcal{X}^{\frac{d}{2}}(x,y) R^{- d + \frac{\delta}{2}}.
		\end{equation}
		We then define the random variable
		\begin{equation*}
			\mathcal{M}(x , y) := 2 \mathcal{X}^{\frac{d}{\delta}}(x , y).
		\end{equation*}
		By shrinking the stochastic integrability exponent $s$ and increasing the value of the constant $C$, we have $\mathcal{M}(x , y) \leq \mathcal{O}_s(C)$. Additionally note that $d/\delta \geq 2$. Thus, if $|x - y| \geq \mathcal{M}(x , y)$, then $R^{1/2}\geq \mathcal{X}(x , y)$ and $\mathcal{X}^{d/2}(x , y) R^{-\delta/2} \leq 1.$ Combining the previous display with~\eqref{eq:two-scale-trivial-bound}, we obtain that: if $|x - y| \geq \mathcal{M}(x , y)$, then
		\begin{equation*}
			| \nabla g(x) - (\nabla \bar{g}(x) + \nabla \chi_{\nabla \bar{g}(\cdot) }(x))| \leq C R^{\frac{1}{2}- d + \delta}.
		\end{equation*}
		We then complete the proof by defining 
		\begin{equation*}
			\mathcal{M}_{\nabla\nabla-\mathrm{Decay}, \delta}(y) := \sup \left\{ |x - y| \, : \, x \in \Zd, \, \mathcal{M}(x , y) \geq |x - y| \right\}.
		\end{equation*}
		The integrability estimate $\mathcal{M}(x , y) \leq \mathcal{O}_s(C)$ and a union bound implies that $\mathcal{M}_{\nabla\nabla-\mathrm{Decay}, \delta}(y)$ satisfies the same stochastic integrability estimate.
		
		\medskip 
		{\it Step 3: Proof of weaker statement.} \\
		It remains to prove \eqref{eq:two-scale-gradient-bound-l2}.  Note that $g$ is harmonic in $\mathscr{C}_{\infty} \setminus \{y\}$
		and so
		\[
		\mbox{$g$ is harmonic in $B_R(x) \cap \mathscr{C}_{\infty}$}.
		\]
		Assume that $R \geq \mathcal{X}(x,y)$ (from \eqref{eq:choice-of-minimal-scale}) and fix an $r \in [\mathcal{X}(x,y), R/2]$. By the large scale regularity \eqref{eq:large scale-regularity} and then~\eqref{eq:green-homogenziation} and~\eqref{eq:decay-estimate}, we have that there exists a slope $\xi \in \R^d$ such that
		\begin{equation} \label{eq:green-large scale-regularity}
			\inf_{a \in \R} \| g - \ell_\xi - a \|_{\underline{L}^2(\mathscr{C}_{\infty} \cap B_r(x))} \leq C \left( \frac{r}{R} \right)^{2} \| g\|_{\underline{L}^2(\mathscr{C}_{\infty} \cap B_R(x))}
			\leq C r^2 R^{-d}.
		\end{equation}
		So, by Caccioppoli's inequality (see, \eg, \cite[Lemma 3.5]{armstrong-dario-2018} or \cite{delmotte-elliptic-estimates}), 
		\begin{equation} \label{eq:cacciopoli-apply}
			\| \nabla g - \nabla \ell_\xi  \|_{\underline{L}^2(\mathscr{C}_{\infty} \cap  B_r(x))} \leq C r R^{-d}.
		\end{equation}
		Also, by~\eqref{eq:green-homogenziation},
		\[
		|g(z) - \bar{g}(z)|  \leq \frac{1}{|z-y|^{d-2}} \frac{C}{|z-y|^{1-\delta}}, \quad \forall z \in B_R(x)
		\]
		and so 
		\begin{equation} \label{eq:homogenization-green-close}
			\| g - \bar{g} \|_{\underline{L}^2(\mathscr{C}_{\infty} \cap B_r(x))} \leq C R^{1 - d + \delta}.
		\end{equation}
		Since $\bar{g}$ is harmonic and hence smooth in $B_R(x)$, we have by Taylor's theorem and then \eqref{eq:green-hessian-estimate} that
		\begin{equation} \label{eq:green-taylor-expansion}
			\sup_{z \in B_r(x)}  
			| \bar{g}(z) - ( \bar{g}(x) + (z-x) \cdot \nabla \bar{g}(x)) |
			\leq C r^2 \| \nabla^2 \bar{g}(x) \|_{L^{\infty}(B_r(x))}
			\leq C r^2 R^{-d}.
		\end{equation}
		For the following computation, we introduce two notations for the affine functions 
		\[
		l_g(z) =  ( \bar{g}(x) + (z-x) \cdot \nabla \bar{g}(x)) \hspace{5mm}\mbox{and} \hspace{5mm} l_\xi(z) := \xi \cdot z
		\]
		and observe that $\nabla \ell_\xi = \xi + \nabla \chi_\xi = \nabla l_\xi + \nabla \chi_\xi.$
		We then write
		\begin{align*}
			r | \nabla \bar{g}(x) - \xi | &= r | \nabla (l_g  - l_\xi) |  \\
			&\leq C \inf_{a \in \R} \| l_g - (l_\xi   +  a   ) \|_{\underline{L}^2(\mathscr{C}_{\infty} \cap B_r(x))}\\
			& \quad \mbox{({for $r \geq \mathcal{M}_{\mathrm{dense}}(x)$ the $L^2$ norm }} \\
			&\quad \mbox{of an affine function in $B_{r}(x)$ majorizes its gradient at $x$}) \\ 
			&\leq C \inf_{a \in \R} \| \bar{g} - (l_\xi +  a ) \|_{\underline{L}^2(\mathscr{C}_{\infty} \cap B_r(x))} + C r^2 R^{-d} \quad \mbox{(by \eqref{eq:green-taylor-expansion})} \\
			&\leq C \inf_{a \in \R}\| g - (\ell_\xi +  a ) \|_{\underline{L}^2(\mathscr{C}_{\infty} \cap B_r(x))} + C \osc_{\mathscr{C}_{\infty} \cap B_r(x)} \chi_{\xi} 
			\\& \qquad + Cr^2 R^{-d} + CR^{1 - d + \delta} \\
			&\quad \mbox{(by the triangle inequality and \eqref{eq:homogenization-green-close})} \\ 
			&\leq C r^2 R^{-d} + C r^{\delta} |\xi| + C R^{1 - d + \delta} \\
			&\quad \mbox{(by \eqref{eq:green-large scale-regularity} and sublinearity of the corrector \eqref{eq:corrector-sublinearity-and-lipschitz})}.
		\end{align*}
		Now, take the constant $R_0$ in \eqref{eq:choice-of-minimal-scale} sufficiently large so that $R_0^{1-\delta/2}$ is larger than $2C$ to see from the triangle inequality that 
		\begin{equation*}
			C r^{\frac{\delta}{2}} |\xi| \leq C r^{\frac{\delta}{2}} |\nabla \bar g(x) - \xi| + C r^{\frac{\delta}{2}} |\nabla \bar g(x)| \leq \frac{r}{2} |\nabla \bar g(x) - \xi| + C r^{\frac{\delta}{2}} |\nabla \bar g(x)|
		\end{equation*}
		and so 
		\begin{align*}
			&\frac{1}{2} | \nabla \bar{g}(x) - \xi |  \\
			&\leq C r^{\delta-1} |\nabla \bar g(x)| + C r R^{-d}  + C r^{-1} R^{1 - d + \delta}   \quad \mbox{(by the above two displays)} \\
			&\leq C  r^{\delta-1} R^{1-d} + C r R^{-d} + C r^{-1} R^{1 - d + \delta}\quad \mbox{(by \eqref{eq:green-gradient-estimate})} \\
			&= C R^{\frac{1}{2} - d + \frac{\delta}{2} } + C R^{\frac{1}{2} - d} + C R^{\frac{1}{2} - d + \delta} \quad\mbox{(set $r = R^{1/2}$)} \\
			&\leq R^{\frac{1}{2}-d + \delta}.
		\end{align*}
		From this and \eqref{eq:cacciopoli-apply}, we obtain that 
		\begin{align*}
			&\| \nabla g - (\nabla \bar{g}(x) + \nabla \chi_{\nabla \bar{g}(x)}) \|_{\underline{L}^2(B_r(x))} \\
			&\leq  \| \nabla g - (\xi + \nabla \chi_{\xi}) \|_{\underline{L}^2(B_r(x))} 
			+ \left( | \nabla \bar{g}(x) - \xi| ( 1 + \sup_{e \in B_1} \| \nabla \chi_e \|_{L^\infty(B_r(x))} )\right) \\
			&\leq C r R^{-d} + r^\delta R^{\frac{1}{2}-d + \delta} \\
			&\leq R^{\frac{1}{2} - d + \frac{3\delta}{2}}  \quad \mbox{(choose $r \leq R^{\frac{1}{2}}$)}.
		\end{align*}
		However, we also have that, 
		\begin{align*}
			&\| (\nabla \bar{g} + \nabla \chi_{\nabla \bar{g}(\cdot)}) - (\nabla \bar{g}(x) + \nabla \chi_{\nabla \bar{g}(x)}) \|_{\underline{L}^2(B_r(x))} \\
			&\leq  \left( \osc_{B_r(x)} |\nabla \bar{g}| \right)  \left( 1 + \sup_{e \in B_1} \| \nabla \chi_e \|_{L^\infty(B_r(x))} \right) \\
			&\leq  r R^{-d}(1 + r^\delta)\quad \mbox{(by \eqref{eq:green-hessian-estimate} and gradient of corrector bound~\eqref{eq:corrector-sublinearity-and-lipschitz})} \\
			&\leq R^{\frac{1}{2} - d + \frac{\delta}{2}}  \quad \mbox{(choose $r \leq  R^{\frac{1}{2}}$)}. 
		\end{align*}
		This implies  \eqref{eq:two-scale-gradient-bound-l2}. 
	\end{proof}

	\begin{proof}[Proof of Proposition~\ref{prop:green-mixed-derivative}]
		Before proving Proposition~\ref{prop:green-mixed-derivative}, we collect two preliminary results. The first one provides a Lipschitz regularity estimate for harmonic functions on the infinite percolation cluster and can be obtained by applying Theorem~\ref{theorem:large scale-regularity} with the value $k=0$ and combining it with the Caccioppoli inequality (the exact statement below can be found in~\cite[Proposition 2.14]{dario-corrector}). The statement reads as follows. For any $r , R \geq \mathcal{M}_{\mathrm{reg}}(x)$ with $r \leq R$, and any harmonic function $u : \mathscr{C}_\infty \cap B_R(x) \to \R$,
		\begin{equation*}
			\left\| \nabla u \right\|_{\underline{L}^2(\mathscr{C}_\infty \cap B_r(x))} \leq \frac{C}{R} \left\|  u - (u)_{\mathscr{C}_\infty \cap B_R(x)} \right\|_{\underline{L}^2(\mathscr{C}_\infty \cap B_R(x))} \leq \frac{C}{R} \left\|  u \right\|_{\underline{L}^2(\mathscr{C}_\infty \cap B_R(x))}.
		\end{equation*}
		In particular, choosing $r = \mathcal{M}_{\mathrm{reg}}(x)$ yields
		\begin{equation} \label{eq:gradient-bound-regularity}
			| \nabla u(x) | \leq C \mathcal{M}_{\mathrm{reg}}(x)^{\frac{d}{2}} \left\| \nabla u \right\|_{\underline{L}^2(\mathscr{C}_\infty \cap B_r(x))}  \leq  \frac{C \mathcal{M}_{\mathrm{reg}}(x)^{\frac{d}{2}}}{R} \left\| u \right\|_{\underline{L}^2(B_R(x) \cap \mathscr{C}_{\infty})}.
		\end{equation}
		The second result is a direct consequence of~\cite[Theorem 3]{dario-gu-green} (or specifically~\cite[Remark 1.3]{dario-gu-green}) and an integration of the heat kernel over time, and reads as follows: for any $y \in \mathbb{Zd}$, there exists a random variable $\mathcal{M}_{\nabla\mathrm{-Decay}}(y)$ such that, if $y \in \mathscr{C}_\infty$ then for any $x \in \mathscr{C}_\infty$ with $|x - y| \geq \mathcal{M}_{\nabla\mathrm{-Decay}}(y)$,
		\begin{equation} \label{eq:nablayG}
			\left| \nabla_y G(x , y) \right| \leq \frac{C \mathcal{M}_{\nabla\mathrm{-Decay}}(y)^{\frac{d}{2}}}{|x - y|^{d-1}}.
		\end{equation}
		We next combine~\eqref{eq:gradient-bound-regularity} and~\eqref{eq:nablayG} and obtain that, for any $x , y \in \mathscr{C}_\infty$ with $|x - y| \geq 2 \left( \mathcal{M}_{\nabla\mathrm{-Decay}}(y) \vee \mathcal{M}_{\mathrm{reg}}(x) \right)$,
		\begin{align*}
			\left| \nabla_x \nabla_y G(x , y) \right| & \leq \frac{C \mathcal{M}_{\mathrm{reg}}(x)^{\frac{d}{2}} 
			}{|x-y|} \left\| \nabla_y G(\cdot , y) \right\|_{\underline{L}^2 \left(B_{|x - y|/2} \cap \mathscr{C}_{\infty} \right)} \\
			& \leq \frac{C \mathcal{M}_{\mathrm{reg}}(x)^{\frac{d}{2}} \mathcal{M}_{\nabla\mathrm{-Decay}}(y)^{\frac{d}{2}}}{|x-y|^d}.
		\end{align*}
		The previous inequality implies Proposition~\ref{prop:green-mixed-derivative} with the definition
		\begin{equation*}
			\mathcal{M}_{\nabla\nabla\mathrm{-Decay}}(y) := \sup \left\{ |x - y| \, : x \in \mathbb{Z}^d, \, \mathcal{M}_{\mathrm{reg}}(x)^{\frac{d}{2}} \mathcal{M}_{\nabla\mathrm{-Decay}}(y)^{\frac{d}{2}} \geq |y - x|^\delta \right\}.
		\end{equation*}
		The stochastic integrability estimates $\mathcal{M}_{\mathrm{reg}}(y) \leq \mathcal{O}_s(C)$ and $\mathcal{M}_{\nabla\mathrm{-Decay}}(x)\leq \mathcal{O}_s(C)$, and a union bound imply that the minimal scale $\mathcal{M}_{\nabla\nabla\mathrm{-Decay}}(y) $ satisfies the stochastic integrability estimate stated in~\eqref{eq:stoch.estMdecay}.
	\end{proof}
	
	\subsection{Topology in two dimensions and planarity}
	A key input into our proof of Theorem~\ref{theorem:fast-decay} is planarity.  In this subsection we collect several standard results on the topology of the plane, lattice, and cluster. First, we first introduce the stereographic projection.
	
	\begin{definition}[Stereographic projection]
		Let $\S^2 \subseteq \R^3$ be the Euclidean ball and let $N := (0,0,1)$ and $S := (0,0,-1)$ be the north and south poles of the sphere respectively.
		We denote by $SP : \S^2 \setminus \left\{ N \right\}  \to \R^2$ the stereographic projection. 
	\end{definition}
	
	The stereographic projection is a homeomorphism $\R^2 \to \S^2 \setminus \left\{ N \right\}$. Any path $\gamma$ is mapped to a continuous path on the sphere through the stereographic projection (with the important remark that any infinite path $\gamma$ is mapped to a continuous map $\tilde \gamma :[0 , 1] \to \S^2$ such that $\tilde \gamma(1) = N$). Any bi-infinite path $\gamma$ is mapped to a continuous loop $\tilde \gamma :[0 , 1] \to \S^2$ with $\tilde \gamma(0) = \tilde \gamma(1) = N$ (where we used the definition introduced in Section~\ref{subsec:notation}, which implies in particular that any bi-infinite path has to come from and go to infinity).
	
	We then record the notions of simple connectivity and simple boundary for a simply connected set from~\cite[Definition 14.16]{rudinanalysis}.
	
	\begin{definition}[Simply connected domain]
		A nonempty open set $S \subseteq \R^2$ is called simply connected if and only if $S$ and its complement in the sphere $\mathbb{S}^2$ (using the stereographic projection) are connected.
	\end{definition}
	
	\begin{definition}[Simple boundary of simply connected sets]
		A boundary point $x$ of a simply connected plane region $\mathcal{S}$
		will be called a simple boundary point of $\mathcal{S}$ if it has the following property: To
		every sequence $(x_n)_{n \in \N}$ in $\mathcal{S}$ such that $x_n \to x$ as $n \to \infty$ there corresponds a
		continuous map $\gamma : [0, 1) \to \mathcal{S}$ with $\lim_{t \to 1} \gamma (t) = x$ and a sequence $(t_n)_{n \in \N}$ with $0 < t_1 < t_2 <
		\ldots <  t_{n - 1}$, such that $\gamma(t_n) = x_n$ for all~$n \in \N$.
		In other words, there is a curve in $\mathcal{S}$ which passes through the points $x_n$
		and which ends at $x$. We say that the boundary $\mathcal{S}$ is simple if all its points are simple.
	\end{definition}
	
	\begin{remark}
		The boundary of any simply connected Lipschitz and bounded domain $\Omega \subseteq \R^2$ is simple. All the sets we consider in this section can be written as a finite union of boxes of the form $z + (- 1/2 , 1/2)^2$ with $z \in \Z^2$. They are thus all Lipschitz.
	\end{remark}
	
	The following lemma shows that if all the points of the boundary of a simply connected open set are simple, then it is a Jordan curve.

	\begin{lemma}[Simply connected sets with simple boundaries, Remark 14.20 of \cite{rudinanalysis}]
		Let $\Omega \subseteq \Z^2$ be a simply connected bounded domain whose boundary $\partial \Omega$ is simple, then $\partial \Omega$ is homeomorphic to a circle, \ie, it is a Jordan curve. 
	\end{lemma} 
	
	We next state the Jordan curve theorem on the sphere $\mathbb{S}^2$.
	
	\begin{lemma}[Jordan curve theorem on $\S^2$] \label{lemma:jordan-curve-theorem-on-s2}
		Let $\gamma$ be a continuous loop on $\S^2$, then the set $\S^2 \setminus \gamma$ is the union of two finite open simply connected components $C_1$ and $C_2$ satisfying the following properties
		\begin{itemize}
			\item The sets $C_1$ and $C_2$ have the same boundary which is equal to the loop $\gamma$.
			\item The sets $\cl(C_1) = C_1 \cup \gamma$ and $\cl(C_2) = C_2 \cup \gamma$ are each homeomorphic to the closed unit disc $\cl(\mathbb{D}_1) := \left\{ z \in \R^2 \, : \, |z| \leq 1 \right\}$.
			\item There exists a homeomorphism $f : \cl(C_1) \to \cl(\mathbb{D}_1)$ whose restriction to $C_1$ is a biholomorphism from $C_1$ to $\mathbb{D}_1 := \left\{ z \in \R^2 \, : \, |z| < 1 \right\}$.
		\end{itemize}
	\end{lemma}

	\begin{proof}
		The result is obtained by combining Jordan curve theorem with the Riemann mapping theorem (with extension to the boundary in the case of simple boundaries) from~\cite[Chapter 14]{rudinanalysis}.
	\end{proof}
	
	\begin{remark}
		Combining Jordan curve theorem on $\mathbb{S}^2$ and the stereographic projection, we see that for any bi-infinite path $\gamma$ on $\R^2$ the set $\R^2 \setminus \gamma$ can be written as the union of two distinct connected components. 
	\end{remark}
	
	We record below a technical lemma which will be used in the proof of Proposition~\ref{prop:infinite-disjoint-paths-limit}. 
	\begin{figure}
		\centering
		\fbox{\includegraphics[width=0.45\textwidth]{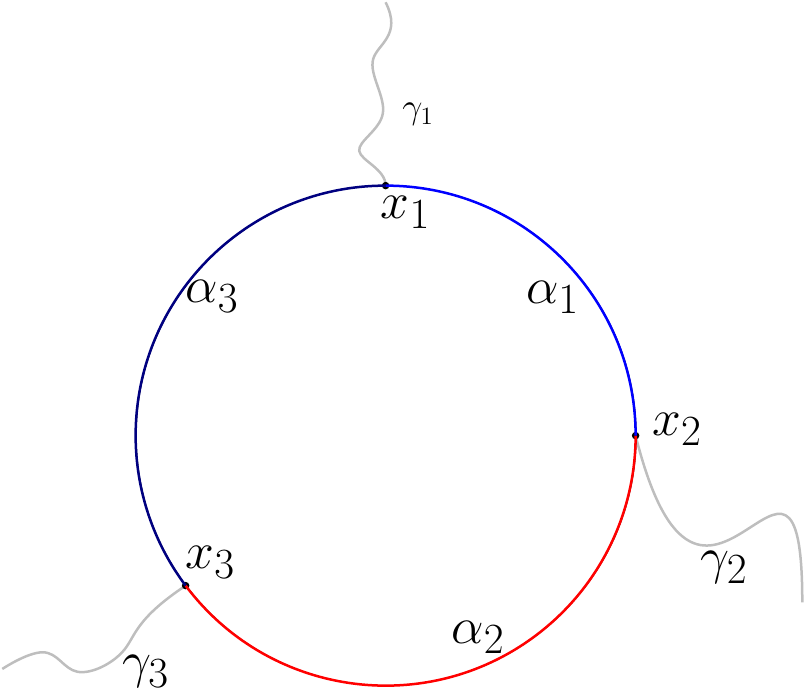}}
		\fbox{\includegraphics[width=0.45\textwidth]{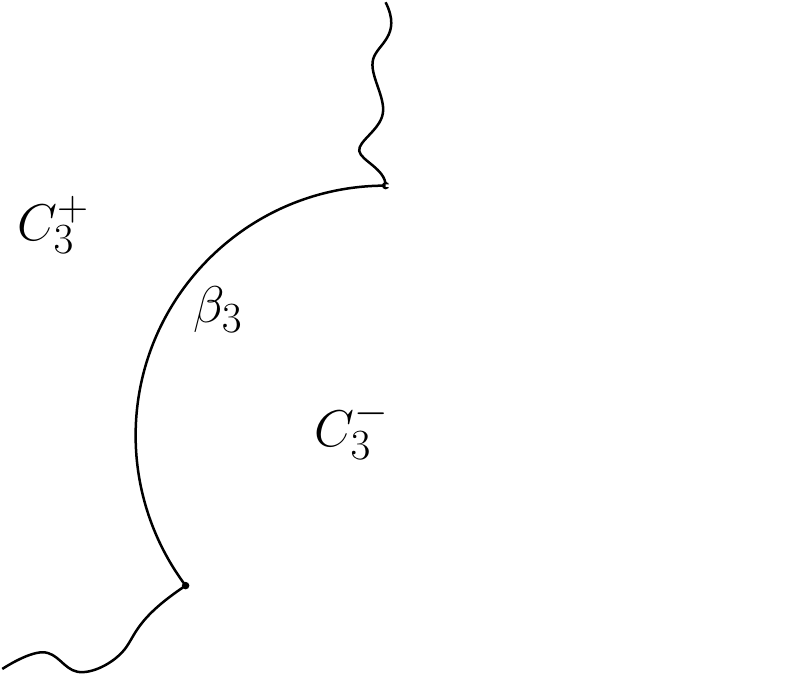}}
		\caption{An example of the decomposition in the proof of Lemma \ref{lemma:jordan-curves-theorem}. The curves $\alpha_1, \alpha_2, \alpha_3$ are denoted by blue, red, and navy arcs respectively. The gray curves
			are the infinite paths $\gamma_i$. The concatenated curve $\beta_3$ defined in \eqref{eq:concatenated-curve} is shown in black on the right together with 
			the complementary connected components $C_{3}^{\pm}$.}
		\label{fig:jordan-curves}
	\end{figure}

	\begin{lemma}[Jordan curves theorem on $\S^2$] \label{lemma:jordan-curves-theorem}
		Fix an integer $K \geq 2$ and let $x_1, \ldots, x_K$ be $K$ distinct points of $\S^2$ and let $\gamma_1 , \ldots, \gamma_K$ be $K$ distinct paths such that $\gamma_i(0) = x_i$ and $\gamma_i(1) = N$, and there exists a loop $\alpha$ such that, for any $i \in \{1 , \ldots, K\}$, $\alpha \cap \gamma_i = \{ x_i \}.$
		Then the set $\S^2 \setminus \left( \alpha \cup \bigcup \gamma_i\right)$ has at least $K$ connected components whose closure contain the north pole $N$.
	\end{lemma}
	
	\begin{proof}
		See Figure \ref{fig:jordan-curves} for a visual description of the argument below. 
		
		Let us denote by $\Gamma := \alpha \cup \bigcup \gamma_i  $. Without loss of generality, we may see the loop $\alpha$ as a function defined on the interval $[0 , 1]$, valued in $\S^2$ such that $\alpha(0) = \alpha(1)$, and assume that there exist $ 0 \leq t_1 < \ldots < t_K < 1$ such that $\alpha(t_i) = x_i$. For any $i \in \{ 1 , \ldots, K \},$ we denote by $\alpha_i$ the path obtained by restricting $\alpha$ to the interval $[t_i , t_{i+1}].$ For each $i \in \{ 1 , \ldots, K\},$ we then let $\beta_i$ be the loop defined by
		\begin{equation} \label{eq:concatenated-curve}
			\beta_i := \gamma_{i-1} \cup \alpha_i \cup  \gamma_{i}.
		\end{equation}
		Applying Jordan curve theorem, we know that the set $\S^2 \setminus \beta_i$ can be written as the union of two connected components which we denote by $C_{i}^+$ and $C_{i}^-$. We next observe that the set $\Gamma \setminus \beta_i$ is connected and disjoint from $\beta_i$, this implies that it is included in either $C_{i}^+$ or $C_{i}^-$. We may assume without loss of generality that it is included in $C_{i}^-$. This implies that the set $C_{i}^+$ is included in $\S^2 \setminus \Gamma$, and thus $C_{i}^+$ is a connected component of $\S^2 \setminus \Gamma$. Consequently, the family $(C_{i}^+)_{1 \leq i \leq K}$ is a collection of connected components of $\S^2 \setminus \Gamma$, and since they do not have the same boundary, they must be disjoint. Finally, for any $i \in \{ 1 , \ldots, K\}$, we have $\partial C_{i}^+ = \beta_i$ (by Jordan curve theorem) and thus $N \in \partial C_{i}^+$. The collection $(C_{i}^+)_{1 \leq i \leq K}$ satisfies the conclusions of the lemma.
	\end{proof}
	
	\iffalse
	Consider a finite subset $S \subseteq \Z^2$ and implement the following construction:
	\begin{itemize}
		\item First consider the subset of $\R^2$
		\begin{equation*}
			\mathcal{S} := \cup_{z \in \Z^2} \left(z + \left( - \frac{1}{2} , \frac{1}{2} \right)^d \right)
		\end{equation*}
		\item Consider the set $\R^2 \setminus \mathcal{S}$. This set has one infinite connected component denoted by $C_\infty$ and by $C_1 , \ldots, C_N$ the finite connected components. Then denote by
		\begin{equation*}
			\tilde{\mathcal{S}} := \mathcal{S} \cup \cup_{i = 1}^N C_i.
		\end{equation*}
	\end{itemize}
	Then we have the following results:
	\begin{itemize}
		\item The set $\tilde{\mathcal{S}}$ is open, connected, bounded and Lipschitz, the set $\Z^2 \setminus \tilde{\mathcal{S}}$ is connected, hence $\tilde{\mathcal{S}}$ is simply connected
		\item The boundary of $\tilde{\mathcal{S}}$ is simple, thus $\partial \tilde{\mathcal{S}} \simeq \mathcal{S}^1$ (equivalently, it is the image of a Jordan curve). Denote by $\gamma : \S^1 \to \partial \tilde{\mathcal{S}}$ the homeomorphism.
		
		Even stronger: there exists an homeomorphism $h : \bar{\tilde{\mathcal{S}}} \to \bar{\mathbb{D}_1}$ (then $\gamma$ is just the restriction of $h$ to $\S^1 \subseteq \bar{\mathbb{D}_1}$)
	\end{itemize}
	\fi

	\section{Lipschitz harmonic functions} \label{sec:lipschitz}
	In this section we prove Theorem \ref{theorem:lipschitz}. The proof utilizes the uniqueness of the first-order corrector 
	together with the finite energy of Bernoulli percolation. Roughly, we construct a modification of the environment involving a finite number of edges
	upon which the first-order corrector must have an arbitrarily large Lipschitz constant. 
	
	We first show in Lemma \ref{lemma:reduction-to-slope-e1}, using an ergodicity argument, that if there are any Lipschitz harmonic functions on the cluster, then the unique modulo additive constant harmonic function which grows like $e_1 \cdot x$ at infinity, $\ellp{e_1}$, is $\overline{L}$-Lipschitz for a deterministic constant $\overline{L} > 0$,
	with probability one. 
	\begin{lemma} \label{lemma:reduction-to-slope-e1}
		If, with positive probability, there exists a non-constant harmonic function which is Lipschitz on $\mathscr{C}_{\infty}$, then 
		there exists a deterministic constant $L > 0$ such that, with probability one, 
		$\ellp{e_1}$ is $L$-Lipschitz.
	\end{lemma}
	\begin{proof}
		We start with the observation that a function is $L$-Lipschitz if and only if its gradient is bounded (in absolute value) by $L$. In particular, even though the corrector is only defined up to additive constant, the property that the corrected plane $\ellp{e_1}$ is almost surely $L$-Lipschitz is well-defined.
		We next recall that, almost surely, every harmonic function on $\mathscr{C}_{\infty}$ with linear growth $p$ at infinity is equal (modulo additive constant) to the corrected plane $\ellp{p}$.

		We split the proof into two steps. We first construct a deterministic slope $\overline{p}$
		and Lipschitz constant $\overline{L} > 0$ such that $\ellp{\overline{p}}$ is almost surely $\overline{L}$-Lipschitz. We then use the lattice rotation invariance 
		of Bernoulli percolation together with linearity to show that $\ellp{e_1}$ is $L$-Lipschitz with probability one (for a deterministic constant $L$).

		{\it Step 1: Finding a deterministic slope.} \\ 
		For $L > 0$, let the event $E_L$ be the event ``there exists a nonconstant function $u \in \mathcal{A}(\mathscr{C}_\infty)$ which is Lipschitz with Lipschitz constant at most $L$''. Since $E_L$ is translation-invariant, by ergodicity of Bernoulli percolation, \eg, \cite[Proposition 7.3]{lyons-peres-book}, it has probability either $0$ or $1$. Thus, by the assumption $\P[\cup_{L > 0} E_L] > 0$ and a union bound over integers, 
		there exists a deterministic $\overline{L} > 0$ such that $\P[E_{\overline{L}}] = 1$. 
		
		Every harmonic function which is Lipschitz is contained in $\mathcal{A}_1(\mathscr{C}_{\infty})$. In particular, by Theorem \ref{theorem:large scale-regularity},
		we may rewrite the event as 
		\[
		E_{\overline{L}} = \{ \mbox{there exists a slope $p \in \R^d \setminus \{ 0 \}$ with $|p| \leq \overline{L}$ such that $\ellp{p}$ is $\overline{L}$-Lipschitz} \}.
		\]	
		We now show that the slope $p$ in the above event can be made deterministic. First, we note that the set of slopes $p \in \Rd$ such that $\ellp{p}$ is $\overline{L}$-Lipschitz is a closed subset of $\R^d$.  
		Indeed, fix a realization of the infinite cluster and consider a sequence of slopes $p_N \to \overline{p}$ such that each $\ellp{p_N}$ is $\overline{L}$-Lipschitz. As, for any pair of vertices $x , y \in \mathscr{C}_\infty$, the map $p \mapsto \ellp{p}(x) - \ellp{p}(y)$ is well-defined and linear, we see that $\ellp{p_N}(x) - \ellp{p_N}(y) \to \ellp{\overline{p}}(x) - \ellp{\overline{p}}(y)$ as $N \to \infty$. 
		As each $\ellp{p_N}$ is Lipschitz, we deduce that $\ellp{\overline{p}}$ is Lipschitz.
		
		Now,  consider the random variable
		\[
		\overline{p} = \arg \sup \{  |p| \leq \overline{L}  \mid \mbox{$\ellp{p}$ is $\overline{L}$-Lipschitz} \},
		\]
		where the $\arg \sup$ indicates a choice of slope $p$ such that $|p| \leq \overline{L}$ is maximized, 
		and ties are broken by choosing the lexicographically largest such $p$ (this quantity is well-defined since the set of slopes $p \in \Rd$ such that $\ellp{p}$ is $\bar L$-Lipschitz is closed).  The random variable $\overline{p}$ is translation invariant and thus deterministic and we have just seen that
		\begin{equation} \label{eq:det-p-lipschitz}
			\ellp{\overline{p}} \mbox{ is $\overline{L}$-Lipschitz},
		\end{equation}
		which completes this step. 
		\medskip 
		
		{\it Step 2: Performing a rotation.} \\ 
		For $L > 0$ and $p \in \R^d$, let $E_L^p$ denote the event  ``the function $\ellp{p}$ has Lipschitz constant $L$''.
		For $j \in \{1, \ldots, d\}$ and $p \in \R^d$, denote reflection around the $j$-th axis by 
		\[
		T_j \circ p := (p_1, \ldots, p_{j-1}, -p_j, p_{j+1}, \ldots p_d).
		\]
		By dihedral symmetry of $\Z^d$ (and of Bernoulli percolation), for each $p \in \R^d$ and $L > 0$, 
		\[
		\P[E_{T_j \circ p}^L] = \P[E_p^L], \quad \forall j \in \{1, \ldots, d\}.
		\]
		Further, with $\overline p$ as in Step 1, we have that 
		\[
		\P[E_{T_j \circ \overline p}^{\overline{L}}] = 1, \quad \forall j \in \{1, \ldots, d\}.
		\]
		Assuming without loss of generality that $\bar p_1 >0 $, we deduce from the previous identity that 
		\[
		\ellp{2\overline{p}_1 e_1} = \ellp{\overline{p}} + \ellp{T_2 \circ \ldots \circ T_d \circ \overline{p}}
		\]
		is $2 \overline{L}$-Lipschitz. By linearity, this implies $\P\left[E_{e_1}^{ \overline{L}/p_1}\right] = 1$, completing the proof.
	\end{proof}

	It remains to prove the following, which together with the previous lemma, implies Theorem \ref{theorem:lipschitz}.
	\begin{prop} \label{prop:not-lipschitz}
		For each $L \geq 1$, one has inequality
		\begin{equation*}
			\P \left( \left| \nabla \ellp{e_1}(0) \right|  > L \right) > 0.
		\end{equation*}
		As a consequence, the map $\ellp{e_1}$ is not $L$-Lipschitz with positive probability. 
	\end{prop}
	\begin{proof}[Proof of Theorem \ref{theorem:lipschitz} assuming Proposition \ref{prop:not-lipschitz}]
		Combine Lemma~\ref{lemma:reduction-to-slope-e1} and Proposition~\ref{prop:not-lipschitz}. 
	\end{proof}

	In order to prove this, we need the following lemma, which we use to compare corrected planes in different environments. 
	\begin{lemma}  \label{lemma:compare-different-planes}
		For almost every realization of the infinite cluster $\mathscr{C}_\infty$ and every finite collection of edges $\mathcal{B} \subseteq E(\mathscr{C}_\infty)$ satisfying the property that removing all the edges of $\mathcal{B}$ does not disconnect $\mathscr{C}_\infty$, the following properties hold.
		First, the corrected plane is well-defined (modulo additive constant) on the cluster $\mathscr{C}_\infty$ and the cluster $\mathscr{C}_\infty'$ from which all the edges of $\mathcal{B}$ have been removed. Second, if we denote these corrected planes by $\ellp{p}$ and $\ellp{p}'$ respectively, then we have the identity
		\begin{equation}  \label{eq:compare-corrected-planes-with-green}
			\nabla \ellp{p} - \nabla \ellp{p}' = \sum_{\substack{e \in E \left(\mathscr{C}_\infty\right) \cap \mathcal{B} \\ e = (x , y)}} \nabla_x \nabla_y G(\cdot , e) (\ellp{p}'(y) - \ellp{p}'(x)),
		\end{equation}
		where $G$ denotes the Green's function on the cluster $\mathscr{C}_\infty$ (including the edges of $\mathcal{B}$).
	\end{lemma}
	
	\begin{remark}
		In the identity~\eqref{eq:compare-corrected-planes-with-green}, we used the notation~\eqref{id:vectorfiedledges} for the gradient in the second variable of the Green's function taking a directed edge as input.
	\end{remark}

	\begin{proof}
		The first part of the statement is a consequence of Lemma~\ref{lemma:finite-energy} applied with the event of full measure where the following properties are satisfied:
		\begin{itemize} 
			\item Sublinear and harmonic function are constant on the infinite cluster,
			\item The corrected plane is well-defined (up to an additive constant) and grows sublinearly, 
			\item the Green's function (and its gradient) are well-defined and grow sublinearly. 
		\end{itemize}
		Note that removing the edges of $\mathcal{B}$ does not modify the set of vertices of the infinite cluster, and thus the two corrected planes are defined on the same set of vertices.
		
		The identity~\eqref{eq:compare-corrected-planes-with-green} is then proved by noting that the function 
		\[
		h(z):= \ellp{p}(z) - \ellp{p}'(z) - \sum_{\substack{e \in \mathcal{B} \\ e = (x , y)}}  \nabla_y G(z , e) (\ellp{p}'(y) - \ellp{p}'(x))
		\] 
		is harmonic and sublinear on the infinite cluster $\mathscr{C}_\infty$, and hence constant.
		Indeed, to see that~$h(x)$ is sublinear, we observe that~$\ellp{p} - \ellp{p}'$ is sublinear and, by Proposition~\ref{prop:homogenizationellipticgreen}, for every edge~$e \in \mathscr{C}_{\infty}$ the function
		$\nabla_y G(\cdot , e)$ is also sublinear. The fact that~$h$ is harmonic follows from the identity, 
		for all edges~$e \in E(\mathscr{C}_{\infty})$ with~$e = (x,y)$, 
		\[
		\Delta_{\mathscr{C}_{\infty}} \nabla_y G(\cdot, e)(\ellp{p}'(y) - \ellp{p}'(x))
		= (\delta_y -\delta_x )(\ellp{p}'(y) - \ellp{p}'(x))
		\]
		and, since~$\ellp{p}'$ is harmonic on the cluster from which all the edges of $\mathcal{B}$ have been removed, 
		\[
		\Delta_{\mathscr{C}_{\infty}} \ellp{p}'(x) = \sum_{\substack{e \in \mathcal{B} \\ e = (x , y)}}(\ellp{p}'(y) - \ellp{p}'(x))   \, . 
		\]
		This completes the proof.
	\end{proof}

	\begin{figure}
		\includegraphics[width=0.5\textwidth]{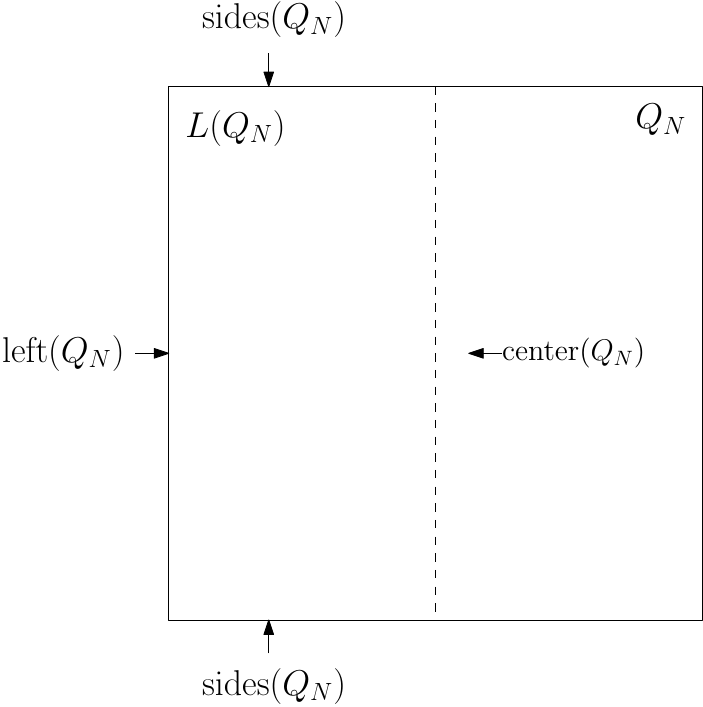}
		\caption{Subsets of the box of radius $N$, $Q_N$, as defined in \eqref{eq:box-decomp} and used in the proof of Proposition \ref{prop:not-lipschitz}.} \label{fig:box-decomp}
	\end{figure}
	
	\begin{figure}
		\includegraphics[width=0.5\textwidth]{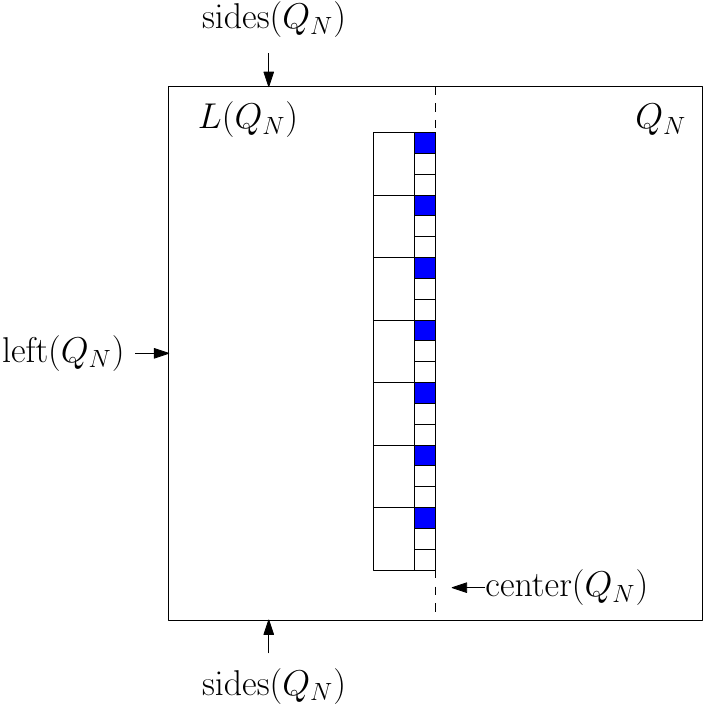}
		\caption{Partitions of $Q_N$ used in the proof of Proposition \ref{prop:not-lipschitz}. Some of the small cubes and medium cubes
			which are adjacent to $\mbox{center}(Q_N)$ are shown. All small boxes other than the blue ones have had associated horizontal edges removed according to the algorithm described at the beginning of Step 2.	
		} \label{fig:box-partition}
	\end{figure}

	\begin{figure}
		\includegraphics[width=0.3\textwidth]{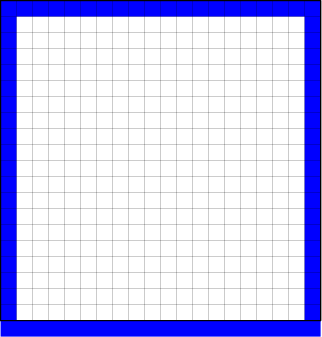}
		\caption{The projection on $\mbox{center}(Q_N)$ of a medium cube (whose size is $N^{1/(10d)}$) partitioned into small cubes (whose sizes are $N^{1/(10d)^2}$). A boundary layer of blue cubes is depicted in blue. The horizontal edges of the blue cubes are not altered.
		} \label{fig:box-partitionboundarylayer}
	\end{figure}

	\begin{proof}[Proof of Proposition~\ref{prop:not-lipschitz}]
		Suppose for the sake of contradiction that for some~$L \geq 1$, one has
		\begin{equation*}
			\P \left( \left| \nabla \ellp{e_1}(0) \right|  > L \right) = 0 .
		\end{equation*}
		\iffalse
		As the random variable~$X := (\ellp{e_1}(e_1) - \ellp{e_1}(0)) \indc_{(0 , e_1) \in E \left( \mathscr{C}_\infty\right)}$ is stationary with respect to~$\Zd$ translations, the event~$\{X > L\}$
		has probability zero or 1. 	Consequently, by rotational symmetry,~$\ellp{e_1}$ is $L$-Lipschitz with probability one. In particular, by Lemma \ref{lemma:finite-energy} the corrected plane in the environment where any finite collection of edges are made open/closed is also well-defined and $L$-Lipschitz.  We will use this fact below, but will first restrict to a high probability event. 
		\fi
		This implies that, almost surely, the gradient of the corrected plane along any edge of the infinite cluster is smaller than $L$. In particular, by Lemma~\ref{lemma:finite-energy} the corrected plane in the environment where any finite collection of edges are made open/closed is also well-defined and satisfies the property that its discrete gradient is bounded by the constant $L$. We will use this fact below, but will first restrict to a high probability event.
		
		{\it Step 1: Set up and restrict to a high probability event.} \\
		For the statement of the event, let $N > 0$ and recall that $\mathscr{C}_*(Q_N)$ is the largest cluster of open edges contained in $Q_N$. 
		We consider the following subset of vertices of $\cl(Q_N)$, see Figure \ref{fig:box-decomp},
		\begin{equation} \label{eq:box-decomp}
			\begin{aligned}
				\mbox{L}(Q_N) &= \{ x \in Q_N : x_1 \leq 0 \} \\
				\mbox{left}(Q_N) &= \{ x \in \partial Q_N : x_1 = -N\}  \\
				\mbox{sides}(Q_N) &= \{ x \in \partial Q_N \cap L(Q_N)\} \\
				\mbox{center}(Q_N) &= \{ x \in Q_N : x_1 = 0 \}.
			\end{aligned}
		\end{equation}
		Let $E_N$ denote the event that the following occurs: 
		\begin{itemize}
			\item The cube $Q_N$ is well-connected as in Proposition \ref{prop:well-connected}.
			\item Green's function decay: the minimal scale for the mixed derivative of the Green's function, as defined in Proposition~\ref{prop:green-mixed-derivative}, is bounded 
			at every edge in $Q_N$: if we let $\delta = \frac{1}{4}$ in the statement of Proposition~\ref{prop:green-mixed-derivative}, then we assume that, for every $x \in Q_N$,
			\[
			\mathcal{M}_{\nabla \nabla \mathrm{-Decay}, \frac 14}(x) \leq N^{1/(10 d^2)}.
			\]
			\item Estimate for the flux: there exists a constant $\bar{\mathbf{a}} := \bar{\mathbf{a}}(d , \mathfrak{p}) \in (0 , 1)$ such that
			\begin{equation} \label{eq:left-side-large}
				\sum_{x \in \mathrm{left}(Q_N)} (\ellp{e_1}(x) - \ellp{e_1}(x+e_1)) \mathbf{1}_{\{ (x,x+e_1) \in \mathscr{C}_{\infty}\}} \geq \frac{\bar{\mathbf{a}}}{2} \times N^{d-1} 
			\end{equation}
			and, for all $j \in \{ 2 , \ldots, d\}$, 
			\begin{equation} \label{eq:other-sides-small}
				\left| \sum_{x \in \mathrm{sides}(Q_N)} \left( \ellp{e_1}(x) - \ellp{e_1}(x \pm e_j) \right) \mathbf{1}_{\{ (x,x \pm e_j) \in \mathscr{C}_{\infty}\}} \right| \leq \bar{\mathbf{a}} \times 10^{-100} N^{d-1}.
			\end{equation}
		\end{itemize}
		The first property of $E_N$ occurs with probability approaching one as $N \to \infty$ by Proposition \ref{prop:well-connected}. The second property also occurs with probability approaching one as $N \to \infty$ by the stochastic integrability property~\eqref{eq:stoch.estMdecay}. The third event occurs with probability approaching one as $N \to \infty$ by the ergodic theorem (the ergodicity comes from \cite[Theorem 3.2]{berger-rw-percolation}), together with the observation that $\E \left[ (\ellp{e_1}(x) - \ellp{e_1}(x+e_1)) \mathbf{1}_{\{ (x,x+e_1) \in \mathscr{C}_{\infty}\}} \right] = \bar{\mathbf{a}} > 0$ (the homogenized coefficient of the percolation cluster) and, by symmetry arguments for $j \neq 1$, $\E \left[ (\ellp{e_1}(x) - \ellp{e_1}(x+e_j)) \mathbf{1}_{\{ (x,x+e_j) \in \mathscr{C}_{\infty}\}} \right] = 0$.
		Thus, by a union bound, the event $E_N$ occurs with probability approaching one as $N \to \infty$. 
		\medskip 
		
		{\it Step 2: Alter environment.} \\
		Fix an $N$ large to be determined below and restrict to the event $E_N$. Consider the infinite cluster
		in the following environment. First, partition $Q_N$ into `medium' cubes of side length $N^{1/(10 d)}$ and further partition into `small' well-connected cubes of side length $N^{1/(10d)^2}$ in such a way that there is a column of medium and small cubes with faces containing $\mbox{center}(Q_N)$
		--- see Figure \ref{fig:box-partition}.

		We define the set of horizontal edges: 
		\[
		\mbox{horizontal edges} = \{   \{ (0, x_2, \ldots, x_d) , (1, x_2, \ldots, x_d) \} \mbox{ for $x \in Q_N$} \}.
		\]
		We will remove most of the horizontal edges. We first label some of the horizontal edges according to the following procedure: 
		\begin{itemize}
			\item Proceed through the list of medium cubes which are adjacent to the center, and for each such cube, enumerate all the horizontal edges within it, except for a boundary layer of small `blue' cubes which are adjacent to a the boundary of the medium cube (see Figure \ref{fig:box-partition} and Figure~\ref{fig:box-partitionboundarylayer}) and the cubes which are not in $Q_{N - N^{1/d}}$;
			\item Proceed through the list of edges which have been enumerated in the previous step, and for each such edge $e$,
			\begin{itemize}
				\item if removing $e$ disconnects the current cluster, do not remove the edge
				\item otherwise; remove the edge and update the cluster. 
			\end{itemize}
		\end{itemize}
		
		\begin{figure}
			\centering
			\includegraphics[width=0.45\textwidth]{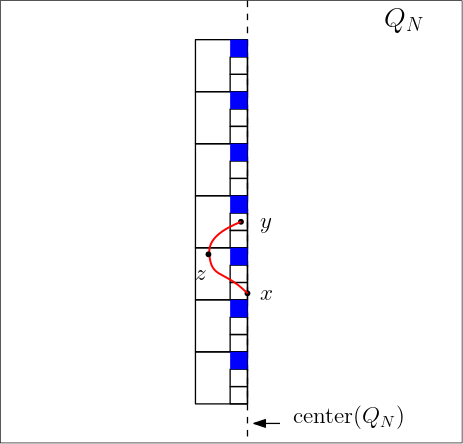}
			\hspace{3mm}
			\includegraphics[width=0.45\textwidth]{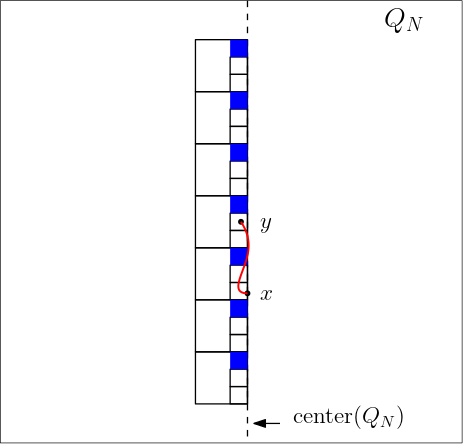}
			\caption{Proof of the inequality~\eqref{eq:short-paths}: any path starting from a vertex $x$ in $\mbox{center}(Q_N)$ which is too long has to either reach a point which is far from $\mbox{center}(Q_N)$ (picture on the left) or cross the blue region where the edges have not been altered (picture on the right).}
			\label{fig:longpaths}
		\end{figure}
		
		Denote the infinite cluster in this altered environment by $\mathscr{C}'_{\infty}$ and denote its edge set by $E \left(\mathscr{C}'_{\infty}\right)$.
		We observe that the following three properties are satisfied:
		\begin{itemize}
			\item Connected: the set of vertices in the altered environment does not change 
			\begin{equation} \label{eq:new-cluster-same-as-old}
				\mathscr{C}'_{\infty} = \mathscr{C}_{\infty} \hspace{5mm} \mbox{and} \hspace{5mm}  E \left(\mathscr{C}'_{\infty}\right) = E \left(\mathscr{C}_{\infty}\right) \setminus \mathcal{B}.
			\end{equation}
			This property follows from the definition the previous algorithm, since at each step, the vertex set of the cluster remains unchanged (as an edge is deleted if and only if removing it does not modify the vertex set of the cluster).
			\item Short paths: for every deleted edge $e = (x, x+e_1)$ we have that 
			\begin{equation} \label{eq:short-paths}
				\dist_{\mathscr{C}'_{\infty}}(x, x + e_1) \leq C N^{1/2},
			\end{equation}
			where $\dist_{\mathscr{C}'_{\infty}}$ denotes the graph distance in $(\mathscr{C}'_{\infty}, E \left( \mathscr{C}'_{\infty}\right))$. 
			
			\iffalse 
			This property is obtained by using that the horizontal edges on the well-connected `blue' cubes have not been removed from the infinite cluster. Indeed, for every such~$x$, since small cubes are crossing, there is a path from~$x$ to the small cube to the left of it of length at most~$2 N^{\frac{1}{100 d}}$. Iterating this, we may find a path to a nearby blue cube of length at most~$C N^{\frac{1}{10}}$. We then use the well-connectedness of the blue cubes and repeat 
			the argument for small well-connected cubes to the right of the center. 
			\fi
			
			This property is obtained by the following argument. First the vertex $x$ belongs to one of the medium cubes. We denote this cube by $Q_x$ and introduce the set
			\begin{equation*}
				\mbox{center}(Q_x) := \{ y \in Q_x : y_1 = 0 \}.
			\end{equation*}
			Since the vertices $x$ and $x + e_1$ are both in the cluster $\mathscr{C}'_{\infty}$, there exists a path which connects $x$ and $x + e_1$ and lies in $\mathscr{C}'_{\infty}$. We denote by $\gamma$ such a path and assume without loss of generality that its length is equal to the distance between $x$ and $x + e_1$ in the cluster $\mathscr{C}'_{\infty}$. We then distinguish two cases: whether $|\gamma| \leq 2N^{1/5}$ or $|\gamma| \geq 2N^{1/5}$.
			
			In the first case, the inequality~\eqref{eq:short-paths} holds since $N^{1/5} \leq N^{1/2}$. In the second case, we claim that there must exist a vertex $y \in \gamma$ at distance (in the cluster $\mathscr{C}'_{\infty}$) less than $4 N^{1/5}$ from $x$ such that the distance (in $\Zd$) of $y$ to the set $\mbox{center}(Q_x)$ is larger than $2 N^{1/(10d)^2}$. This follows from the computation
			\begin{align*}
				\left| \left\{ y \in \Zd \, : \, \dist(y , \mbox{center}(Q_x)) \leq 2N^{1/(10d)^2} \right\}  \right| & \leq \size(Q_x)^{d-1} \times (4 N^{1/(10d)^2}) \\
				& \leq 4 N^{(d-1)/(10d) + 1/(10d)^2} \\
				& \leq 4 N^{1/10 + 1/100} \\
				& < 4 N^{1/5}.
			\end{align*}
			The first $4 N^{1/5}$ vertices visited by the path $\gamma$ (which are all disjoint by definition of a path) must include a vertex which is not in this set.
			
			There are then two possibilities (see Figure~\ref{fig:longpaths}):
			\begin{itemize}
				\item Either the path $\gamma$ visits a vertex, denoted by $z$ in the rest of the proof, which is at distance (in the cluster $\mathscr{C}'_{\infty}$) less than $2N^{1/5}$ from $x$ and is at distance (in $\Zd$) larger than $2N^{1/(10d)^2}$ from $\mbox{center}(Q_N)$ (N.B. $\mbox{center}(Q_N)$ is a larger set than $\mbox{center}(Q_x)$),
				\item Or the path $\gamma$ crosses the boundary layer of blue cubes surrounding the box $Q_x$ or the reflection of the blue cubes across the plane $\mbox{center}(Q_N)$.
			\end{itemize}
			
			\smallskip
			
			In the first case, the box of center $z$ and of size $N^{1/(10d)^2}$ is well-connected for the cluster $\mathscr{C}'_{\infty}$ (as it was well-connected in $\mathscr{C}_{\infty}$ and none of the edges in the box of center $z$ and of size $2N^{1/(10d)^2}$ have been deleted). The path $\gamma$ is thus connected to the largest cluster of this box. Once the connection with a well-connected cube is established, we can use the fact the largest open clusters of two neighboring well-connected boxes of the same size are connected to follow a path of well-connected boxes in the cluster $\mathscr{C}'_{\infty}$ (staying at distance at least $2N^{1/(10d)^2}$ from $\mbox{center}(Q_N)$) to one of the blue cubes on the boundary of the box $Q_x$. These operations generate a path connecting $x$ to one of the blue cubes on the boundary of $Q_x$ whose length is smaller than $2 N^{1/5} + C N^{1/5} N^{d/(10d)^2} \leq C N^{1/2}$: the first term is the length of the path from $x$ to $z$, the second term is an overestimation of the length needed to go from $z$ to one of the blue cubes following well-connected boxes (the path can be constructed using at most $C N^{1/5}$ boxes and to go from one box to its neighbour, we can always use less than $CN^{d/(10d)^2}$ vertices, which is the total number of vertices in the union of the two boxes).
			
			In the second case, the path $\gamma$ must cross the boundary layer of blue cubes (or the reflection of the blue cubes across the plane $\mbox{center}(Q_N)$, the argument below is identical in both cases). Since none of the edges of the blue cubes have been deleted, and since the blue cubes are well-connected in the cluster $\mathscr{C}_{\infty}$, we obtain that the path $\gamma$ must be connected to the largest connected component of one of these cubes.
			
			In both cases, we have found a path connecting $x$ to one of the blue cubes on the boundary of $Q_x$ which lies inside $\mathscr{C}'_{\infty}$ and has length smaller than $C N^{1/2}$. We may use the same argument to find a path connecting $x + e_1$ to one of the blue cubes on the boundary of $Q_x$. Finally, using the well-connectedness of the blue cubes (and the fact that we did not remove edges in this region), we may connect any pair of blue cubes on the boundary of the box $Q_x$ using at most $C N^{1/(10d)} N^{d/(10d)^2} \leq C N^{1/2}$ vertices.
			\item Few horizontal edges contributing to the flux: at most $N^{d-1-1/2}$ edges which contribute to the flux are left behind, \ie,
			for any harmonic function $\ell'$ on $L(Q_N)$,
			\begin{equation} \label{eq:divergence-theorem-flux}
				\begin{aligned}
					0 &= \sum_{x \in \mathscr{C}_*(L_N)} \Delta_{\mathscr{C}_{\infty}'} \ell'(x) \\
					&= \sum_{e \in \mathrm{sides}(Q_N) \cap \partial_e \mathscr{C}_*(L_N)} \mathbf{n} \cdot \nabla \ell'(e)  \\
					& \quad + \sum_{e \in \mathrm{left}(Q_N)\cap \partial_e \mathscr{C}_*(L_N)} \mathbf{n} \cdot \nabla \ell'(e)    \\
					& \quad + \sum_{e \in \mathrm{center}(Q_N)\cap \partial_e \mathscr{C}_*(L_N)} \mathbf{n} \cdot \nabla \ell'(e),
				\end{aligned}
			\end{equation}
			and we note that
			\begin{equation} \label{eq:horizontal-flux-small}
				\sum_{e \in \mathrm{center}(Q_N) \cap E \left( \mathscr{C}_\infty' \right)} \mathbf{1}_{\{\mathbf{n} \cdot \nabla \ell'(e)   \neq 0 \}} \leq C N^{d - 1- 1/(10d)+ 1/(10d)^2}.
			\end{equation}
			For a set $A \subset \mathscr{C}_{\infty}$ the notation $\mathbf{n} \cdot \nabla f(e)$ for an edge on $\partial_e A$ denotes the difference  $f(x) - f(y)$
			where $x \in A$ and $y \not \in A$, \ie, it is the discrete normal derivative. The inequality~\eqref{eq:horizontal-flux-small} is obtained by noting that only the edges in the boundary layer $Q_N \setminus Q_{N - N^{1/d}}$ and the edges in the boundary of the blue cubes contribute to the sum. Indeed all the other horizontal edges have either been removed, or if they have not been removed, the gradient of the corrected plane through the edge has to be equal to $0$ (N.B. this follows from the observation that the gradient of any harmonic function on the infinite percolation cluster evaluated at an edge $e$ satisfying the property that removing $e$ disconnect the cluster must be equal to $0$).
		\end{itemize}
		\medskip
		
		{\it Step 3: Contradiction.} \\
		Denote the set of deleted edges by $\mathcal{B}$ and the corrected plane in $\mathscr{C}_{\infty}'$ by ${\ellp{e_1}}'$ so that by Lemma~\ref{lemma:compare-different-planes},
		\[
		\nabla \ellp{e_1} - \nabla {\ellp{e_1}'} = \sum_{\substack{e \in \mathcal{B} \\ e = (x , y)} } \nabla_x \nabla_y G(e, \cdot) \left( {\ellp{e_1}}'(y) - {\ellp{e_1}}'(x) \right).
		\]
		As ${\ellp{e_1}}'$ is $L$-Lipschitz, by \eqref{eq:short-paths}, we have for every deleted edge $e = (x, x + e_1)$, 
		\begin{equation} \label{eq:deleted-gradient-bound}
			|{\ellp{e_1}}'(x) - {\ellp{e_1}}'(x + e_1)| \leq L C N^{\frac12}.
		\end{equation}
		By combining the previous two displays with \eqref{eq:left-side-large}, Proposition~\ref{prop:green-mixed-derivative}, and the fact that $|\mathcal{B}| \leq C N^{d-1}$, we have that 
		\begin{equation} \label{eq:altered-grad-left-bound}
			\begin{aligned}
				\lefteqn{\sum_{x \in \mathrm{left}(Q_N)} ({\ellp{e_1}}'(x) - {\ellp{e_1}}'(x+e_1)) \mathbf{1}_{\{ (x,x+e_1) \in E \left(\mathscr{C}'_{\infty} \right)\}} } \qquad & \\
				&\geq \sum_{x \in \mathrm{left}(Q_N)} (\ellp{e_1}(x) - \ellp{e_1}(x+e_1)) \mathbf{1}_{\{ (x,x+e_1) \in E \left(\mathscr{C}_{\infty}\right)\}}  \\
				& \qquad 
				- \sum_{e' \in \mathrm{left}(Q_N)} \sum_{\substack{e \in \mathcal{B} \\ e = (x , y)}} \left|\nabla_x \nabla_y G(e, e') \right| \left| {\ellp{e_1}}'(y) - {\ellp{e_1}}'(x)\right| \\
				&\geq \frac{\bar{\mathbf{a}}}{2} \times N^{d-1}   -  C N^{d-1} \times N^{d-1} \times N^{- d + \frac{1}{4}}  \times (L N^{\frac12})   \\
				&\geq  \frac{\bar{\mathbf{a}}}{4} N^{d-1},
			\end{aligned}
		\end{equation}
		for $N$ sufficiently large. 
		Similarly,
		\begin{equation}  \label{eq:altered-grad-top-and-bottom-bound}
			\left | \sum_{x \in \mathrm{sides}(Q_N)} ({\ellp{e_1}}'(x) - {\ellp{e_1}}'(x \pm e_j)) \mathbf{1}_{\{ (x,x\pm e_j) \in \mathscr{C}'_{\infty}\}}  \right| \leq \bar{\mathbf{a}} \times 10^{-50} N^{d-1}.
		\end{equation}
		Thus, we have, by \eqref{eq:divergence-theorem-flux}
		\[
		\sum_{e \in \mathrm{center}(Q_N) \cap E\left( \mathscr{C}_\infty' \right)} \nabla {\ellp{e_1}}'(e) \geq \frac{\bar{\mathbf{a}}}{8} N^{d-1}.
		\]
		However, by \eqref{eq:horizontal-flux-small} and for $N$ sufficiently large, 
		this implies the existence of some edge with gradient of ${\ellp{e_1}}'$ larger than $L$, a contradiction. \end{proof}

	\section{Integer-valued harmonic functions} \label{sec:integer-harmonic}
	%% gadget hyperlink shortcuts
	\newcommand{\sgadget}{\hyperref[eq:st-gadget]{s}}
	\newcommand{\tgadget}{\hyperref[eq:st-gadget]{t}}
	\newcommand{\agadget}{\hyperref[eq:st-gadget]{a}}
	\newcommand{\bgadget}{\hyperref[eq:st-gadget]{b}}
	\newcommand{\Tn}[1]{\hyperref[eq:gadget-def]{T_{#1}}}
	\newcommand{\Rn}[1]{\hyperref[eq:effective-resistance]{R_{#1}}}

	In this section we prove Theorem \ref{theorem:integer-valued-linear-growth}. The idea is to explicitly compute the solution to the Dirichlet problem 
	on a family of {\it gadgets}, finite connected subgraphs of $\Z^d$. On these gadgets, non-constant harmonic functions are rational-valued
	with greatest common denominator growing like an exponential in the size of the graph. This implies that integer-valued harmonic functions on the gadgets have to oscillate faster than exponential, violating the regularity of harmonic functions afforded by Theorem \ref{theorem:large scale-regularity}. 
	
	In Section~\ref{subsec:gadget-def}, we identify the aforementioned gadget and prove the growth property. 
	Next, in Section \ref{subsec:linear-growth-gadget} we use this together with the uniqueness of the first-order corrector to show that integer-valued harmonic functions of linear growth do not exist. An outline of this argument appears at the beginning of the subsection. We conclude with an extension to higher order growths in Section \ref{subsec:polynomial-growth}. 
	
	\subsection{Gadgets} \label{subsec:gadget-def}
	We define the graph given in Figure \ref{fig:elec1}.
	For each $n \geq 1$, consider the following subgraph of $\Z^d$, 
	\begin{equation} \label{eq:gadget-def}
		\Tn{n} =  [1,4] \times [0,n+1],
	\end{equation}
	set
	\begin{equation} \label{eq:st-gadget}
		\sgadget := (1,1) \quad \agadget := (2,1) \quad \bgadget := (3,1) \quad \tgadget := (4,1). 
	\end{equation}
	and designate the following sites as {\it open}
	\begin{equation}
		T^{1}_n =  \{[1,4] \times \{ 1 \} \}   \cup \{[2,3] \times [1,n]\}
	\end{equation}
	and the following as {\it closed}
	\begin{equation}
		T^{0}_n =  \partial T_n  \backslash \{  \sgadget \cup \tgadget \}.
	\end{equation}
	The graph Laplacian operates on functions  $f: \Tn{n} \to \R$ via
	\begin{equation}
		\Delta_{{T_{n}}} f(x) = \sum_{|y-x| =1} ( f(y) - f(x)) \mathbf{1}_{\{ y \in T^{1}_n\}}, \quad \forall x \in \inte(T_n).
	\end{equation}
	In this subsection we prove the following. 
	\begin{prop} \label{prop:exponential-growth-gadget}
		For all $n \geq 1$, if $u: \Tn{n} \to \Z$ is harmonic 
		in $\inte(\Tn{n})$ and non-constant, then  $|u(\tgadget) - u(\sgadget)| \geq C 3^n$. 
	\end{prop}
	Our strategy for doing so is to recursively compute the {\it effective resistance} of the electric network induced by $\Tn{n}$ with unit 
	resistances along each edge. For further background on electric networks, we refer the reader to \cite{doyle-snell-electric-networks} and 
	\cite[Section 2.2]{lyons-peres-book}. In our setting, by, \eg,  \cite[Equation (2.5)]{lyons-peres-book} the effective resistance can be expressed via the solution to the following Dirichlet problem  
	\begin{equation} \label{eq:escape-probability}
		\begin{cases}
			\Delta_{{T_{n}}} h = 0 &\mbox{ on $\inte(\Tn{n})$}  \\
			h = 0 &\mbox{ on $\sgadget$} \\
			h = 1 &\mbox{ on $\tgadget$},
		\end{cases}
	\end{equation}
	as
	\begin{equation} \label{eq:effective-resistance}
		\Rn{n} := 1/h(\agadget).
	\end{equation}
	In particular, we will (implicitly) see that $\lim_{n \to \infty} \Rn{n} = (1 + \sqrt{3})$ and thus
	each  $\Rn{n}$ is a convergent of $(1 + \sqrt{3})$. We then use the fact that the terms in the continued fraction expansion of an irrational number grow exponentially. 
	
	\begin{lemma} \label{lemma:effective-resistance-recurrence}
		The effective resistance of the electric network induced by $\Tn{n}$ with unit 
		resistances along each edge, $\Rn{n}$, satisfies the following recurrence relation, 
		\begin{equation} \label{eq:effective-resistance-recurrence}
			\begin{cases}
				\Rn{1} = 3 \\
				\Rn{n+1} = (3 \Rn{n} + 2)/(\Rn{n} + 1).
			\end{cases}
		\end{equation}
	\end{lemma}
	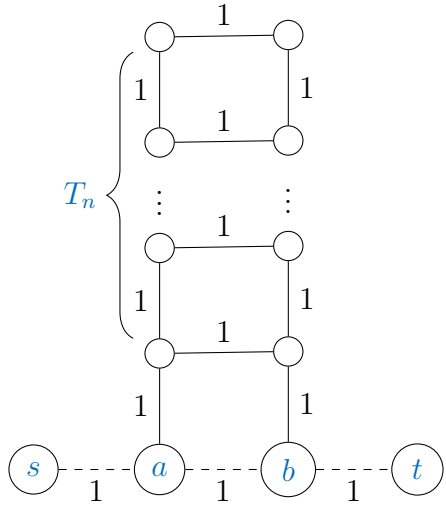
\begin{figure}
		\begin{center}
			\begin{tikzpicture} [main/.style = {draw, circle}] 
    \node[main] (s) {$\sgadget$};
    \node[main] [right=of s] (1) {$\agadget$};
    \node[main] [above=of 1] (2) {};
    \node[main] [above=of 2] (3) {};
    \node[main] [above=of 3] (4) {};
    \node[main] [above=of 4] (5) {};
    \node[main] [right=of 1] (a) {$\bgadget$};
    \node[main] [above=of a] (b) {};
    \node[main] [above=of b] (c) {};
    \node[main] [above=of c] (d) {};
    \node[main] [above=of d] (e) {};
    \node[main] [right=of a] (t) {$\tgadget$};

    \draw[dashed] (s) -- (1) node[midway, below] {$1$};
    \draw (1) -- (2) node[midway, left] {$1$};
    \draw (2) -- (3) node[midway, left] {$1$};
    \draw (4) -- (5) node[midway, left] {$1$};
    \draw[dashed] (1) -- (a) node[midway, below] {$1$};
    \draw (a) -- (b) node[midway, right] {$1$};
    \draw (b) -- (c) node[midway, right] {$1$};
    \draw (d) -- (e) node[midway, right] {$1$};
    \draw[dashed] (a) -- (t) node[midway, below] {$1$};
    \draw (2) -- (b) node[midway, above] {$1$};
    \draw (3) -- (c) node[midway, above] {$1$};
    \draw (4) -- (d) node[midway, above] {$1$};
    \draw (5) -- (e) node[midway, above] {$1$};
    \path (3) -- (4) node[midway] {$\vdots$};
    \path (c) -- (d) node[midway] {$\vdots$};

    \draw[decorate, decoration={brace, amplitude=10pt, raise=10pt}] (2) -- (5) node[midway, xshift=-30pt] {$\Tn{n}$};
\end{tikzpicture}  
		\end{center} \caption{Gadget with unit resistance along each edge. Notation given at the beginning of Section \ref{subsec:gadget-def}. }\label{fig:elec1}
	\end{figure}
	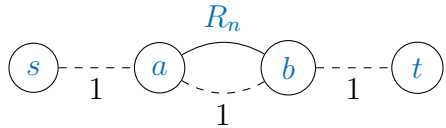
\begin{figure}
		\begin{center}
			\begin{tikzpicture} [main/.style = {draw, circle}] 
    \node[main] (s) {$\sgadget$};
    \node[main] [right=of s] (1) {$\agadget$};
    \node[main] [right=of 1] (a) {$\bgadget$};
    \node[main] [right=of a] (t) {$\tgadget$};

    \draw[dashed] (s) -- (1) node[midway, below] {$1$};
    \draw[dashed] (1) to[bend right] node[midway, below] {$1$} (a);
    \draw[dashed] (a) -- (t) node[midway, below] {$1$};
    \draw (1) to[bend left] node[midway, above] {$\Rn{n}$} (a);
\end{tikzpicture}  
		\end{center}
		\caption{Recursive decomposition of the effective resistance as in Lemma \ref{lemma:effective-resistance-recurrence}.} \label{fig:elec2}
	\end{figure}
	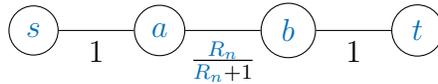
\begin{figure}
		\begin{center}
			\begin{tikzpicture} [main/.style = {draw, circle}] 
    \node[main] (s) {$\sgadget$};
    \node[main] [right=of s] (1) {$\agadget$};
    \node[main] [right=of 1] (a) {$\bgadget$};
    \node[main] [right=of a] (t) {$\tgadget$};

    \draw[] (s) -- (1) node[midway, below] {$1$};
    \draw[] (1) -- (a) node[midway, below] {$\frac{\Rn{n}}{\Rn{n}+1}$} ;
    \draw[] (a) -- (t) node[midway, below] {$1$};
\end{tikzpicture}  
		\end{center}
		\caption{Parallel law applied to the network in Figure \ref{fig:elec2}.}\label{fig:elec3}
	\end{figure}
	\begin{figure}
		\begin{center}
			\begin{tikzpicture} [main/.style = {draw, circle}] 
    \node[main] (s) {$\sgadget$};
    \node[main] [right=of a] (t) {$\tgadget$};

    \draw[] (s) -- (t) node[midway, below] {$\frac{3 \Rn{n} + 2}{\Rn{n} + 1}$};
\end{tikzpicture}  
		\end{center}
		\caption{Series law applied to the network in Figure \ref{fig:elec3}.}\label{fig:elec4}
	\end{figure}
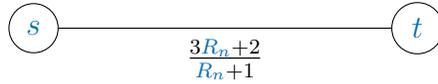

	\begin{proof}
		This is an explicit computation using the series and parallel laws for electric networks as in, \eg, \cite[Section 2.3]{lyons-peres-book}. See Figures \ref{fig:elec1}, \ref{fig:elec2}, \ref{fig:elec3}, and \ref{fig:elec4}.
	\end{proof}
	
	It is clear from \eqref{eq:effective-resistance-recurrence} that each $\Rn{n}$ is a rational number. We aim to show that the 
	numerator of the reduced rational form of $\Rn{n}$ grows exponentially. Fortunately, we may rewrite the recursion in terms of this reduced rational form.
	\begin{lemma} \label{lemma:reduced-rational-recursion}
		For each $n \geq 1$, 
		\begin{equation} \label{eq:rec-form}
			\Rn{n} = \frac{A_{n+1}}{B_n}
		\end{equation}
		where
		\begin{equation} \label{eq:A-recursion}
			\begin{cases}
				A_0 := 1 \\
				A_1 := 1 \\
				A_{n} := 4 A_{n-1} - A_{n-2}, \quad \forall n \geq 2,
			\end{cases}
		\end{equation}
		and
		\begin{equation} \label{eq:B-recursion}
			\begin{cases}
				B_0 := 0 \\
				B_1 := 1 \\
				B_{n} := 4 B_{n-1} - B_{n-2}, \quad \forall n \geq 2.
			\end{cases}
		\end{equation}
	\end{lemma}
	In order to prove Lemma \ref{lemma:reduced-rational-recursion}, we first check an identity. 
	\begin{lemma} \label{lemma:num-denom-recs}
		Let $A_n$ and $B_n$ be as in Lemma \ref{lemma:reduced-rational-recursion}. 
		We have that
		\begin{equation} \label{eq:rec-form-1}
			A_n + B_{n-1} = B_{n}, \quad \forall n \geq 1
		\end{equation}
		and
		\begin{equation} \label{eq:rec-form-2}
			3 A_n +2  B_{n-1} = A_{n+1}  \quad \forall n \geq 1.
		\end{equation}
		
	\end{lemma}
	
	\begin{proof}
		We proceed by induction on $n \geq 1$, the base cases can be verified directly. 
		
		\medskip
		{\it Inductive step for \eqref{eq:rec-form-1}.} \\
		We compute using the recursive definition and the inductive hypothesis, 
		\begin{align*}
			B_{n+1} &= 4 B_{n} - B_{n-1}  \qquad \mbox{(by \eqref{eq:B-recursion})} \\
			&=   4 ( A_{n} + B_{n-1}) - ( A_{n-1} + B_{n-2}) \qquad \mbox{(inductive hypothesis \eqref{eq:rec-form-1} for $n$ and $n-1$)} \\
			&=  (4 A_{n} - A_{n-1} ) + (4 B_{n-1} - B_{n-2})  \\ 
			&= A_{n+1} + B_{n} \qquad \mbox{(by \eqref{eq:A-recursion} and \eqref{eq:B-recursion})}.
		\end{align*}
		\medskip
		{\it Inductive step for \eqref{eq:rec-form-2}.} \\
		Similarly, 
		\begin{align*}
			&A_{n+2} \\
			&= 4 A_{n+1} - A_{n} \qquad \mbox{(by \eqref{eq:A-recursion})} \\
			&= 4 ( 3 A_n +2  B_{n-1} ) - (3 A_{n-1} +2  B_{n-2} ) \qquad \mbox{(inductive hypothesis \eqref{eq:rec-form-2} for $n$ and $n-1$)} \\
			&= 3 ( 4 A_n - A_{n-1}) + 2 ( 4 B_{n-1} - B_{n-2}) \\
			&= 3 A_{n+1} - 2 B_{n} \qquad \mbox{(by \eqref{eq:A-recursion} and \eqref{eq:B-recursion})}, 
		\end{align*}
		which completes the proof.
	\end{proof}
	\begin{proof}[Proof of Lemma \ref{lemma:reduced-rational-recursion}.]
		The base case is immediate and we check the inductive step,
		\begin{align*}
			&\Rn{n+1} \\
			&= \frac{3 \Rn{n} + 2}{\Rn{n} + 1}  \qquad \mbox{(by \eqref{eq:effective-resistance-recurrence})} \\
			&= \frac{ 3 \frac{A_{n+1}}{B_n} + 2}{ \frac{A_{n+1}}{B_n} + 1} \qquad \mbox{(inductive hypothesis \eqref{eq:rec-form} for $n$)} \\
			&= \frac{ 3 A_{n+1} + 2 B_n}{A_{n+1} + B_n} \\
			&= \frac{ A_{n+2}}{B_{n+1}} \qquad \mbox{(by Lemma \ref{lemma:num-denom-recs})}, 
		\end{align*}
		completing the proof.
	\end{proof}
	
	We then check that the expression for $\Rn{n}$ given by \eqref{eq:rec-form} is in reduced rational form. 
	\begin{lemma} \label{lemma:reduced-rational}
		$A_{n+1}$ and $B_{n}$ are coprime for all $n \geq 1$. 
	\end{lemma}
	\begin{proof}
		We first observe using \eqref{eq:rec-form-1}
		\[
		\gcd(A_{n+1}, B_n) = \gcd(A_{n+1} +  B_n, B_n) = \gcd(B_{n+1}, B_n).
		\]
		We then check, by induction, that for all $n \geq 0$, $ \gcd(B_{n+1}, B_n) = 1$. The base case is automatic, so
		\[
		\gcd(B_{n+1}, B_n) = \gcd( 4 B_{n} - B_{n-1}, B_n)  = \gcd(-B_{n-1}, B_n) = 1,
		\]
		completing the proof.
	\end{proof}
	We next observe that numerator of $\Rn{n}$ grows exponentially.
	\begin{lemma} \label{lemma:exponential-growth}
		For $n \geq 1$, 
		\[
		A_{n+1} > 3^{n-1}.
		\]
	\end{lemma}
	\begin{proof}
		The result follows by an induction. The base case can be checked directly. The inductive step is, 
		\begin{align*}
			A_{n+1} &= 4 A_{n} - A_{n-1} \\
			&= 3 A_n + (A_n - A_{n-1}) \\
			&> 3 A_n,
		\end{align*}
		completing the proof. 
	\end{proof}
	
	We indicate how the above lemmas lead to the desired claim. 
	\begin{proof}[Proof of Proposition \ref{prop:exponential-growth-gadget}]
		If $u$ is non-constant, then by subtracting a constant, we may assume $u(\sgadget) = 0$ and $u(\tgadget) \neq 0$ and therefore, by the maximum principle, 
		\[
		h =  \frac{u}{u(\tgadget)},
		\]
		where $h$ is given by \eqref{eq:escape-probability}. Hence, by Lemma \ref{lemma:reduced-rational-recursion},
		and the definition of the effective resistance $\Rn{n}$
		\[
		u(\agadget) =  \frac{B_n}{A_{n+1}} u(\tgadget).
		\]
		In particular, by Lemma \ref{lemma:reduced-rational},  since $u(\agadget)$ is an integer, this means $u(\tgadget)$ must divide $A_{n+1}$ --- this implies the claim by Lemma \ref{lemma:exponential-growth}.
	\end{proof}

	\subsection{Ruling out linear growth} \label{subsec:linear-growth-gadget}
	\begin{figure}
		\includegraphics[width=0.5\textwidth]{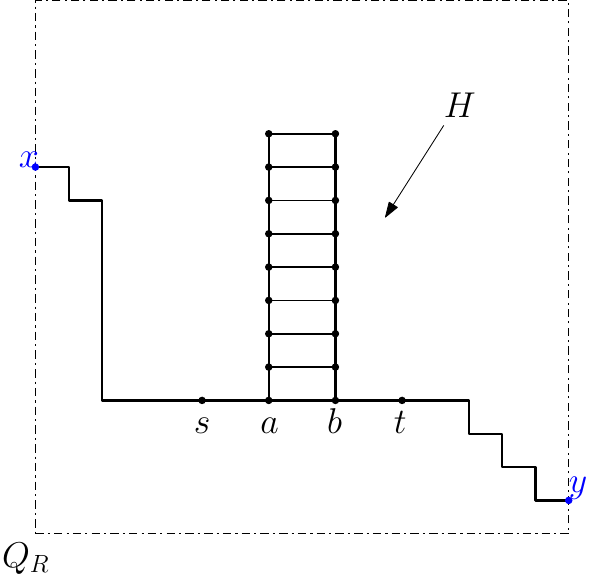}
		\caption{Attaching the gadget, $\Tn{R}$, as in the proof of Lemma \ref{lemma:non-constant-on-gadget}. The cube of radius $R$, $Q_R$,  is outlined by a dashed-dotted line and the subgraph $H$ containing the gadget is in black.} \label{fig:attach-gadget}
	\end{figure}
	As we see at the end of this subsection, by ergodicity, it will suffice to prove Theorem \ref{theorem:integer-valued-linear-growth} for a single slope. 
	\begin{lemma} \label{lemma:det-slope}
		Fix a deterministic, nonzero slope ${p \in \R^d}$.  Almost surely, there exists an edge $e \in \mathscr{C}_\infty$ such that
		\begin{equation*}
			\nabla \ellp{p}(e) \notin\Z.
		\end{equation*}
	\end{lemma}

	The idea of the proof is as follows. By Proposition \ref{prop:exponential-growth-gadget}, any integer-valued harmonic function $u$
	which is non-constant along a gadget $\Tn{R}$ must oscillate exponentially within $\Tn{R}$. Thus, roughly, if with relatively decent probability there is a gadget near the origin, and $u$ is not constant on it, we contradict the exponential oscillation bound guaranteed by Theorem \ref{theorem:large scale-regularity}. However, there is no way of ensuring, a priori, that an integer-valued harmonic function is not constant on every gadget.  Ruling this out is most of the proof of Lemma \ref{lemma:det-slope}. 
	\begin{lemma} \label{lemma:non-constant-on-gadget}
		Fix a deterministic, nonzero slope $p \in \R^d$. There exists $R_0(p)$ such that 
		for all  $R \geq R_0$, there is a deterministic subgraph $H \subset Q_R$ containing 
		a translation of the gadget $\Tn{R}$ (as an induced subgraph), such that,
		\begin{equation*}
			\P \left[ \ellp{p} ~\mbox{is not constant on} ~ \Tn{R} \, | \, \mathscr{C}_{\infty} \cap Q_R = H \right] \geq c R^{-2d+2}.
		\end{equation*}
	\end{lemma}

	\begin{proof}[Proof of Lemma \ref{lemma:det-slope} assuming Lemma \ref{lemma:non-constant-on-gadget}]
		
		First observe that the event that there exists an integer-valued harmonic function which grows like $p$ at infinity is translation invariant. Therefore, by ergodicity of Bernoulli percolation, \eg, \cite[Proposition 7.3]{lyons-peres-book}, the event occurs with probability zero or one. Suppose for sake of contradiction that it occurs with probability $1$.

		By Lemma \ref{lemma:non-constant-on-gadget} and Proposition \ref{prop:exponential-growth-gadget}, we have the lower bound
		\begin{equation*}
			\P \left[ \osc_{\mathscr{C}_\infty \cap Q_R}  \ellp{p}  \geq c 3^{R} \, \big| \, \mathscr{C}_\infty \cap Q_R = H \right] \geq c R^{-2d+2}.
		\end{equation*}
		Using that the probability of the event $\mathscr{C}_\infty \cap Q_R = H$ is lower bounded by $\min( \mathfrak{p}, 1 - \mathfrak{p})^{\left| Q_R \right|}$, we deduce that
		\begin{equation} \label{eq:prob-lower-bound}
			\P \left[ \osc_{\mathscr{C}_\infty \cap Q_R} \ellp{p}  \geq c 3^{R}  \right] \geq c R^{-2d+2} \min( \mathfrak{p}, 1 - \mathfrak{p})^{\left|Q_R\right|}  \geq c \exp \left( - C R^d \right).
		\end{equation}
		We conclude by observing that \eqref{eq:prob-lower-bound} contradicts Theorem \ref{theorem:first-order-corrector} which states that 
		\[
		\P \left[ \osc_{Q_N} \ellp{p} \geq |p| N^{2} \right] \leq 2 \exp(-c N^s).
		\] 
		Indeed, choose a constant $C'(\mathfrak{p},d) < \infty$ so that, by \eqref{eq:prob-lower-bound} with $R = C' \log N$
		\[
		\P \left[ \osc_{Q_N} \ellp{p}   \geq 2 |p| N^{2} \right] \geq 	\P\left[\osc_{Q_R} \ellp{p}  \geq 2 |p| N^{2} \right] \geq  \exp( - (C' \log N)^d),
		\]	
		and note that $2 \exp(-c N^s) \leq \exp( - (C' \log N)^d)$ for $N$ sufficiently large.
	\end{proof}

	It remains to prove Lemma \ref{lemma:non-constant-on-gadget}.

	\begin{proof}[Proof of Lemma \ref{lemma:non-constant-on-gadget}]
		We first restrict to a high probability event. Let $E_R(\mathscr{C}_{\infty})$ denote the event that the following occurs: 
		\begin{itemize}
			\item The cube $Q_R$ is well-connected as in Proposition \ref{prop:well-connected};
			\item Every non-constant function in $\mathcal{A}_1(\mathscr{C}_{\infty})$ takes on more than one value in $Q_R \cap \mathscr{C}_{\infty}$.  
		\end{itemize}
		By Proposition~\ref{prop:well-connected} the first property of $E_R$ occurs with probability approaching one as $R \to \infty$. By Theorem \ref{theorem:large scale-regularity}, the second property also occurs with probability approaching one as $R \to \infty$. Thus, we may fix a deterministic $R \geq 10^6$ large enough so that $\P[E_R] \geq 1/2$. 
		
		Restrict to the event that $E_R$ occurs and let $\ellp{p}$ be the corrected plane in $\mathscr{C}_{\infty}$.  Consider the infinite cluster $\mathscr{\tilde C}_{\infty}$ in the environment where every edge which contains a vertex in $\partial Q_R$ is closed
		and let $\tilde{\ellp{p}}$ be the corrected plane in~$\mathscr{\tilde C}_{\infty}$. As the event $E_R$ occurs, $\mathscr{\tilde C}_{\infty} \subset \mathscr{C}_{\infty}$ and each face of $Q_{R+1}$ contains a site in $\mathscr{\tilde C}_{\infty}$. Further, since $\ellp{p}$ is not constant in $Q_R$, $\tilde{\ellp{p}} $ is not constant on $Q_{R+1}$. Indeed, if this were the case, 
		we could extend $\tilde{\ellp{p}}$ by a constant in $Q_{R} \cap \mathscr{C}_{\infty}$ and considering $(\ellp{p} -\tilde{\ellp{p}})$ would give rise to a non-constant, sublinear harmonic function on $\mathscr{C}_{\infty}$, contradicting 
		Theorem \ref{theorem:large scale-regularity}. 
		
		Let $\tilde x$ and $\tilde y$ be sites on opposite faces of $Q_{R+1} \cap \mathscr{\tilde C}_{\infty}$ for which $\tilde{\ellp{p}} (\tilde x) \neq \tilde{\ellp{p}}(\tilde y)$ and let $x$ and $y$ be sites  in $Q_R \cap \mathscr{C}_{\infty}$ which, respectively, share an edge with $\tilde x$ and $\tilde y$. Let $H$ be the subgraph of $Q_R$ containing $\Tn{R} \cup \{x, y\}$  given by Figure \ref{fig:attach-gadget}. 
		Specifically, let $a$ and $b$ be two adjacent sites on $\partial Q_{R/2}$. Let $P_1$ and $P_2$ be paths in $\Z^d$ from $x$ and $y$ to $a$ and $y$ to $b$ respectively which do not intersect $Q_{R/2}$. 
		Let $H = P_1 \cup P_2$ together with a subgraph of $Q_{R/2} \cup \{s, a, b, t\}$ which is isometric to $\Tn{R}$. 
		
		Let ${\ellp{p}}'$ be the corrected plane defined on $\mathscr{C}_{\infty}' := \mathscr{\tilde C}_{\infty} \cup H$. We claim that ${\ellp{p}}'(x) \neq  {\ellp{p}}'(y)$. Indeed, if not, then ${\ellp{p}}'$ restricted to $\mathscr{\tilde C}_{\infty}$ would be harmonic and the difference $({\ellp{p}}' - \tilde{\ellp{p}})$ would give rise to a non-constant sublinear harmonic function, again contradicting Theorem \ref{theorem:large scale-regularity}. 
		
		By taking a union bound over all possible choices of boundary placements $x$ and $y$, we have the desired claim. 
	\end{proof}

	We use Lemma \ref{lemma:det-slope} together with an ergodicity argument to prove Theorem~\ref{theorem:integer-valued-linear-growth}. 
	\begin{lemma} \label{lemma:ergodicity-lipschitz}
		If the event that there exists an integer-valued harmonic function of linear growth on the cluster 
		has positive probability, then there exists a deterministic slope $\underline{p} \in \R^d$
		such that, on an event of probability one, there is an integer-valued harmonic function
		of linear growth $\underline{p}$ on the cluster. 
	\end{lemma}
	\begin{proof}[Proof of Theorem~\ref{theorem:integer-valued-linear-growth} assuming Lemma~\ref{lemma:ergodicity-lipschitz}]
		By Theorem \ref{theorem:first-order-corrector}, the conclusion of Lemma \ref{lemma:ergodicity-lipschitz} is incompatible with Lemma \ref{lemma:det-slope} applied with slope $\underline{p}$. 
	\end{proof}
	\begin{proof}[Proof of Lemma \ref{lemma:ergodicity-lipschitz}]
		The event that there exists some integer-valued harmonic function of linear growth on $\mathscr{C}_{\infty}$ is translation invariant and hence occurs with probability zero or one by ergodicity of Bernoulli percolation. Restrict to the event of probability one that it occurs. We first show, using compactness, that the set of integer-valued harmonic functions of linear growth is closed. 
		
		Fix a realization of the infinite cluster and consider a sequence of slopes $p_n \to \underline{p}$ and observe that, there exists an integer-valued harmonic function with asymptotic growth $p_n \in \R^d$ if and only if, for any edge $e \in E \left( \mathscr{C}_\infty \right)$, $\nabla \ellp{p_n}(e)$ is an integer. Since, for any edge $e \in E \left( \mathscr{C}_\infty \right)$, the map $p \mapsto \nabla \ellp{p}(e)$ is linear, we have that $\nabla \ellp{p_n}(e) \to \nabla \ellp{\underline{p}}(e)$ as $p_n \to \underline{p}$. A combination of the two previous observations shows that, for any edge $e \in E\left( \mathscr{C}_\infty \right)$, $\nabla \ellp{\underline{p}}(e)$ is an integer which implies that there exists an integer-valued harmonic function with asymptotic growth $\underline{p} \in \R^d$, and completes the proof of closure. 
		
		Now, consider the random variable, $\underline{p}$ defined as
		\[
		\begin{aligned}
			\underline{p} = \arg \inf &\{  |p| \geq 1 \mid \\
			& \mbox{ there exists an integer-valued
				harmonic function of growth $p$ at infinity} \},
		\end{aligned}
		\]
		where the $\arg \inf$ indicates a choice of slope $p$ such that $|p| \geq 1$ is minimized, 
		and ties are broken by choosing the lexicographically smallest such $p$ (the existence of this slope is guaranteed by the closure of the set of slopes $p \in \Rd$ such that there exists an integer-valued harmonic function with growth $p$). Note that if an integer-valued harmonic function exists, then by multiplying
		by a sufficiently large integer, we may produce an integer-valued harmonic function with asymptotic growth $|p| \geq 1$. 
		
		The random variable $\underline{p}$ is translation invariant and hence, by ergodicity again, is deterministic.
		Since the set of integer-valued harmonic functions of linear growth is closed, this implies the existence of an integer-valued harmonic function 
		on the cluster growing like $\underline{p}$ at infinity.
	\end{proof}

	\subsection{Polynomial growth}  \label{subsec:polynomial-growth}
	In this section we indicate some progress towards extending Theorem \ref{theorem:integer-valued-linear-growth} to all polynomial growths. The below is stated for quadratic polynomials, 
	but the statement and its proof can be extended to arbitrary polynomial growths.

	\begin{prop} \label{prop:rational-matrix}
		For each deterministic harmonic polynomial $\overline{p} \in \overline{\mathcal{A}}_2$, almost surely, there is no nonconstant integer-valued harmonic function $u$ such that 
		\[
		\lim_{|x| \to \infty} \frac{1}{|x|}|u(x) - \overline{p}(x)| = 0.
		\]
	\end{prop}
	\begin{proof}
		Denote by $E$ the event that there exists an integer-valued harmonic function $u$ with 
		$\lim_{|x| \to \infty} \frac{1}{|x|}|u(x) - \overline{p}(x)| = 0$
		and suppose
		\[
		\overline{p}(x) = x^T  M x + b^T x.
		\]
		Suppose for sake of contradiction that $\P[E] > 0$. 
		By the ergodic theorem, we have that almost surely, 
		\begin{equation} \label{eq:density-quadratics}
			\lim_{N \to \infty} \frac{1}{|B_N|} \sum_{x \in B_N} 1\{ \mbox{$E$ occurs in the environment translated by $x$} \} = \P[E] > 0.
		\end{equation}
		In particular, almost surely, there is some random $N$, two sites $y^{\pm} \in B_N$, and two integer-valued harmonic functions $u^{\pm}$
		such that, 
		\[
		p^{\pm}(x) =  x^T M x + (2 M y^{\pm} + b)^T x,
		\] 
		and
		\[
		\lim_{|x| \to \infty} \frac{1}{|x|}|u^{\pm}(x) - \overline{p}^{\pm}(x)| = 0.
		\]
		Note that, by \eqref{eq:density-quadratics}, we may assume that $(y^{+} - y^{-})$ is not in the kernel of $M$ (if $M =0$, the result is a direct consequence of Theorem \ref{theorem:integer-valued-linear-growth}). Thus, setting $v: = u^+ - u^-$ we have 
		constructed a non-trivial integer-valued harmonic function of linear growth, 
		\[
		\lim_{|x| \to \infty} \frac{1}{|x|} |v(x) - 2 M (y^{+} - y^-)^T x | = 0,
		\]
		contradicting Theorem \ref{theorem:integer-valued-linear-growth}.  
	\end{proof}

	\section{Integer-valued Laplacian} \label{sec:toppling-invariants}
	In this section, we fix the dimension to be $d = 2$ and prove the following theorem. 
	\begin{theorem} \label{theorem:fast-decay-stronger}
		Almost surely, any function $u: \mathscr{C}_{\infty} \to \R$ which decays to 0,  $\lim_{|x| \to \infty} u(x) = 0$,
		and has integer-valued Laplacian,  $\Delta_{\mathscr{C}_{\infty}} u \in \Z$, satisfies the following dichotomy:
		\begin{itemize}
			\item Either the function $u$ decays at most like $|x|^{-1}$ at infinity, \ie,
			\begin{equation} \label{eq:linear-decay}
				\limsup_{|x| \to \infty} \frac{|u(x)|}{|x|} > 0;
			\end{equation}
			\item Or $u$ is finitely supported, \ie, 
			\begin{equation} \label{eq:compactly-supported}
				u(x) \equiv 0, \quad \mbox{for all $x \in \mathscr{C}_{\infty} \setminus B_R$ for some $R > 0$}. 
			\end{equation}
		\end{itemize}
	\end{theorem}
	
	\begin{remark}
		This result fails on $\Z^2$ as the function $u = G^{\mathbb{Z}^2}(\cdot - e_1) + G^{\mathbb{Z}^2}(\cdot + e_1) - 2 G^{\mathbb{Z}^2}$, where $G^{\mathbb{Z}^2}$ is the discrete Green's function on the lattice, is not finitely supported and decays like $|x|^{-2}$ -- see \cite{sandpiles-square-lattice} or \cite{schmidt-verbitskiy-2009}.
	\end{remark}

	The proof of Theorem~\ref{theorem:fast-decay-stronger} proceeds by successive reduction.
	
	\subsection{Reduction to mean zero deterministic pole functions}
	
	We first reduce the proof of Theorem~\ref{theorem:fast-decay-stronger} to deterministic pole functions $f$ on the event $\sum_{x \in \mathscr{C}_\infty} f(x) = 0$ using the two following propositions.
	\begin{prop} \label{prop:fast-decay-stronger}
		Fix a deterministic and compactly supported function $f : \Z^2 \to \Z$ and let 
		\begin{equation} \label{eq:green-potential}
			u_f(x) := \sum_{y \in \mathscr{C}_\infty} f(y) G(x , y),
		\end{equation}
		where $G$ is the elliptic Green's function on $\mathscr{C}_{\infty}$. 
		Then, almost surely, either $\sum_{y \in \mathscr{C}_\infty} f(y) \neq 0$, or it satisfies~\eqref{eq:linear-decay}, or it satisfies~\eqref{eq:compactly-supported}. 
	\end{prop}
	
	\begin{remark}
		On the event of positive probability where $\supp f \cap \mathscr{C}_\infty = \emptyset$, we have that $u_f \equiv 0$.
	\end{remark}
	
	\begin{prop} \label{prop:not-mean-zero-log-growth}
		Almost surely, if $\sum_{x \in \mathscr{C}_{\infty}} f(x) \neq 0$, then $u_f = \Theta(\log|x|)$.
	\end{prop}

	\begin{proof}[Proof of Theorem~\ref{theorem:fast-decay-stronger} assuming Proposition~\ref{prop:fast-decay-stronger} and Proposition~\ref{prop:not-mean-zero-log-growth}]
		We start by observing that the set of integer-valued and finitely supported functions $\mathscr{F}$ is countable, 
		thus we may restrict to the event of probability one where Proposition~\ref{prop:fast-decay-stronger} holds for all $f \in \mathscr{F}$. 
		
		As $\lim_{|x| \to \infty} u_f(x) = 0$ and $-\Delta_{\mathscr{C}_{\infty}} u_f = f \in \Z$, we must have that $f$ is compactly supported. Consequently, 
		by extending $f$ to be zero outside of $\mathscr{C}_{\infty}$, we have that $f \in \mathscr{F}$. This completes the proof by the assumption that $u_f$ decays to zero at infinity. 
	\end{proof}

	In the remainder of this subsection, we prove Proposition \ref{prop:not-mean-zero-log-growth}. 
	In subsequent subsections we further reduce the proof of Proposition~\ref{prop:fast-decay-stronger}
	and then prove its reduction. 
	
	For the remainder of the proof, we fix a function $f: \Z^2 \to \Z$ as in the statement of Proposition~\ref{prop:fast-decay-stronger} and a function $u_f$ as in \eqref{eq:green-potential}. Using the terminology of potential theory, we refer to the function $u_f$ as the {\it potential}. We additionally allow all constants and exponents to depend on the function $f$ (and the probability $\mathfrak{p}$).
	
	We first observe that under the assumption $\sum_{x \in \mathscr{C}_\infty} f(x) = 0$, the function $f$ can be written as the divergence of a compactly supported vector field on the infinite cluster. This then allows us to represent $u_f$ as a linear combination of gradients of the Green's function. 
	
	\begin{lemma} \label{lemma:mean-zero-representation}
		For almost every realization of the infinite percolation cluster satisfying the condition $\sum_{x \in \mathscr{C}_{\infty}} f(x) = 0$, there exists a random compactly supported vector field $F : E_d \left( \mathscr{C}_\infty \right) \to \R$
		such that
		\begin{equation} \label{eq:gradient-representation}
			u_f (x) = \sum_{y \in \mathscr{C}_\infty} f(y) G(x , y)  = \sum_{e \subseteq \mathscr{C}_\infty} F(e) \nabla G(x, e).
		\end{equation}
		The vector field $F$ can be chosen so that there exist $C_f(\mathfrak{p} , f) < \infty$ and $s(\mathfrak{p} ) > 0 $ such that
		\begin{equation} \label{eq:estimates-on-F}
			\left| \supp F\right| \leq \mathcal{O}_s (C_f) ~~\mbox{and} ~~ \left\| F \right\|_{L^2(\mathscr{C}_\infty)}  \leq \mathcal{O}_s (C_f).
		\end{equation}
	\end{lemma}

	\begin{proof}
		Fix a realization of the percolation cluster and let $\cu_f$ be the smallest box centered at $0$ such that $\mathscr{C}_*(\cu_f)$ contains $\supp f \cap \mathscr{C}_\infty$, $\mathscr{C}_*(\cu_f) \subset \mathscr{C}_{\infty}$ and such that the Poincar\'e inequality applies in $\mathscr{C}_*(\cu_f)$ in the following form: there exists a constant $C_{\mathrm{Poinc}}( \mathfrak{p}) < \infty$ such that, for any function $v : \mathscr{C}_*(\cu_f) \to \R$ satisfying $\sum_{x \in \mathscr{C}_*(\cu_f)} v(x) = 0$,
		\begin{equation*}
			\left\| v \right\|_{L^2 ( \mathscr{C}_*(\cu_f) )} \leq C_{\mathrm{Poinc}} \size(\cu_f) \left\| \nabla v \right\|_{L^2 ( \mathscr{C}_*(\cu_f))}
		\end{equation*}
		Such a box exists almost surely and its size satisfies the stochastic integrability estimate
		\begin{equation} \label{eq:bound-on-size}
			\size(\cu_f)  \leq \mathcal{O}_s(C_f). 
		\end{equation}
		The Neumann problem (with the normalizing condition $\sum_{x \in \mathscr{C}_*(\cu_f)} v(x) = 0$)
		\begin{equation} \label{eq:Neumannprob}
			\left\{ 
			\begin{aligned}
				- \Delta_{\mathscr{C}_\infty} v = f &~\mbox{in}~ \mathscr{C}_* (\cu_f), \\
				\mathbf{n} \cdot \nabla v = 0 &~\mbox{on} ~ \partial_e \mathscr{C}_*(\cu_f),
			\end{aligned}
			\right.
		\end{equation}
		is well-posed because the set $\mathscr{C}_*(\cu_f)$ is connected and $\sum_{x \in \mathscr{C}_*(\cu_f)} f(x) = 0$. We then define $F(e) := \nabla v(e)$ if $e \in E \left(\mathscr{C}_*(\cu_f)\right)$ and $F(e) = 0$ if $e \in E \left( \mathscr{C}_\infty \right)  \setminus E \left(\mathscr{C}_*(\cu_f)\right)$. It is a consequence of the definition of $F$ that $\nabla \cdot F = f$ in $\mathscr{C}_\infty$
		and the representation \eqref{eq:gradient-representation} follows by the discrete divergence theorem. 
		
		We finally prove the estimates~\eqref{eq:estimates-on-F}. The first one is a consequence of~\eqref{eq:bound-on-size} and the observation $\supp F \subseteq \cu_f$. The second one follows from testing the function $v$ in the Neumann problem~\eqref{eq:Neumannprob} and applying the Cauchy-Schwarz and Poincar\'e inequalities. We obtain
		\begin{align*}
			\left\| \nabla v \right\|_{L^2 ( \mathscr{C}_*(\cu_f))}^2  = \sum_{x \in \mathscr{C}_*(\cu_f)} f(x) v(x) 
			& \leq \left\| f \right\|_{L^2(\mathscr{C}_\infty)} \left\| v \right\|_{L^2(\mathscr{C}_\infty)} \\
			& \leq  C_{\mathrm{Poinc}} \size(\cu_f) \left\| f \right\|_{L^2(\mathscr{C}_\infty)} \left\| \nabla v \right\|_{L^2 ( \mathscr{C}_*(\cu_f))},
		\end{align*}
		which implies, using the identity $F = \nabla v$ inside the box $\cu_f$ and that $F$ is equal to $0$ outside this box,
		\begin{equation*}
			\left\| F \right\|_{L^2 ( \mathscr{C}_\infty)} \leq C_{\mathrm{Poinc}} \left\| f \right\|_{L^2(\mathscr{C}_\infty)}\size(\cu_f).
		\end{equation*}
		The second estimate of~\eqref{eq:estimates-on-F} is then a consequence of the bound~\eqref{eq:bound-on-size}.
	\end{proof}
	
	We are now ready to complete the proof of Proposition~\ref{prop:not-mean-zero-log-growth} using Lemma~\ref{lemma:mean-zero-representation}.
	
	\begin{proof}[Proof of Proposition~\ref{prop:not-mean-zero-log-growth}]
		As recalled above in \eqref{eq:nablayG}, there exists a random variable 
		\[
		\mathcal{M}_{\nabla\mathrm{-Decay}}(y)
		\]
		such that, if $y \in \mathscr{C}_\infty$ then for any $x \in \mathscr{C}_\infty$ with 
		\[
		|x - y| \geq \mathcal{M}_{\nabla\mathrm{-Decay}}(y),
		\]
		we have 
		\begin{equation} \label{eq:green-decay-estimate}
			\left| \nabla_y G(x , y) \right| \leq \frac{C \mathcal{M}_{\nabla\mathrm{-Decay}}(y)^{\frac{d}{2}}}{|x - y|}.
		\end{equation}
		Pick a site $z \in \mathscr{C}_{\infty}$, let $K := \sum_{x \in \mathscr{C}_{\infty}} f(x)$ so that $g := f -  K \delta_z$ has mean zero. Thus, letting $u_g$ be defined by \eqref{eq:green-potential}, we have by  \eqref{eq:gradient-representation} and \eqref{eq:green-decay-estimate} that $u_g = O(|x|^{-1})$. By definition, we have that $u_f = u_g + u'$, where 
		\[
		u' :=  K  G(\cdot , z),
		\]
		and by \cite[Theorem 2]{dario-gu-green}, $u' = \Theta(\log |x|)$, completing the proof. 
	\end{proof}
	
	\subsection{Reduction to a statement on the corrected plane}
	In light of the previous proposition, we may restrict our attention to the event that $\sum_{x \in \mathscr{C}_{\infty}} f(x) = 0$.
	We then show that Proposition~\ref{prop:fast-decay-stronger} can be obtained as a consequence of the two following propositions.
	
	The first one is a result related to the behavior of the corrected plane, its proof is the core of the argument and occupies the rest of Section~\ref{sec:toppling-invariants}. 
	
	\begin{prop} \label{prop:gradient-trichotomy}
		For any fixed slope $p \in \R^2$, almost surely on the event $\sum_{x \in \mathscr{C}_\infty} f(x) = 0$, 
		
		\begin{equation} \label{eq:gradient-growth-bound}
			\mbox{either}~ \sum_{y \in \mathscr{C}_\infty} f(y) \ellp{p}(y) \neq 0 ~\mbox{or}~ |\supp u_f| < \infty.
		\end{equation}
	\end{prop}
	
	The second proposition is a stochastic homogenization result which identifies the first order-term of the asymptotic behavior of the function $u_f$ in terms of the corrected plane. Its proof builds upon the results collected and proved in the prior subsection, and is presented in the rest of this subsection.

	\begin{prop} \label{prop:two-scale-expansion-mean-zero}
		There exist a constant $C := C ( f, \mathfrak{p}) < \infty$, an exponent $s := s(f, \mathfrak{p}) < \infty$ and a minimal scale $\mathcal{M}_f$ satisfying the stochastic integrability estimate
		\begin{equation} \label{eq:stochastic-integrability-bound-two-scale-expansion}
			\mathcal{M}_f \leq \mathcal{O}_s (C) 
		\end{equation}
		such that, if we consider a realization of the infinite cluster $\mathscr{C}_\infty$ such that $\sum_{x \in \mathscr{C}_{\infty}} f(x) = 0 $, then, for any $x \in \mathscr{C}_\infty$ satisfying $|x| \geq \mathcal{M}_f$, 
		\begin{equation} \label{eq:two-scale-expansion-mean-zero}
			\left| u_f (x) -  \sum_{i = 1}^d \left(\sum_{y \in \mathscr{C}_\infty} f(y) \ellp{e_i}(y) \right) \partial_i \bar G(x) \right| \leq C |x|^{-\frac{5}{4}},
		\end{equation}
		where $\bar G$ is a multiple of the elliptic Green's function on $\R^2$. 
	\end{prop} 
	
	\begin{remark}
		We note that in the right hand side, the exponent $5/4$ is strictly larger than $1$ and that the two terms in the left-hand side (considered individually) decay like $1/|x|$. Their difference is thus smaller than their typical size. 
	\end{remark}
	
	\begin{remark}
		Although in this section, we set $d = 2$, 
		Proposition~\ref{prop:two-scale-expansion-mean-zero} is valid for any $d \geq 3$ by replacing $2$ by $d - \frac{3}{4}$ in the right-hand side of~\eqref{eq:two-scale-expansion-mean-zero}. 
		Moreover, the exponent $\frac{3}{4}$ is not optimized but is sufficient for our purposes. 
	\end{remark}

	We first show how to complete the proof of Proposition~\ref{prop:fast-decay-stronger} (and thus of Theorem~\ref{theorem:slow-mixing-percolation}).
	
	\begin{proof}[Proof of Proposition~\ref{prop:fast-decay-stronger} assuming Proposition~\ref{prop:gradient-trichotomy} and Proposition~\ref{prop:two-scale-expansion-mean-zero}]
		By Proposition~\ref{prop:not-mean-zero-log-growth}, we may restrict to the event that $\sum_{x \in \mathscr{C}_\infty} f(x) = 0$.
		Let $p \in \R^2$ be given by Proposition~\ref{prop:gradient-trichotomy} and define the ray 
		\[
		\mathcal{R}_p := \{ R p : R \in \R \}
		\]
		so that for all $x \in \Z^2$ which lie on the ray (or are at distance smaller than $1$ from it), 
		\[
		\sum_{i = 1}^2 \left(  \sum_{y \in \mathscr{C}_\infty} f(y) \ellp{e_i}(y) \right)  \partial_i \bar G(x) = - c |x|^{-1} \left( \sum_{y \in \mathscr{C}_\infty} f(y) \ellp{p}(y) \right)  + o(|x|^{-1}),
		\]
		where $c \neq 0$ is a constant depending on $\mathfrak{p}$ (involving the density and the diffusivity of the infinite cluster, see, \eg, \cite[equation (1.9)]{BH09} or \cite[equation (1.9)]{dario-gu-green}). We thus have by Proposition~\ref{prop:two-scale-expansion-mean-zero}
		that
		\begin{equation*}
			u_f (x) =  c \left(\sum_{y \in \mathscr{C}_\infty} f(y) \ \ellp{p}(y)  \right)  |x|^{-1}  + o(|x|^{-1}), \quad \mbox{for all $x \in \Z^2$ such that $\dist( x, \mathcal{R}_p) \leq 1$ }
		\end{equation*}
		which implies the claim by Proposition~\ref{prop:gradient-trichotomy}.  
	\end{proof}
	
	\begin{proof}[Proof of Proposition~\ref{prop:two-scale-expansion-mean-zero}]
		By Lemma \ref{lemma:mean-zero-representation}, we have the identity 
		\begin{equation*}
			u_f (x) = \sum_{y \in \mathscr{C}_\infty} f(y) G(x , y)  = \sum_{e \subseteq \mathscr{C}_\infty} F(e) \nabla_y G(x, e),
		\end{equation*}
		for a compactly supported vector field $F : E \left( \mathscr{C}_\infty \right) \to \R$ satisfying the estimates~\eqref{eq:estimates-on-F}. Denote by $\bar G_x := \bar G(\cdot - x)$ and define the minimal scale $\mathcal{M}_f$ according to the formula
		\begin{equation*}
			\mathcal{M}_f := \left( \mathcal{M}_{\nabla-\mathrm{Homog}, \frac{1}{8}}(0) \vee \size(\cu_f) \vee \mathcal{M}_{\mathrm{corr}, \frac12}(0) \vee  \left\| F \right\|_{L^2 ( \mathscr{C}_\infty)}\right)^{16(2+2)}
		\end{equation*}
		where $\cu_f$ is as in Lemma \ref{lemma:mean-zero-representation}. Applying Proposition~\ref{prop:homogenizationellipticgreen} with the observation that $F$ is supported in the box $\cu_f$, we deduce that
		\begin{equation*}
			\left| u_f(x) - \sum_{i = 1}^2 \sum_{\substack{e \subseteq E(\mathscr{C}_\infty) \\ e = (y , y+e_i)}} F(e) \nabla \ellp{e_i}(e) \nabla \bar G_x(e) \right| \leq \frac{C}{|x|^{\frac 54}}.
		\end{equation*}Since the elliptic Green's function on $\R^2$ is smooth away from its pole, we have the inequality, for any $y \in \Z^2$ with $|y - x| \leq |x|/2$ and any $i \in \{ 1 , 2\}$,
		\begin{equation} \label{eq:gradient-green-taylor}
			\left| \nabla \bar G_x (\{y , y+e_i\}) - \partial_i \bar G(x)  \right| \leq \frac{C \left|y\right|}{|x|^2}.
		\end{equation}
		From the inequality~\eqref{eq:gradient-green-taylor}, we deduce that, for any $i \in \{1, 2\}$ and any $x \in \mathscr{C}_\infty$ satisfying $|x| \geq \mathcal{M}_f$,
		\begin{align*}
			\left| \sum_{\substack{e \subseteq E(\mathscr{C}_\infty) \\ e = (y , y+e_i)}} F(e) \nabla \ellp{e_i}(e) \left( \nabla \bar G_x(e) - \partial_i \bar G(x) \right) \right| & \leq \frac{C \size(\cu_f)^{2} \left\| F \right\|_{L^2 ( \mathscr{C}_\infty)} \sup_{e \in \supp F} \left| \nabla \ellp{e_i}(e) \right|}{|x|^2} \\
			& \leq \frac{C}{|x|^{\frac{5}{4}}}.
		\end{align*}
		The exponent $2 = (d/2+1)$ in $\size(\cu_f)^{2}$ in the first inequality is due to the observation that the set of edges in the support of $F$ has cardinality smaller than $C \left| \cu_f\right| = C  \size(\cu_f)^2$ and that the diameter of the support of $F$ is smaller than $C \size(\cu_f)$ (as the right-hand side of~\eqref{eq:gradient-green-taylor} depends on the parameter $|y|$), and the Cauchy-Schwarz inequality. The second inequality uses the definition of the minimal scale. 
		
		Performing a discrete integration by parts and using the identity $\nabla \cdot F = f$ completes the proof of Proposition~\ref{prop:two-scale-expansion-mean-zero}.
	\end{proof}

	\subsection{Further reduction to a statement on the level sets of harmonic functions}
	In this subsection we reduce, via a martingale sensitivity argument,  the proof of Proposition~\ref{prop:gradient-trichotomy} to the following statement.
	\begin{prop} \label{prop:level-set-of-harmonic-function}
		Almost surely, on the event that $\sum_{x \in \mathscr{C}_{\infty}} f(x) = 0$ and $|\supp u_f| = \infty$, we have that 
		\begin{equation} \label{eq:level-set-of-harmonic-function}
			\limsup_{k \to \infty}  \frac{1}{|B_{2^k}|} \left| \left\{ e \in E(B_{2^k}) : \nabla u_f(e) \neq 0 \mbox{ and }  \nabla \ellp{p}(e) \neq 0  \right\} \right|  > 0.
		\end{equation}
	\end{prop}
	\begin{proof}[Proof of Proposition~\ref{prop:gradient-trichotomy} assuming Proposition~\ref{prop:level-set-of-harmonic-function}]
		Fix $p \in \Rd$ and denote by $E$ the event
		\begin{equation*}
			E := \{ \sum_{x \in \mathscr{C}_\infty} f(x) = 0 \} \cap \{ \sum_{y \in \mathscr{C}_\infty} f(y) \ellp{p}(y)  = 0 \}  \cap \left\{ |\supp \, u_f| = \infty \right\};
		\end{equation*}
		we will show that $E$ has probability zero. We first prove, using an argument similar to the proof of the Efron-Stein inequality, that the occurrence of the event $E$ does not depend too much on the value of a single edge. We then show, using Proposition~\ref{prop:level-set-of-harmonic-function}, that the event $E$ is in fact sensitive. In the third step we use these two deductions to show that $E$ has probability zero.  
		
		\medskip 
		
		{\it Step 1: The event $E$ cannot not depend too much on the value of a single edge.} \\
		Order the edges of the lattice following a deterministic procedure and denote the set of ordered edges by $( e_i )_{i \in \N} $. For each $n \in \N$, introduce the sigma-algebra
		\begin{equation*}
			\mathcal{F}_n := \sigma \left( \mathbf{a}(e_i) \, : \, i \in \{ 1 , \ldots, n \} \right),
		\end{equation*}
		that is, $\mathcal{F}_n$ contains the information of the first $n$ edges (where $\mathbf{a}(e_i) = 0$ or $1$ tells us whether or not the edge is open or closed). Then define the martingale
		\begin{equation*}
			M_n := \E \left[ \indc_{E} | \mathcal{F}_n \right],
		\end{equation*}
		and note that
		\begin{equation} \label{eq:martingale-increment-converges}
			\sum_{n \geq 0} \E \left[\left( M_{n+1} - M_n \right)^2\right] = \mathrm{var} \left[ \indc_E \right] \leq 1
		\end{equation}
		by the orthogonality of martingale increments and
		\begin{equation} \label{eq:martingale-close-to-mean}
			\lim_{n \to \infty} \E \left[ \left( \indc_E -  M_n  \right)^2 \right] = 0
		\end{equation}
		by the convergence theorem for bounded martingales.
		
		The previous two properties imply that the event $E$ is not very dependent on the value of the edge $e_n$. To state this precisely, we introduce the notation $\mathbf{a} = (\mathbf{a}(e_i))_{i\in \N}$ and $\mathbf{a}^{e_n} := ((\mathbf{a}(e_i))_{i < n} , 1- \mathbf{a}(e_n), (\mathbf{a}(e_i))_{i > n})$ (\ie, the environment $\mathbf{a}^{e_n}$ is equal to the environment $\mathbf{a}$ except at the edge $e_n$ where we have flipped the value). 
		We claim that for any fixed small $\varepsilon > 0$, there exists an integer $n_0(\varepsilon)$ so that, for all $n \geq n_0$,
		\begin{equation} \label{eq:event-not-dependent}
			\E \left[ \left( \indc_{E} \left( \mathbf{a} \right) -  \indc_{E} \left( \mathbf{a}^{e_n} \right) \right)^2\right] \leq \varepsilon.
		\end{equation}
		To prove~\eqref{eq:event-not-dependent}, we first choose $n$ large enough so that, by \eqref{eq:martingale-increment-converges}, we have
		\begin{equation*}
			\sum_{k \leq n-2} \E[(M_{k+1}-M_k)^2] > \mathrm{var} \left[ \indc_{E} \right] \left(1- \frac{\min(\mathfrak{p} , 1 - \mathfrak{p})}{ 4 \max(\mathfrak{p} , 1 - \mathfrak{p})} \varepsilon \right).
		\end{equation*} 
		This implies, by \eqref{eq:martingale-close-to-mean}, that $\E[(\indc_E(\mathbf{a})-M_{n-1}(\mathbf{a}))^2] < (\varepsilon \min(\mathfrak{p} , 1 - \mathfrak{p}))/(4 \max(\mathfrak{p} , 1 - \mathfrak{p}))$. Since the value $\mathbf{a}(e_i)$ is sampled according to the Bernoulli distribution of parameter $\mathfrak{p}$ independently of the collection $(\mathbf{a}(e_j))_{j \neq i}$, we have the inequality
		\begin{equation*}
			\E[(\indc_E(\mathbf{a}^{e_n})-M_{n-1}(\mathbf{a}^{e_n}))^2] \leq \frac{\max(\mathfrak{p} , 1 - \mathfrak{p})}{ \min(\mathfrak{p} , 1 - \mathfrak{p})} \E[(\indc_E(\mathbf{a})-M_{n-1}(\mathbf{a}))^2] \leq \varepsilon/4.
		\end{equation*}
		Combining these inequalities yields \eqref{eq:event-not-dependent}.
		\medskip
		
		{\it Step 2: The event $E$ is sensitive.} \\
		We show that,
		\begin{equation}  \label{eq:flip-edge-implication}
			\indc_{E}(\mathbf{a}) = 1,~ \nabla u_f(e_n) \neq 0, ~\mbox{ and}~  \nabla \ellp{p}(e_n) \neq 0 \implies \indc_{E}(\mathbf{a}^{e_n}) = 0,
		\end{equation}
		where $\nabla \ellp{p}(e_n)$ and $\nabla u_f(e_n)$ denote the difference of the values at the two vertices of the edge $e_n$
		of the corrected plane and $u_f$ respectively (we are overloading notation here and using $\nabla u_f(e_n)$ to refer to this difference when $e_n$ is closed but both of its endpoints are in the infinite cluster). 
		When either vertex of the edge $e_n$ is not present in $\mathscr{C}_{\infty}$, we define $\nabla u_f(e_n) = \nabla \ellp{p}(e_n) = 0$.
		
		Lemma~\ref{lemma:finite-energy} guarantees that the gradient of the corrected plane in the environment~$\mathbf{a}^{e_n}$ exists almost surely. We denote this corrected plane by ${\ellp{p}}^{e_n}$ and the infinite cluster after flipping the edge $e_n$, \ie, changing $\mathbf{a}(e_n)$ from 0 to 1 or from 1 to 0, by $\mathscr{C}^{e_n}_{\infty}$. 
		
		To prove \eqref{eq:flip-edge-implication}, we first rule out the case where the set of vertices in $\mathscr{C}^{e_n}_{\infty}$ does not coincide with the set of vertices in $\mathscr{C}_{\infty}$.  If the set becomes larger after flipping the edge, then we have by definition that $\nabla u_f(e_n) = 0$ (as we have connected a finite isolated component to the infinite cluster). If the set becomes smaller, then $e_n$ must disconnect an isolated component and hence $\nabla u_f(e_n) = 0$.

		Now, suppose that the set of vertices of $\mathscr{C}^{e_n}_{\infty}$ and $\mathscr{C}_{\infty}$ are the same and denote by $\Delta_{\mathscr{C}^{e_n}_{\infty}}$ the graph Laplacian on the infinite cluster (with the edge flipped). We first note that, if $\nabla {\ellp{p}}^{e_n}(e_n) = 0$, then the map ${\ellp{p}}^{e_n}$ is harmonic on the infinite cluster $\mathscr{C}_\infty$ and thus the difference between the two functions $\ellp{p}$ and $\ellp{p}^{e_n}$ is constant (as the difference $\ellp{p} - \ellp{p}^{e_n}$ is a sublinear harmonic function on the infinite cluster). The same argument shows that if $\nabla \ellp{p}(e_n) = 0$ then $\ellp{p} - \ellp{p}^{e_n}$ is constant. These observations imply that $\nabla \ellp{p}(e_n) \neq 0$ if and only if $\nabla {\ellp{p}}^{e_n}(e_n) \neq 0$.

		Since the set of vertices of the two infinite clusters coincide, we can write
		\begin{align*}
			\sum_{y \in \mathscr{C}_\infty} f(y) (\ellp{p}(y) - {\ellp{p}}^{e_n}(y)) & = \sum_{y \in \mathscr{C}_\infty} -\Delta_{\mathscr{C}_{\infty}} u_f(y) \left( \ellp{p}(y) - {\ellp{p}}^{e_n}(y))\right) \\
			&= \sum_{y \in \mathscr{C}_\infty}  u_f(y) (\Delta_{\mathscr{C}_{\infty}} - \Delta_{\mathscr{C}^{e_n}_{\infty}}) {\ellp{p}}^{e_n}(y)  \\
			&\quad  \mbox{(integration by parts and $\Delta_{\mathscr{C}_{\infty}} \ellp{p} = \Delta_{\mathscr{C}^{e_n}_{\infty}} {\ellp{p}}^{e_n} = 0$)}  \\
			& =  (\mathbf{a} - \mathbf{a}^{e_n})(e_n) \nabla u_f(e_n)   \nabla {\ellp{p}}^{e_n}(e_n).
		\end{align*}
		We note that the integration by parts is justified because the difference $\ellp{p} - \ellp{p}^{e_n}$ is sublinear (and in fact one could prove that it converges to $0$ at infinity) and the function $u_f$ and its gradient decay sufficiently fast to zero at infinity.
		
		Since the event $E$ includes $\{ \sum_{y \in \mathscr{C}_\infty} f(y) \ellp{p}(y)  = 0 \} $, the previous display shows
		\[
		\indc_{E}(\mathbf{a}) = 1 ~\mbox{and}~ \nabla u_f(e_n) \neq 0 ~\mbox{and}~  \nabla {\ellp{p}}^{e_n}(e_n) \neq 0 \implies \indc_{E}(\mathbf{a}^{e_n}) = 0.
		\]
		This shows \eqref{eq:flip-edge-implication} as (as mentioned above) $\nabla \ellp{p}(e_n) \neq 0$ implies that $\nabla {\ellp{p}}^{e_n}(e_n) \neq 0$. 
		\medskip

		{\it Step 3: Conclusion.} \\
		In terms of indicator functions, the property \eqref{eq:flip-edge-implication} can be rewritten as 
		\begin{equation*}
			1\{E \cap \left\{ \nabla u_f(e_n) \neq 0 \right\}\cap \left\{ \nabla \ellp{p}(e_n) \neq 0 \right\}\}  \left( \mathbf{a} \right) \leq  \left( \indc_{E} \left( \mathbf{a} \right) -  \indc_{E} \left( \mathbf{a}^{e_n} \right) \right)^2.
		\end{equation*}
		By taking the expected value of the previous display and using \eqref{eq:event-not-dependent}, we obtain, for any $\varepsilon > 0$ and $n \geq n_0(\varepsilon)$,
		\begin{equation*}
			\E \left[ 1\{E \cap \left\{ \nabla u_f(e_n) \neq 0 \right\}\cap \left\{ \nabla \ellp{p} (e_n) \neq 0 \right\}\} \right] \leq  \varepsilon.
		\end{equation*}
		Summing over the edges in the ball $B_{2^k}$, we obtain, for $k$ large enough,
		\begin{equation*}
			\E \left[ \indc_{E}\left| \left\{ e \in E(B_{2^k}) \, : \, \nabla u_f(e) \neq 0 ~\mbox{and}~ \nabla \ellp{p}(e) \neq 0 \right\} \right|  \right] \leq \varepsilon  |B_{2^k} | .
		\end{equation*}
		In particular, we have
		\begin{equation} \label{eq:no-sensitive-edges}
			\lim_{k \to \infty} \frac{1}{|B_{2^k}|}  \E \left[ \indc_{E} \left| \left\{ e \in E(B_{2^k}) \, : \, \nabla u_f(e) \neq 0 ~\mbox{and}~ \nabla \ellp{p}(e) \neq 0 \right\} \right|  \right] = 0,
		\end{equation}
		which contradicts \eqref{eq:level-set-of-harmonic-function} except if $\P(E) = 0$. \end{proof}

	\subsection{Topological obstructions in the level set of the potential} \label{subsec:obstruction}
	Our goal for the next two subsections is to prove Proposition~\ref{prop:level-set-of-harmonic-function}.  We seek to show, for a fixed direction $p \in \R^2$, that the set of {\it sensitive} edges, \ie, edges $e$ for which $\nabla u_f(e) \neq 0$ and $\nabla \ellp{p}(e) \neq 0$, has a density. Our strategy is to identify a `good event' 
	which occurs with positive density in the cluster and map each good event to a sensitive edge.
	
	Specifically, in this subsection, we use Theorem \ref{theorem:lipschitz} and regularity properties of the cluster to 
	show that we may find a positive density of `good' edges $\tilde e$ in the cluster for which 
	$\nabla \ellp{p}(\tilde e)$ is large and for which the cluster near the edge is sufficiently well-behaved. 
	In the subsequent subsection we use an exploration process based on the construction of a block-cut tree to show that the level set of the potential around every good edge contains a nearby sensitive edge. 
	
	For the entirety of this subsection and the next, we restrict ourselves to the event in the statement of the proposition that  $\sum_{x \in \mathscr{C}_{\infty}} f(x) = 0$, and $|\supp u_f| = \infty$. Recall that this event implies that $u_f$ decays to zero at infinity. 
	Also note that as $|\supp u_f| = \infty$, we have that  $(\supp f ) \cap \mathscr{C}_\infty \neq \emptyset$, (as otherwise $u_f \equiv 0$).

	In Section \ref{subsubsec:graph-theory-defs} we collect some preliminary definitions on graphs, connectivity, 
	and the block-cut tree. Given a real number $a \in \R$, we denote by
	\begin{equation} \label{eq:level-set-def}
		L_a := \left\{ x \in \mathscr{C}_\infty \, : \, u_f(x) = a \right\}
	\end{equation}
	the level set of $u_f$ at the value $a$.  In Section \ref{subsubsec:unboundedlevelsets} we use planarity
	and the fact that $u_f$ is harmonic outside of a finite set to show that the connected components of $L_a$ cannot contain many disjoint infinite paths.  Then, in Section \ref{subsubsec:large-biconnected-components} 
	we show that the biconnected components of $L_a$ are not too large. Finally 
	in Section \ref{subsubsec:building-the-obstruction} we set up the topological obstructions and good event which
	we couple with sensitive edges.

	\subsubsection{Graph theory definitions and block cut tree} \label{subsubsec:graph-theory-defs}
	We consider graphs $G := (V,E)$ which are a collection of vertices $V$
	and set of unordered vertices $E$ also called edges. A subgraph induced by $V' \subset V$ is the graph $(V,' E')$ where $E'$ are the subset of edges with both ends in $V'$. 
	
	\begin{itemize}
		\item The \emph{degree} of a vertex $x \in V$ is the cardinality of the set of edges $e \in E$ satisfying $x \in e$.
		
		\item A \emph{(finite) path} in $G$ is a function $\gamma : [1 , \ldots, \mathrm{end}] \to G$ (with $\mathrm{end} \in \N$ if~$\gamma$ is finite) such that $\{ \gamma(i+1) , \gamma(i) \} \in E$.
		
		\item We say that $G$ is {\it connected} if for every two vertices 
		$x,y \in V$, there is a path  $\gamma \subset V$ connecting $x$ and $y$.
		\item The graph $G$ is {\it biconnected} if for every $x \in G$, the induced subgraph $G \setminus \{x\}$ is connected.
		\item A {\it connected component} of $G$ is an induced subgraph of $G$ which is connected
		and is maximal for the inclusion.  {\it Biconnected components} are defined analogously.
		
		\item A {\it cut-vertex} $v$ of $G$ is a vertex which lies in a connected component 
		such that the component becomes disconnected in the induced subgraph formed by removing the vertex.
		
	\end{itemize}

	We note that the collection of connected components of a graph $G$ partitions the set of vertices of $G$, but the collection of biconnected components do not (they in fact partition the set of edges of $G$), and two different biconnected components overlap at at most one vertex which is a cut-vertex of $G$. Reciprocally, any cut-vertex belongs to at least two biconnected components of $G$. Building upon these observations, it is possible to construct the block-cut tree of a graph $G$, and we first introduce a few defintions.
	
	\begin{itemize}
		\item A \emph{tree} $T = (V , E)$ is a graph in which any two vertices are connected by exactly one path. 
		\item A \emph{rooted tree} $(T, v)$ is a tree in which a special (labeled) vertex $v \in V$ has been singled out. We equip the rooted tree $(T, v)$ with a partial order by writing, for $x , y \in V$, $x \preceq y$ if the unique path going from the root to $y$ passes through $x$. We denote by $\dist(x , y)$ the length of the path connecting $x$ to $y$. 
		\item Given a rooted tree $(T, v)$ with $T = (V , E)$, \emph{a leaf} is a vertex of $V \setminus \{ v\}$ whose degree is equal to $1$. 
		\item A {\it branch} is a path from the root to a leaf. 
		\item  Given a vertex $y \in V \setminus \{ v\}$, we define \emph{the parent} of $y$ to be the only vertex $x \in V$ such that $x  \preceq y$ and $\dist(x , v) = \dist(y , v) -1$. Similarly, we define \emph{the children} and \emph{the descendants} of $y$ to be respectively the collections of vertices
		$$ \{ x \in V \, : \, y \preceq x \, \mbox{and} \, \dist(x , v) = \dist(y , v) +1 \}~\mbox{and} ~\{ y \in V \, : \, x \preceq y \}.$$
		\item The {\it block-cut tree} $G_{\mathrm{tree}}$ of a connected graph $G := (V,E)$ is the tree
		formed in the following way. The vertex set of $G_{\mathrm{tree}}$ is the collection of all the biconnected components of $G$ and all the cut-vertices of $G$.  
		There is an edge between a biconnected component $C$ and a cut-vertex $x$ if and only if $x \in C$ (see Figure~\ref{fig:graph-and-tree}). Note that with this construction, biconnected components are only connected to cut-vertices and cut-vertices are only connected to biconnected components.
	\end{itemize}

	\begin{figure}
		\centering
		\fbox{\includegraphics{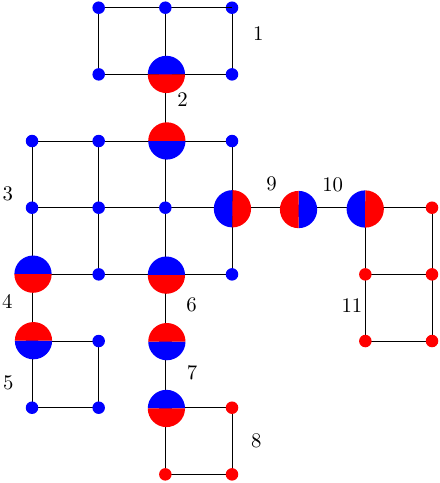}} \qquad 
		\includegraphics{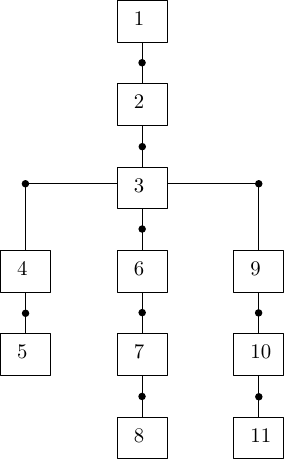}
		\caption{A subgraph of $\Z^2$ with its biconnected components drawn in alternating blue and red. Multi-colored vertices are cut vertices which belong to multiple biconnected components. On the right 
			is the block-cut tree, the biconnected components are drawn in the blocks with the same labelling as on the left. The cut points between the biconnected components are displayed as black vertices.}
		\label{fig:graph-and-tree}
	\end{figure}
	
	\subsubsection{Unbounded level sets of the potential} \label{subsubsec:unboundedlevelsets}
	In this section, we establish that the connected components of the level sets of the function $u_f$ cannot contain more than a deterministic number (depending only on the function $f$) of disjoint infinite paths.

	\begin{prop} \label{prop:infinite-disjoint-paths-limit}
		Let $\mathcal{S}$ be a connected component of $L_0$. Then there exists at most $(\left| \supp f \right|+2)$ disjoint infinite paths in $\mathcal{S}$. 
	\end{prop}
	
	\begin{remark}
		This statement is only relevant for the level set associated with the value $0$ as all the other level sets are bounded and thus finite (since $u_f$ decays to $0$ at infinity)
	\end{remark}
	
	Before proving Proposition~\ref{prop:infinite-disjoint-paths-limit}, we collect a few results pertaining to supercritical percolation and the level sets of discrete harmonic functions. They are written in Lemma~\ref{lemma:loops-around-annuli} to Lemma~\ref{lemma:long-nonzero-paths} below. We start with a connectivity lemma on the cluster. 
	\begin{lemma} \label{lemma:loops-around-annuli}
		There exists a minimal scale $\mathcal{M}_{\mathrm{annuli}}$ which is finite almost surely such that for all $R  \geq \mathcal{M}_{\mathrm{annuli} }$
		there is a loop in $\mathscr{C}_{\infty} \cap (Q_{2 R} \setminus Q_R)$ which has nonzero winding number around the origin.
	\end{lemma}
	
	\begin{remark}
		We only need to prove that the minimal scale $\mathcal{M}_{\mathrm{annuli} }$ is finite almost surely, but we remark that the proof below is quantitative and shows that its tail decays exponentially fast.
	\end{remark}
	
	\begin{proof}
		By, for example, \cite[Theorem 11.1 and (11.5)]{kesten-book}, there exists a constant $c(\mathfrak{p}) \in (0,1)$ such that 
		\[
		\P[\mbox{there exists an open left-to-right path in $[0 , 2R] \times [0 , R]$}] \geq 1- C e^{-c R}, \quad \mbox{for all $R \geq 1$}. 
		\]
		The claim (which is standard, see, \eg, \cite{grimmett-long-paths} for a similar argument) then follows by a union bound and planarity. 
	\end{proof}
	
	In the next lemma we use the following stronger version of the maximum principle which holds only when the Green's function of the graph is unbounded. 
	Note in the statement the important omission of the assumption that $f \leq M$ at infinity.  
	\begin{lemma} \label{lemma:twod-strong-max-principle}
		Suppose $f: \mathscr{C}_{\infty} \to \R$ is bounded and harmonic on an infinite subset $\Omega \subset \mathscr{C}_{\infty}$
		which is not the entire cluster, $\Omega \neq \mathscr{C}_{\infty}$. For every $M \in \R$, if $f \leq M$ on $\partial \Omega$, then $f \leq M$ on $\Omega$.
	\end{lemma}
	\begin{proof}
		Since $\Omega$ is not the entire cluster, there is a point $a \in \partial \Omega$ for which $f(a) \leq M$. 
		Let $v := G(\cdot, a)$, where $G$ is the Green's function for the cluster as defined in \eqref{eq:green-function-2D}. 
		For a parameter $t > 0$, let $w_t :=  f - M - t v$ and observe that $w_t$ is harmonic on $\Omega$ and, since $v$ is positive, 
		$w_t  \leq 0$ on $\partial \Omega$. By, for example, \cite{BH09}, for every $t > 0$, $\lim_{|x| \to \infty} t v(x) = \infty$. 
		In particular, since $f$ is bounded, for every $t > 0$, $w_t < 0$ at infinity. Hence, by the standard maximum principle, $w_t < 0$ on $\Omega$. 
		By taking $t \to 0$, this implies $f \leq M$ on $\Omega$. 
	\end{proof}

	\begin{lemma} \label{lemma:poles-sub-super-level-sets}
		For every $a \in \R$, each connected component of $\{x : u_f(x) > a \}$ must intersect a pole. 
	\end{lemma}
	\begin{proof}
		If $a \geq 0$, this follows from the standard maximum principle, using the assumption that $u_f$ is zero at infinity. If $a \leq 0$, this 
		follows from Lemma \ref{lemma:twod-strong-max-principle} and the assumption that $u_f$ is bounded (as it decays to zero at infinity).  \end{proof}

	We use the infinite support assumption to construct arbitrarily long paths where the function $u_f$ is nonzero. 
	\begin{lemma} \label{lemma:long-nonzero-paths}
		Fix an integer $R > 0$ such that $\supp f \subseteq Q_R$, for every integer $R' > R$ and on the event that $\supp u_f$ is infinite, there exists a path in the infinite cluster $\mathscr{C}_\infty$ connecting $Q_R$ to $\mathscr{C}_\infty \setminus Q_{R'}$ along which $u_f \neq 0$. 
	\end{lemma}
	\begin{proof}
		As the support of $u_f$ is infinite, there is a point, $x \in \mathscr{C}_{\infty} \setminus Q_{R}$ such that $u_f(x) \neq 0$. 
		By Lemma \ref{lemma:poles-sub-super-level-sets}, the possibly infinite connected component of $\{ z : u_f(z) \geq u_f(x) \}$ must intersect a pole, \ie, the support of $f$. \end{proof}
	
	The rest of this section is devoted to the proof of Proposition~\ref{prop:infinite-disjoint-paths-limit}, and we first establish in Lemma~\ref{lemma:no-thick-paths} that the level set $L_0$ cannot contain a bi-infinite path $\gamma$ as well as the intersection of the infinite cluster with one of the two connected components of $\R^2 \setminus \gamma$. The proof relies crucially on planarity and is thus restricted to the two dimensional setting (while up to now all the techniques introduced were valid in any dimension $d \geq 2$).
	
	\begin{figure}
		\centering
		\includegraphics[width=0.5\textwidth]{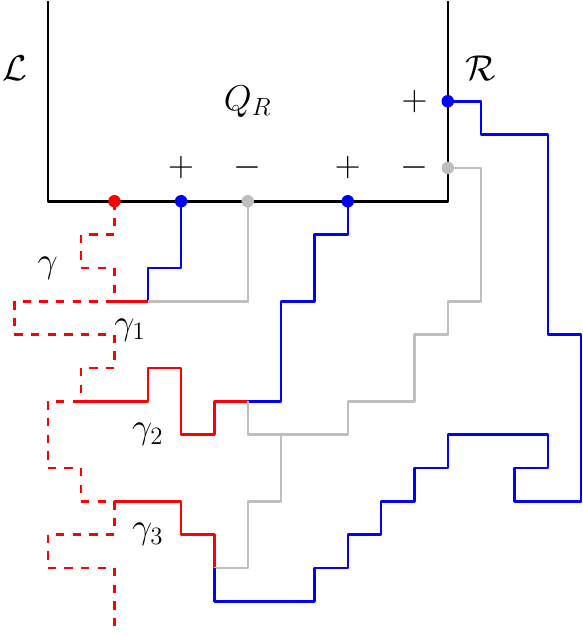}
		\caption{A description of what occurs in the proof of Lemma \ref{lemma:no-thick-paths}. The box $Q_R$ contains the support of $f$, the red lines are paths in the zero level set, the blue lines in the positive level set, and the gray lines in the negative level set.
			The dashed red line is $\gamma(1, \ldots)$ and the solid red lines are $\gamma_i$. }
		\label{fig:thick-paths}
	\end{figure}

	\begin{figure}
		\centering
		{\includegraphics[width=0.35 \textwidth]{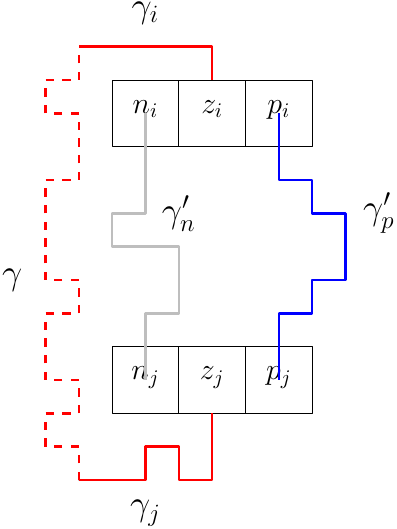}} 
		\caption{A visual description of the argument at the end of the proof of Lemma \ref{lemma:no-thick-paths}. The same color scheme as Figure \ref{fig:thick-paths} is used and only part of the bi-infinite path $\gamma$ is shown.}
		\label{fig:disjoint-loops-not-possible}
	\end{figure}

	\begin{figure}
		\centering
		{\includegraphics[scale=0.75]{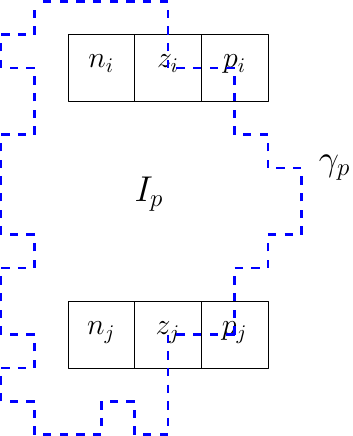}} 
		{\includegraphics[scale=0.75]{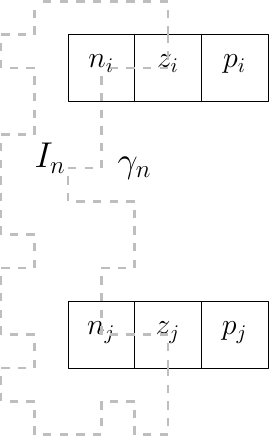}} 
		{\includegraphics[scale=0.75]{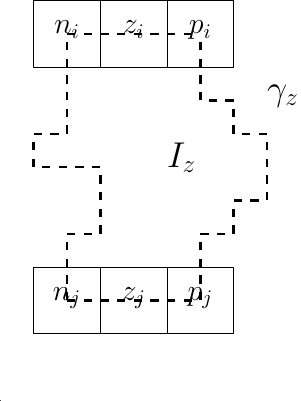}}
		\caption{These display the three loops corresponding to the paths in Figure \ref{fig:disjoint-loops-not-possible} as defined in equations \eqref{eq:p-loop} and \eqref{eq:z-loop}. }
		\label{fig:loops-labels}
	\end{figure}

	\begin{lemma} \label{lemma:no-thick-paths}
		Almost surely, the two following events do not occur simultaneously: 
		\begin{itemize}
			\item[(i)] The support of $u_f$ is infinite;
			\item[(ii)] The level set $L_0$ contains a bi-infinite path $\gamma$ and one of the connected components of $\mathscr{C}_{\infty} \cap (\R^2 \setminus \gamma)$ is included in $L_0$. 
		\end{itemize}
	\end{lemma}
	
	\begin{proof}
		Suppose for sake of contradiction that the support of $u_f$ is infinite and that the level set $L_0$ contains a bi-infinite path $\gamma$. Let us denote by $\mathcal{L}$ and $\mathcal{R}$ the intersection of the two connected components of $(\R^2 \setminus \gamma)$ with the infinite cluster $\mathscr{C}_\infty$, so that the bi-infinite path~$\gamma$ partitions the cluster as $\mathscr{C}_{\infty} = \gamma \cup \mathcal{L} \cup \mathcal{R}$ (note that with these definitions, we may have $\mathcal{L} = \emptyset$ or $\mathcal{R} = \emptyset$). Assume additionally for the sake of contradiction that $\mathcal{L} \subset L_0$. 
		
		Take $R$ sufficiently large so that $R \geq \mathcal{M}_{\mathrm{annuli} }$ and $Q_R$ contains the support of $f$. Suppose, without loss of generality that $\gamma(1) \in \partial Q_R$ and $\gamma(1, \ldots, \infty)$ does not intersect $Q_{R}$. 
		Apply Lemma \ref{lemma:long-nonzero-paths} with $R$ and $R'$ sufficiently large, depending on $R$,  to construct a long 
		path $\gamma'$ in $\mathscr{C}_{\infty} \setminus Q_R$ which passes through at least $C |\partial Q_R|^2$ disjoint annuli of the form, $Q_{2 m} \setminus Q_m$ for $m \geq R$. Thus, by Lemma \ref{lemma:loops-around-annuli} there are at least  $C |\partial Q_R|^2$ disjoint paths, $\gamma_i$, for which $\gamma_i(1) \in \gamma$, $u_f(\gamma_i(\mathrm{end})) \neq 0$, and $\gamma_i$ is disjoint from $\mathcal{L} \cup Q_R$.
		By the pigeonhole principle, we may assume that for each $i$,  $\gamma_i(1) \in \gamma(1, \ldots, \infty)$. 
		
		By examining the arrangement of nonzero neighbors around each $\gamma_i$, we find a connected component of either $\{x : u_f(x) > 0 \}$
		and $\{x : u_f(x) < 0 \}$ --- see Figure \ref{fig:thick-paths}.  By Lemma \ref{lemma:poles-sub-super-level-sets}, each connected component connects
		to the boundary of $Q_R$. Start by letting, for each path $\gamma_i$ the number $r_i$ to be the first index where $\gamma_i(r_i) \neq 0$.  Observe that since $u_f$ is harmonic at $\gamma_i(r_i-1)$,
		we must have at least one neighbor $p_i \sim \gamma_i(r_i-1)$ for which $u_f(p_i) > 0$ and another neighbor $n_i \sim \gamma_i(r_i-1)$ for which $u_f(n_i) < 0$. 
		Denote by $P_i$ the connected component of the set $\left\{u_f > 0 \right\} \cap (\mathscr{C}_{\infty} \setminus Q_R)$ containing the vertex $p_i$ and by $N_i$ the connected component of the set $\left\{u_f < 0 \right\} \cap (\mathscr{C}_{\infty} \setminus Q_R)$ containing the vertex $n_i$. 
		
		By our assumption on the number of such paths, Lemma \ref{lemma:poles-sub-super-level-sets}, and the pigeonhole principle, there are two paths $\gamma_i$ and $\gamma_{j}$ (with $i \neq j$) such that $P_i = P_{j}$ and $N_i = N_{j}$. Write $z_i := \gamma_i(r_i-1)$ (the first point on the path which lies in $L_0$). Let $\gamma_n'$ be a path in $N_i \cap N_{j}$ connecting $n_i$ and $n_{j}$,  $\gamma_p'$ a path in $P_i \cap P_{j}$ connecting  
		$p_i$ and $p_{j}$ and $\gamma_z'$ a path in $\{\gamma_i \cup \gamma_{j} \cup \gamma\} \setminus Q_R$ connecting $z_i$ and $z_{j}$. 
		See Figure \ref{fig:disjoint-loops-not-possible} for a visualisation of these loops. 
		
		We use these paths to form loops in $\mathcal{R} \setminus Q_R$, which we denote by $\gamma_p$, $\gamma_n$, and  $\gamma_z$ --- see Figure \ref{fig:loops-labels}: 
		\begin{equation} \label{eq:p-loop}
			\begin{aligned}
				\gamma_p(1) &= p_i \\
				\gamma_p(1, \ldots, k_p) &= \gamma_p'  \\
				\gamma_p(k_p) &= p_{j} \\
				\gamma_p(k_p+1) &= z_{j} \\
				\gamma_p(k_p+1, \ldots, \mathrm{end}) &= \gamma_z' \\
				\gamma_p(\mathrm{end}) = z_i ,
			\end{aligned}
		\end{equation}
		where $k_p = |\gamma_p'|$. We define $\gamma_n$ in the same way, replacing the occurrences of $p$ in the above equation by $n$. 
		We define the loop $\gamma_z$ by forming a rectangle with sides $\gamma_n'$ and $\gamma_p'$ and the paths connecting $n_i, z_i, p_i$ and $n_{j}, z_{j}, p_{j}$ respectively: 
		\begin{equation} \label{eq:z-loop}
			\begin{aligned}
				\gamma_z(1) &= p_i \\
				\gamma_z(1, \ldots, k_p) &= \gamma_p'  \\
				\gamma_z(k_p) &= p_{j} \\
				\gamma_z(k_p + 1) &= z_{j} \\
				\gamma_z(k_p + 2) &= n_{j} \\
				\gamma_z(k_p + 2, \ldots, k_p + 2 + k_n) &= \gamma_n' \\
				\gamma_z(k_p + 2 + k_n) &= n_i \\
				\gamma_z(k_p + 2 + k_n + 1) &= z_i   .
			\end{aligned}
		\end{equation}
		
		By the Jordan curve theorem, each such loop decomposes space into an inside (finite connected component) and outside (infinite connected component): $\gamma_{\bullet}$ into $(I_{\bullet}, O_{\bullet})$ respectively. 
		First note that the closure of each such finite connected component $I_{\bullet}$ cannot contain $Q_R$. Indeed, by construction the boundaries cannot contain $Q_R$. Also, 
		if $I_{\bullet}$ contains $Q_R$, then some $\gamma_{\bullet}$ must intersect $\mathcal{L}$. However, by definition $\gamma$ does not and $u_f$ is nonzero 
		on both $\gamma_n'$ and $\gamma_p'$, which cannot occur as $u_f \equiv 0$ on $\mathcal{L}$ by assumption.
		
		Since $\gamma$ is a bi-infinite path which partitions the space into at least one connected component on which $u_f$ is zero, $\gamma$ must be disjoint from $I_z$. 
		This implies that either $\cl(I_p)$ or $\cl(I_n)$ contain $\cl(I_z)$, which respectively, disconnects $n_i$ or $p_i$ from $Q_R$. (For example, in Figures \ref{fig:disjoint-loops-not-possible} and \ref{fig:loops-labels}, $n_i$ is disconnected from $Q_R$.)  This is a contradiction as, by, Lemma \ref{lemma:poles-sub-super-level-sets}, there must be a path from $n_i$ and $p_i$ to $\partial Q_R$.
		
		Therefore, either $P_i \neq P_{j}$ or $Q_i \neq Q_{j}$, which is a contradiction to our choice of paths, completing the proof.  \end{proof}
	
	We finally prove Proposition~\ref{prop:infinite-disjoint-paths-limit} building upon Lemma~\ref{lemma:no-thick-paths} (and once again planarity in the form of Jordan curve theorem on the sphere $\mathbb{S}^2$).
	
	\begin{proof}[Proof of Proposition~\ref{prop:infinite-disjoint-paths-limit}]

		Since we are in the event where  $\lim_{|x| \to \infty} u_f(x) = 0$, the only
		possibly infinite level set can be $L_0$. Further, since we have that assumed that the support of $u_f$ is infinite,  we have that $L_0 \neq \mathscr{C}_{\infty}$.
		
		Now suppose for sake of contradiction that, for $K := \left| \supp f \right|$, there are $(K+2)$ disjoint infinite paths in~$\mathcal{S}$. Denote these paths by $\gamma_1 , \ldots, \gamma_{K+2}$ and denote their initial points by $y_1 , \ldots, y_{K+2}$ (\ie, $\gamma_i(0) = y_i$). Since the set $\mathcal{S}$ is connected, for any $i \in \{1 , \ldots, K \}$, we may find a finite path $\tilde \gamma_i$ which connects the vertices $y_i$ and $y_{i+1}$ in $\mathcal{S}$. 
		
		We next consider the finite (discrete) connected components of $C_1 , \ldots, C_N$ of $\Z^2 \setminus \bigcup \tilde \gamma_i$ and then denote by $\mathscr{S}$ the interior of the set $(\bigcup \tilde \gamma_i \cup \bigcup C_i) + [- 1/2 , 1/2]^2 \subseteq \R^2$. This construction ensures that $\mathscr{S}$ is a bounded domain of $\R^2$ whose complement is connected, it is thus simply connected. Additionally, its boundary is simple (as it is composed of a finite union of straight lines). In particular, its boundary is homeomorphic to the circle $\S^1$, and is a Jordan curve which we denote by $\beta$.

		For the rest of the argument, we will work in the continuous space $\R^2$ instead of $\Z^2$ (in order to apply Jordan curve theorem). To this end, we extend the definitions of the paths $\gamma_i$ from the discrete to the continuum by a piece-wise linear interpolation.

		For each $i \in \{ 1 , \ldots, K+2\}$, we let $x_i$ be the last point in $\gamma_i$ which intersects $\mathscr{S}$. Specifically, these points can be defined by the identity (the supremum being finite since the paths $\gamma_i$ go to infinity and the set $\mathscr{S}$ is bounded)
		\begin{equation*}
			t_i := \sup \left\{ t \in [0 , \infty) \, : \, \gamma_i(t) \in \mathscr{S} \right\} ~~\mbox{and}~~ x_i := \gamma_i(t_i).   
		\end{equation*}
		By redefining $\gamma_i$, we may assume that $\gamma_i(0) = x_i$. Using the (inverse of the) stereographic projection, we may see the paths $\gamma_1, \ldots, \gamma_{K+2}$ as continuous functions defined in $[0 , 1]$ and valued in $\S^2$ such that $\gamma_i(0) = x_i$ and $\gamma_i(1) = N$. Similarly, we consider the set $\mathscr{S}$ as a subset of the sphere $\S^2$.
		
		We can thus apply Lemma \ref{lemma:jordan-curves-theorem} with the loop $\beta$ and the paths $\gamma_i$, and obtain that the set $\S^2 \setminus \{\beta \cup \bigcup \gamma_i\}$ has at least $(K+2)$ disjoint connected components, denoted by $C_1 , \ldots, C_{K+2}$ such that $N \in \partial C_i$. The pigeonhole principle implies that one of these components does not contain a pole of the function $u_f$. We may without loss of generality assume that it is the connected component $C_1$.

		To conclude the proof, we consider the image of the set $C_1$ by the stereographic projection (to see it as a subset of $\R^2$ instead of $\S^2$) and denote by $\mathcal{C}_1 := C_1 \cap \mathscr{C}_\infty$. The function $u_f$ is harmonic in the set $\mathcal{C}_1$, and the (discrete) outer boundary of $\mathcal{C}_1$ is included in the set $\mathcal{S}$ (since $\partial C_1 = \gamma_1 \cup \gamma_{K+1} \cup \beta_1$). The maximum principle and the observation that $u_f$ tends to $0$ at infinity imply that $u_f = 0$ on $\mathcal{C}_1$. 
		
		Combining all the observations above, we deduce that the level set $L_0$ contains a bi-infinite path $\gamma$ such that, if we denote by $C_1 , C_2$ the connected components of $\R^2 \setminus \gamma$, then the function $u_f$ is identically equal to $0$ on either $\mathscr{C}_\infty \cap C_1$ or $\mathscr{C}_\infty \cap C_2$. This behavior is ruled out by Lemma~\ref{lemma:no-thick-paths}.
	\end{proof}

	\subsubsection{Large biconnected components in the level set of the potential}
	\label{subsubsec:large-biconnected-components}
	
	In this section, we establish a second statement pertaining to the level set $u_f$. Specifically, we prove that any biconnected component included in the level set cannot be too large. The proof builds upon a similar mechanism as the one exploited in Section~\ref{subsubsec:unboundedlevelsets} (and uses a topological obstruction which crucially relies on planarity arguments and is thus restricted to the two dimensional setting).
	
	The argument is split into different steps: we first prove (similarly as in the proof of Proposition~\ref{prop:infinite-disjoint-paths-limit}) that any biconnected component cannot be connected to infinity by more than a deterministic number (depending only on $f$) of disjoint paths. We then use the notion of Kesten channels (following~\cite[Theorem 11.1]{kesten-book} and using the terminology of~\cite{isoperimetry-mathieu-remy}) to prove that the probability of a set to be connected to infinity by more than the aformentioned deterministic number of disjoint paths is (stretched) exponentially close to $1$ as the diameter of the set tends to infinity.
	
	\begin{figure}
		\centering
		\includegraphics[width=0.35\textwidth]{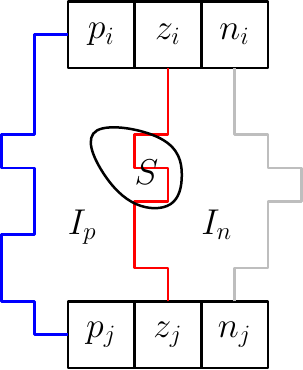} \qquad
		\includegraphics[width=0.35\textwidth]{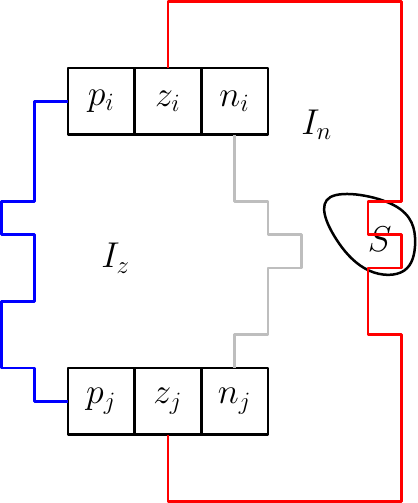}
		\caption{The two possible cases in the proof of Proposition~\ref{prop:boundary-size-of-set}.}
		\label{fig:tri-connected-components-disconnect-cases}
	\end{figure}

	\begin{prop} \label{prop:boundary-size-of-set}
		There exists an integer $C_f \in \mathbb{N}$ depending only $f$ such that the following properties hold:
		\begin{itemize}
			\item The support of the function $f$ is contained in the box $Q_{C_f/2}$.
			\item For every value $a \in \R$ and every biconnected set $S \subseteq L_a$ satisfying the property that the finite connected components of $\Z^2 \setminus S$ do not intersect the support of $f$, it holds that there are at most $C_f$ disjoint paths connecting $S$ to infinity.
		\end{itemize}
	\end{prop}
	
	\begin{proof}
		Let $R \geq 20$ be a constant (depending only on $f$), so that $f = 0$ outside of $Q_{R/2}$ and let $C_f = C R^3$. 
		Suppose for sake of contradiction that there is a biconnected subset of $L_a$ violating the statement of the lemma
		and denote it by $S$. 
		
		Note that, by Lemma \ref{prop:infinite-disjoint-paths-limit} (and reducing the number of paths by a constant) we can assume that every path $\gamma_i$ eventually exits $L_a$ and that none of the paths intersect the support of $f$. 
		Let $p_i$ be the first point on each path $\gamma_i$ for which $u_f(\gamma_i(r_i)) = u_f(p_i) \neq a$ and assume that $u_f(p_i) > a$. Relabel the paths 
		so that $\gamma_i(1) = \gamma_i(r_i-1) \in L_a$ and $\gamma_i(2) = \gamma_i(r_i)$.   Write $z_i = \gamma_i(1)$, and note that since $u_f$ is harmonic at each $z_i$ there is a neighbor $n_i \sim z_i$ for which $u_f(n_i) < a$.
		
		We show, using an argument similar to the proof of Lemma \ref{lemma:no-thick-paths} that since there are too many paths, there must be a topological contradiction. 
		Write $P_i$ and $N_i$ for the connected components of the sets $ \left\{u_f > a  \right\} \cap (\mathscr{C}_{\infty} \setminus Q_R)$ and $ \left\{u_f < a  \right\} \cap (\mathscr{C}_{\infty} \setminus Q_R)$ containing the vertices $p_i$ and $n_i$ respectively. 
		
		By the pigeonhole principle, there are three distinct indices $i, j, k$ such that $N_i = N_{j} = N_k$ and $P_i = P_{j} = P_k$ and let $\gamma_p'$, $\gamma_n'$ denote paths in $P_i$ and $N_i$ which connect $p_{i}$ to $p_{j}$ and $n_{i}$ to $n_{j}$ respectively. 
		Let $\gamma_z'$ denote a path connecting $z_{i}$ to $z_{j}$ in $L_a$ through $\gamma_i$, $\gamma_{j}$ and $S$ which does not contain $z_k$ (this can be achieved since $S$ is biconnected). Define the loops $\gamma_z, \gamma_n, \gamma_p$
		and the finite connected sets $I_z, I_n, I_p$ as in the end the proof of Lemma \ref{lemma:no-thick-paths}. 
		
		After a symmetry reduction, there are two possible cases for the topology of the sets --- see Figure \ref{fig:tri-connected-components-disconnect-cases} ---
		which are both incompatible with the fact that the sets $N_i = N_j = N_k$ and $P_i = P_j = P_k$. 
		We refer  to the case on the left in Figure \ref{fig:tri-connected-components-disconnect-cases} as Case 1, the case where $I_p \cup I_n \subset I_z$ and the case on the right, the case where $I_z \cup I_n \subset I_p$, as Case 2. By definition, we must have that $z_k$ is either contained in $I_p$, $I_n$, $I_z$ or their complement, \ie, it cannot be on the boundaries. 
		\medskip
		
		{\it Case 1.} \\
		If $z_k$ is in $I_p$, then $n_k$ is disconnected from $n_j$. Symmetrically, $z_k$ cannot be in $I_n$. Thus, $z_k$ is in the complement of $I_p \cup I_n$.
		However, since the set $S$ is biconnected, there must be a path from $z_k$ to $S \cap \gamma_z'$ which does not use $z_i$ or $z_j$. If there exists a path using neither $z_i$ nor $z_j$, then it must cross either the set $I_p$ or the set $I_n$ which is a contradiction. If there are two disjoint paths, one using $z_i$ and one using $z_j$, we may then generate a loop in the level set $L_a$ containing either $I_n$ or $I_p$ which cannot happen by the assumption that the finite connected components of $\Z^2 \setminus S$ do not intersect the support of $f$. So Case 1 cannot occur. 
		\medskip
		
		{\it Case 2.} \\
		If $z_k$ is in $I_z$, then as there are two disjoint paths connecting $z_k$ to $S$ there must be one path going through $z_i$ and one going through $z_j$, which allows to generate a loop and to conclude as in Case 1. Also, as in Case 1, $z_k$ cannot be in $I_n$. 
		However, if $n_k \in I_p^c$, then it is disconnected from $n_j$. \end{proof}
	
	Let us remark that Proposition~\ref{prop:boundary-size-of-set} relies on a deterministic statement, and its proof only relies on the harmonicity of the function $u_f$ and planarity. The next step of the proof is to combine it with tools of percolation to rule out with high probability the existence of large biconnected components on the level sets of the function $u_f$. Specifically, for each vertex $x \subseteq \Zd$, we introduce the following event
	\begin{multline} \label{eq:def.eventE}
		E_R \left( x \right) := \left\{ \mbox{For each finite connected subset } S \subseteq \mathscr{C}_\infty \mbox{ with } \diam S \geq R \mbox{ and } x \in S , \right. \\
		\left. \mbox{there exist at least } (C_f + 1) \mbox{ disjoint paths connecting } S \mbox{ to infinity} \right\}.
	\end{multline}
	We remark that, with this definition, the event $ E_R \left( x \right)$ is satisfied if $x$ does not belong to the infinite cluster $\mathscr{C}_\infty$.
	
	\begin{prop} \label{prop:disjoint-paths-to-infinty-exponential}
		There exist two constants $C(\mathfrak{p},f) < \infty$ and $c(\mathfrak{p},f) > 0$ and a universal exponent $s > 0$ such that, for any $x \in \Z^2$ and any $R \geq 1$,
		\begin{equation} \label{eq:prop:disjoint-paths-to-infinty-exponential}
			\P \left(  E_R \left( x \right) \right) \geq 1 - C \exp \left( - c R^{s} \right).
		\end{equation}
	\end{prop}
	
	\begin{remark}
		An explicit value for the stochastic integrability exponent $s > 0$ could be obtained from the proof, and we believe that the argument could be modified so as to optimize this parameter. The exponent is not equal to $1$ due to the definition of well-connected boxes which for instance includes that mesoscopic boxes of size $N^{1/(400)}$ (with $400 = (10 d)^2$ in two dimensions) are crossing. Nevertheless, writing the result in the form of~\eqref{eq:prop:disjoint-paths-to-infinty-exponential} simplifies the proof and is sufficient for our purposes.
	\end{remark}

	The proof of Proposition~\ref{prop:disjoint-paths-to-infinty-exponential} relies on a auxiliary percolation process on the renormalized lattice defined as follows. 
	We first introduce a definition and say a box $(x + Q_R)$ is called {\it very good} if the following occurs:
	\begin{itemize}        
		\item The box $(x + Q_R)$ is well-connected, as in Proposition \ref{prop:well-connected}.
		\item The box $(x + Q_R)$ contains Kesten channels (following the terminology from \cite{isoperimetry-mathieu-remy}): there is a partition of $(x  + Q_R)$ into $R^{1/2}$ disjoint horizontal rectangles (channels) and disjoint vertical channels (which may intersect the horizontal channels) of short side length $R^{1/2}$ and long side length $R$ such that each channel contains an open path which connects the faces of the rectangles. (See Figure \ref{fig:horizontal-kesten-channels}.)
		\item The box $(x + Q_R)$ contains horizontal Kesten channels (following the terminology from \cite{isoperimetry-mathieu-remy}): there is a partition of $(x  + Q_R)$ into $R^{1/2}$ disjoint horizontal rectangles (channels) of short side length $R^{1/2}$ and long side length $R$ such that each channel contains an open path which connects the faces of the rectangles. (See Figure \ref{fig:horizontal-kesten-channels}.)
		\item The box $(x + Q_R)$ contains vertical Kesten channels (using the same definition as above but replacing the word horizontal by vertical)
	\end{itemize}
	By Proposition~\ref{prop:well-connected} and~\cite[Theorem 11.1]{kesten-book}, a cube is very good with stretched exponentially high probability in its side length, and one has the lower bound
	\begin{equation} \label{eq:verygoodstretchedexpclose1}
		\P \left[ (x + Q_R) \mbox{ is very good} \right] \geq 1 - C \exp \left( - c R^{s} \right).
	\end{equation}
	
	We then consider the site percolation process where the sites are the boxes of the form $(x + Q_R)$ with $x \in (2R + 1) \Z^2$, and two sites $(x + Q_R)$ and $(y + Q_R)$ are neighbors if $|x - y| = 2R + 1$. A site $(x + Q_R)$ is declared open if the box $(x + Q_{11R/10})$ is very good (note that we slightly enlarge the size of the box here, this is to ensure that any set $S$ as in~\eqref{eq:def.eventE} satisfies $\diam (S \cap (x + Q_{11R/10})) \geq c R$). We call this percolation process the \emph{renormalized percolation process of size $R$}, and collect below two of its properties:
	\begin{itemize}
		\item The renormalized percolation process is $1$-dependent: for any $x , y \in (2R + 1) \Z^2$ with $|x - y| > 2R + 1$, the events 
		\[
		\{ (x + Q_{11R/10}) \quad \mbox{is very good} \} \qquad \mbox{and} \qquad     \{ (y + Q_{11R/10}) \quad \mbox{is very good} \}
		\]
		are independent (the nearest neighbor dependency is both due to the definition of well-connectedness and the fact that the boxes $(x + Q_{11R/10})$ for $x \in (2R+1)\Z^2$ have some overlap).
		\item If we consider two very good boxes $(x + Q_{11R/10})$ and $(y + Q_{11R/10})$ with $|x - y| = 2R +1$, then the horizontal and vertical Kesten channels of these boxes are connected.
		Indeed, by well-connectedness, each Kesten channel in~$(x+Q_{11R/10})$ is connected within~$(x+Q_{11R/10})$ to~$\mathscr{C}_{*}(x+Q_{11R/10})$. Also, as each Kesten channel in~$(y + Q_{11R/10})$ with an endpoint in $(x+ Q_{11R/10})$ has length at least~$\frac{R}{10}$ in~$(x+ Q_{11R/10})$ it is also connected to~$\mathscr{C}_{*}(x+Q_{11R/10})$. 
		\item If we consider two very good boxes $(x + Q_{11R/10})$ and $(y + Q_{11R/10})$ with $|x - y| = 2R +1$, then the matching horizontal Kesten channels of these boxes are disjointly connected if the two boxes are \emph{horizontal} neighbors (by this, we mean that the horizontal channels in the highest horizontal rectangles are connected, the horizontal channels in the second highest horizontal rectangles are connected etc. and that these connections happen disjointly). This is a consequence of the following observation: for $R$ sufficiently large, the intersection of the boxes $(x + Q_{11R/10})$ and $(y + Q_{11R/10})$ (which is a vertical rectangle of short side length $R/5$ and of long side length $11R/10$) contains one of the vertical channels of the box $(x + Q_{11R/10})$. This channel is, by the definition of a very good box, crossed by an open path which, by the planarity of the two dimensional lattice, connects all the horizontal channels of the boxes $(x + Q_{11R/10})$ and $(y + Q_{11R/10})$ together (N.B. with this construction, the paths connecting two distinct pairs of matching horizontal channels can be chosen so that they do not intersect, i.e., the connections happen disjointly).
		
		Similarly, the matching vertical Kesten channels are connected if the two boxes are \emph{vertical} neighbors (see Figure~\ref{fig:kesten-channels-infinite}).
	\end{itemize}
	
	\begin{figure}
		\centering
		\includegraphics{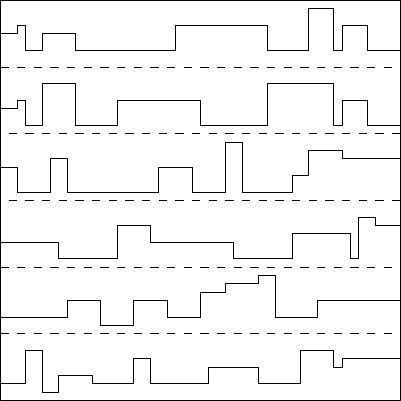}
		\caption{Horizontal Kesten channels in a very good box.}
		\label{fig:horizontal-kesten-channels}
	\end{figure}
	
	\begin{figure}
		\centering
		\includegraphics[width=0.5\textwidth]{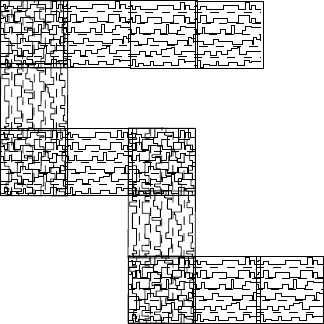}
		\caption{Part of an infinite path in the coarsened grid. We use the fact that each infinite path in the coarsened grid corresponds to a large number of disjoint paths 
			to infinity in the microscopic cluster $\mathscr{C}_{\infty}$.}
		\label{fig:kesten-channels-infinite}
	\end{figure}
	
	The lemma below, which is essentially a consequence of the main result of~\cite{domination-by-product-measures} shows that if $R$ is chosen sufficiently large, then the renormalized percolation process stochastically dominates a (very) supercritical i.i.d. Bernoulli site percolation on $\Z^2$. In particular, this implies that it contains a unique infinite cluster which we denote by $\mathcal{C}_{\infty}$ (the uniqueness of the infinite cluster is not a monotone property, and thus does not follow immediately from the stochastic domination, but the existence of circuits in annuli as in Lemma~\ref{lemma:loops-around-annuli} is monotone and, combined with planarity, guarantees the uniqueness of the infinite cluster).

	\begin{lemma} \label{lemma:renormalized-domination}
		There exist universal constants $C < \infty$ and $c > 0$ and an exponent $s > 0$ so that for any $R \geq 1$, the renormalized percolation process of size $R$ stochastically dominates a site percolation process with probability $p := 1 - C \exp \left( - c R^s \right)$ on the renormalized lattice.
	\end{lemma}
	\begin{proof} 
		The lemma is a direct consequence of~\cite[Theorem 1.3]{domination-by-product-measures} (with the value $\Delta = 4$ which is the degree of the lattice $\Z^2$), the inequality~\eqref{eq:verygoodstretchedexpclose1}, and the observation that the renormalized percolation process is $1$-dependent (in particular, it satisfies the assumption~\cite[(1.0)]{domination-by-product-measures} with $p = \P \left[ Q_R \mbox{ is very good} \right]$).
	\end{proof}
	
	We have now collected all the preliminary results and ingredients necessary to establish Proposition~\ref{prop:disjoint-paths-to-infinty-exponential}.
	
	\begin{proof}[Proof of Proposition~\ref{prop:disjoint-paths-to-infinty-exponential}]
		Since the constants $C$ and $c$ can be chosen depending on $f$, it is sufficient prove the result under the assumption the $R \geq R_0 \vee (10 (C_f + 1))^2$, where $R_0$ is chosen universally so that it satisfies two conditions: first $\sqrt{R_0} \leq R_0/10$ and second, using Lemma~\ref{lemma:renormalized-domination}, for any $R \geq R_0$, the renormalized site percolation process of size $R$ stochastically dominates a supercritical i.i.d. Bernoulli site percolation on $\mathbb{Z}^2$ (with, for instance, parameter $p = 3/4$ which is supercritical for two dimensional site percolation). 
		
		We next consider the renormalized percolation process of size $R$, and denote its infinite cluster by $\mathcal{C}_\infty$. We next claim that the following inclusion of events holds
		\begin{equation} \label{eq:implication-renormalized-cluster}
			(x + Q_R) \in \mathcal{C}_\infty \implies \, \forall y \in (x + Q_R), \mbox{ the event } E \left( y \right) \mbox{ holds}.
		\end{equation}
		To prove the previous implication, we observe that by definition of Kesten channels, the faces of adjacent very good cubes are connected by at least $R^{1/2}$ disjoint paths --- see Figure~\ref{fig:kesten-channels-infinite}. Moreover, adjacent Kesten channels are connected to each other. In the case when there is a corner, we use planarity to link together horizontal and vertical channels, ordering them in such a way that no intersections of the paths occur.
		
		By the observation in the previous paragraph, each path to infinity on the renormalized lattice corresponds to at least $\sqrt{R}_0 \geq 10(C_f + 1)$ disjoint paths to infinity in $\mathscr{C}_{\infty}$. In fact, one could be more specific with the previous statement: for any integer $K \in \{ 1 , \ldots, \lfloor R^{\frac{1}{2}} \rfloor \}$, if we let $\mathcal{R}$ be an horizontal or vertical rectangle of long side length $R$ and short side length $K R^{\frac{1}{2}}$ included in the box $(x + Q_{11R/10})$, then $\mathcal{R}$ is connected to infinity by at least $(K-1)$ disjoint paths. Additionally, these paths cross the rectangle $\mathcal{R}$ in the long direction.

		To complete the proof of~\eqref{eq:implication-renormalized-cluster}, we observe that every connected set $S$ on the microscopic lattice of diameter larger than $R$ which intersects $(x + Q_{R})$ must cross a rectangle with short side length $R/10$ and long side length $R$ included in $(x + Q_{11R/10})$, and is thus connected to infinity by at at least $(C_f + 1)$ disjoint paths.
		
		To complete the proof of Proposition~\ref{prop:disjoint-paths-to-infinty-exponential}, we use the inclusion~\eqref{eq:implication-renormalized-cluster} and Lemma~\ref{lemma:renormalized-domination} to write, for any $y \in \Z^2$,
		\begin{equation*}
			\P \left(  E \left( y \right) \right) \geq \theta_{\mathrm{site}} \left( 1 - C \exp \left( - c R^s \right) \right),
		\end{equation*}
		where $\theta_{\mathrm{site}}(p) \in [0,1]$ denotes the density of the infinite cluster for an i.i.d. site percolation with probability $p \in [0 ,1 ]$ of the lattice $\Z^2$. Using that the function $p \to \theta_{\mathrm{site}}(p)$ is smooth around the value $p = 1$ (see~\cite[Theorem 8.92]{percolation-book}), we see that there exists a (universal) constant $C < \infty$ such that, for any $p \in [0,1]$
		\begin{equation*}
			\theta_{\mathrm{site}}(p) \geq 1-C( 1 - p).
		\end{equation*}
		Combining the two previous displays completes the proof of~\eqref{eq:verygoodstretchedexpclose1}.
	\end{proof}
	
	We finally complete this section by proving a refinement of Proposition~\ref{prop:disjoint-paths-to-infinty-exponential}. Specifically, we establish that since the renormalized percolation stochastically dominates a (very) supercritical site Bernoulli percolation, any path must intersect the infinite cluster at least a fraction of the length between the two endpoints of the path. The statement is contained in the following proposition, and can be compared to~\cite[Lemma 5.3]{mathieu-homogenization}.
	
	\begin{prop}\label{prop:longpathsupercritrenormlatt}
		There exists a constant $R_1(\mathfrak{p},f) < \infty$ such that for any $R \geq R_1$, the following holds. There exist constants $C( R) < \infty$, $c( R) > 0$, $c_1( R) > 0$ such that, for any $K \geq 1$,
		\begin{equation} \label{eq:longpathsupercitrenormlatt}
			\begin{aligned}
				&\P \left( \exists \gamma \subseteq \mathscr{C}_\infty \, \mbox{ with } \gamma(0) = 0,~  \left| \gamma(\mathrm{end}) \right| \geq K ~\mbox{and}~ \sum_{i = 1}^{|\gamma|} \indc_{E_R \left( \gamma(i) \right)} \leq c_1 | \gamma(\mathrm{end}) |  \right)\\
				& \leq C \exp \left( - c K \right).
			\end{aligned}
		\end{equation}
	\end{prop}

	\begin{proof}
		We first set $R_1 := R_0 \vee \left( 10 (C_f + 1) \right)^2$ where $R_0$ is the constant in the proof of Proposition~\ref{prop:disjoint-paths-to-infinty-exponential} which ensures that, for any $R \geq R_0$, the renormalized percolation process of size $R$ stochastically dominates a i.i.d. Bernoulli site percolation on $\Z^2$ with probability $p = 3/4$ (hence supercritical).
		
		We next fix $R \geq R_1$ and consider the renormalized percolation process of size $R$. We then introduce the notions of distance and rectangles in the renormalized lattice:
		\begin{itemize}
			\item Given two sites $(x + Q_R)$ and $(y + Q_R)$ with $x , y \in (2R + 1) \Z^2$, we denote by $$\dist_R ( (x + Q_R) , (y + Q_R) ) := |x - y|/(2R+1).$$
			\item For $N \in \N$, we define the box of side length $N$ in the renormalized lattice
			\begin{equation*}
				Q^R_N : = \left\{ x + Q_R \, : \, \frac{x}{2R +1} \in   [-N , N] \times [- N,  N] \right\}
			\end{equation*}
			as well as the rectangle of long side length $2N$ and short side length $N$ in the renormalized lattice according to the formula
			\begin{equation*}
				\mathrm{HR}^R_N := \left\{ x + Q_R \, : \, \frac{x}{2R +1} \in   [- 2N , 2 N] \times [0,  2N] \right\}. 
			\end{equation*}
			We similarly define vertical rectangles in the renormalized lattice
			\begin{equation*}
				\mathrm{VR}^R_N := \left\{ x + Q_R \, : \, \frac{x}{2R +1} \in   [0 ,  2N] \times [- 2N,  2N]   \right\}.
			\end{equation*}
		\end{itemize}
		We next make the two observations. First, to each path $\gamma$ in the infinite cluster, we can associate a path in the renormalized lattice obtained by listing all the sites of the form $(x + Q_R)$ with $x \in (2R+1)\Z^2$ visited by the path $\gamma$ and erasing the loops. We denote this path in the renormalized lattice by $[\gamma]$ and observe that, for some universal constant $C$,
		\begin{equation} \label{eq:firstconditiongammarenormalized}
			[\gamma](0) = Q_R \quad \mbox{and} \quad \dist_R ([\gamma](\mathrm{end}) , Q_R) \leq \frac{\left| \gamma(\mathrm{end})\right|}{CR}.
		\end{equation}
		Next, since the renormalized percolation process of size $R$ dominates a supercritical Bernoulli site percolation, it possesses Kesten channels with exponentially high probability. 
		That is, by \cite[Theorem 11.1]{kesten-book} and a union bound,  for any $N_0 \geq 1$,
		\begin{equation}
			\begin{aligned} \label{eq:kestenchannelsrenormalized}
				&\P (\mbox{for all $N \geq N_0$ there exists at least $c_2 N$} \\
				&\mbox{disjoint horizontal crossings of open sites in the rectangle $\mathrm{HR}^R_N $} \\
				&\mbox{for the renormalized percolation process of size $R$} )\\
				&\geq 1 - C \exp \left( - c N_0 \right).
			\end{aligned}
		\end{equation}
		A similar statement holds for vertical rectangles (with vertical crossings). We remark that the constants $C , c, c_2$ depend only on the probability of the stochastically dominated Bernoulli site percolation, and can thus be chosen universally (at the cost of increasing the value of $R_0$).
		
		We then combine the two previous remarks as follows. First, the conditions~\eqref{eq:firstconditiongammarenormalized} imply that we may find a  constant $C$ (depending only on $R$) such that $[\gamma](\mathrm{end})$ does not belong to the box $Q^{R}_{|\gamma(\mathrm{end})|/(CR)}$. Set $N := |\gamma(\mathrm{end})|/(CR)$. The previous observation implies that:
		\begin{itemize}
			\item Either the path $[\gamma]$ crosses in the vertical direction one of the horizontal rectangles 
			$$\mathrm{HR}^R_{N} ~~\mbox{or} ~~ - \mathrm{HR}^R_{N}.$$
			\item Or the path $[\gamma]$ crosses horizontally one of the vertical rectangles 
			$$ \mathrm{VR}^R_{N} ~~\mbox{or}~~ -\mathrm{VR}^R_{N}.$$
		\end{itemize}
		In any of each of the two cases, we may apply the inequality~\eqref{eq:kestenchannelsrenormalized} to obtain (allowing the constants $C , c, c_1$ to depend on $R$)
		\begin{align*}
			&\P \left( \exists \gamma \subseteq \mathscr{C}_\infty \, \mbox{ with } \gamma(0) = 0,~  \left| \gamma(\mathrm{end}) \right| \geq K ~\mbox{and}~ \sum_{i = 1}^{|[\gamma]|} \indc_{[\gamma](i) \in \mathcal{C}_\infty } \leq c_1 | \gamma(\mathrm{end}) |  \right) \\
			&\leq C \exp \left( - c K \right).
		\end{align*}
		Using the implication~\eqref{eq:implication-renormalized-cluster}, this implies the bound~\eqref{eq:longpathsupercitrenormlatt} and thus completes the proof of Proposition~\ref{prop:longpathsupercritrenormlatt}. 
	\end{proof}

	\subsubsection{Building the obstruction} \label{subsubsec:building-the-obstruction}
	We introduce the good event which we will use in the next subsection, together with the 
	topological results which we have just established, to prove Proposition~\ref{prop:level-set-of-harmonic-function}.
	
	Start by setting $R := R_1(f , \mathfrak{p}) < \infty$ where $R_1$ is the constant which appears in the statement of Proposition~\ref{prop:longpathsupercritrenormlatt}. For a parameter $K \geq 1$ and $x \in \Z^2$, we introduce the first good event $G_1(x , K)$ according to the formula
	\begin{equation} \label{eq:defG1}
		G_1(x , K)  := \left\{ \exists \gamma \subseteq \mathscr{C}_\infty \, \mbox{ with } \gamma(0) = x ,~  \left| \gamma(\mathrm{end}) \right| \geq K ~\mbox{and}~ \sum_{i = 1}^{|\gamma|} \indc_{E_R \left( \gamma(i) \right)} \geq c_1 | \gamma(\mathrm{end}) | \right\},  
	\end{equation}
	where $c_1$ is the constant which appears in the statement of Proposition~\ref{prop:longpathsupercritrenormlatt}. In particular, by Proposition~\ref{prop:longpathsupercritrenormlatt}, there exist constants $c > 0$ and $C < \infty$ depending only on $\mathfrak{p}$ and $f$ (as $R_1$ depends only on these parameters) such that
	\begin{equation} \label{eq:PG1goesto1}
		\P \left( G_1(x , K) \right) \geq 1 - C \exp \left( - c K \right).
	\end{equation}
	We then define a second good event $G_2(x , K)$ which controls the oscillation of the corrected plane,
	\begin{equation} \label{eq:corrector-oscillation}
		G_2(x,K) := \left\{ \forall R' \geq  K \, : \, \osc_{Q_{R'}(x) \cap \mathscr{C}_{\infty}} \ellp{e_1} \leq 2 \diam R' \right\}.
	\end{equation}
	Note that, by the oscillation estimate on the corrector (see Theorem~\ref{theorem:first-order-corrector}), we have the lower bound
	\begin{equation} \label{eq:PG2goesto1}
		\P \left( G_2(x,K) \right) \geq 1 - C \exp \left( - c K^s \right).
	\end{equation}
	We finally fix a constant $C_\chi > 0$ and define and consider a third good event $G_3 (x , C_{\chi})$ which is satisfied if the edge $(x, x+e_1)$ is open in the infinite cluster and if the corrector has
	gradient larger than $C_\chi$ across the edge $(x, x+e_1)$. It is formally defined as
	\begin{equation} \label{eq:not-lipschitz}
		G_3(x,  C_{\chi}) := \left\{  (\ellp{e_1}(x + e_1) - \ellp{e_1}(x) )\indc_{(x, x+e_1) \in E(\mathscr{C}_\infty)} \geq C_\chi  \right\}.
	\end{equation}
	Contrary to the two events $G_1(x , K)$ and $G_2(x,K)$, the probability of the event $G_3(x, C_{\chi})$ goes to $0$ as $C_{\chi}$ tends to infinity. The crucial property we require on the event $G_3(x , C_{\chi})$ is that its probability is strictly positive for any $C_{\chi} < \infty$. Theorem~\ref{theorem:lipschitz} and symmetry arguments ensure that this assumption can be made without loss of generality. Using the previous observations, we see that for any $C_\chi < \infty$, we may choose $K$ (depending on $C_\chi$) such that the event $G(x) = G_1(x, K) \cap G_2(x , K) \cap G_3(x,  C_{\chi})$ has positive probability. This property is collected in the following lemma.
	
	\begin{lemma} \label{lemma:good-event-lower-bound}
		For every $C_{\chi} \in [1 , \infty)$, there exists a constant $K_0(\mathfrak{p} , f , C_\chi) < \infty$ such that, for any $K \geq K_0$ and all $x \in \Z^2$, the event 
		\begin{equation} \label{eq:def-good-event}
			G(x, C_\chi , K) := G_1(x, K) \cap G_2(x, K) \cap G_3(x, C_{\chi})
		\end{equation}
		has probability bounded from below by a constant, $\P[G(x , C_\chi , K )] \geq c(\mathfrak{p} , C_\chi) > 0$.     
	\end{lemma}
	
	\begin{proof}
		By stationarity, it suffices to consider the case $x = 0$.  Let $C_{\chi}$ be given. 
		By~\eqref{eq:PG1goesto1} and~\eqref{eq:PG2goesto1}, we have that as $K$ goes to infinity, $\P[G_1(0 , K) \cap G_2(0 , K)]$ goes to one
		and by Proposition \ref{prop:not-lipschitz}, $\P[G_3(0 , C_\chi)] \geq c(\mathfrak{p}, C_{\chi}) > 0$. 
		
		Thus, by choosing $K$ sufficiently large depending on $\mathfrak{p}, f , C_{\chi}$,  we have that  $\P[G_1(0 , K) \cap G_2(0 , K)] \geq 1 - c(C_{\chi})/2$. 
		Hence, a union bound implies the claim. \end{proof}

	This, together with the ergodic theorem, implies the following density result. In the following lemma, we recall the definition of the triangle $T_L := \{ x = (x_1 , x_2) \in \Z^2 \, : \, |x| \leq L, x_1 \geq |x_2|\}$.
	
	\begin{lemma} \label{lemma:density-bound-good-event}
		For every $C_{\chi} \in [1 , \infty)$, there exists a constant $K_0(\mathfrak{p} , f , C_\chi) < \infty$ such that, for any $K \geq K_0$, 
		\begin{equation*}
			\liminf_{L \to \infty} \frac{1}{L^2}\sum_{y \in T_L}1\{ {G(y , C_\chi , K)}\} =  \P[G(0 , C_\chi , K )].
		\end{equation*}
	\end{lemma}
	\begin{proof}
		By, for example, the proof of \cite[Lemma 3.2]{ergodic-theorem-lemma}, which combines the ergodic theorem together with translation invariance, 
		for every pair of Lipschitz domains $V,W \subset \R^2$ with $\cl(W) \subset V$, we have that 
		\begin{equation} \label{eq:spreads-evenly}
			\lim_{L \to \infty} \frac{1}{|L \cdot V|} \sum_{y \in (L \cdot W) \cap \Z^2} 1\{ {G(y , C_\chi , K)}\}  
			= \P( G(y , C_\chi , K)) \frac{|W|}{|V|}
			> 0
		\end{equation}
		where $L \cdot V$ denote the scaling of the domain by $L$. This implies the claim with $V = Q_1$ and $W = \{ x = (x_1 , x_2) \in Q_1 : x_1 \geq |x_2| \}$.
	\end{proof}

	\subsection{Exploring the block-cut tree}
	Building upon all the preliminary results established in the previous sections, we establish the following result, which implies Proposition~\ref{prop:level-set-of-harmonic-function}.
	\begin{lemma} \label{lemma:large-gradient-implies-not-in-level-set}
		There exist three constants $C_\chi (\mathfrak{p}, f) < \infty$, $K_0(\mathfrak{p} , f) < \infty$, and $C_0(\mathfrak{p} , f) < \infty$ and a random scale $M_{\mathrm{h}}$ which is almost surely finite such that the following results hold:
		\begin{itemize}
			\item[(i)] The event $G(0, C_\chi , K_0)$ defined in~\eqref{eq:def-good-event} has strictly positive probability.
			\item[(ii)] For any vertex $x_0 = (x_1 , x_2) \in \Z^2$ with $x_1 \geq |x_2|$ and $x_1 \geq M_{\mathrm{h}}$, if $G(x_0, C_\chi , K_0)$ holds, then there exists an edge $e \in E \left( \mathscr{C}_\infty \right)$ with $\dist(e , x) \leq C_0$ such that $ \nabla \ellp{e_1} (e) \neq 0$ and $\nabla u_f(e) \neq 0$.
		\end{itemize}
	\end{lemma}

	\begin{proof}[Proof of Proposition~\ref{prop:level-set-of-harmonic-function} assuming Lemma \ref{lemma:large-gradient-implies-not-in-level-set}]
		By Lemma~\ref{lemma:large-gradient-implies-not-in-level-set}, for almost every realization of the percolation configuration, there exists a mapping
		\begin{equation*}
			\Theta : \left\{ x \in \mathscr{C}_\infty \, : \, G(x) \, \mathrm{ holds}  \right\} \mapsto \left\{ e \in E \left( \mathscr{C}_\infty \right) \, : \, \nabla \ellp{p}(e) \neq 0 \mbox{ and } \nabla u_f(e) \neq 0 \right\}
		\end{equation*}
		satisfying the property that $\dist(\Theta (x) , x) \leq C_0$. This property implies that, for any edge $e \in E \left( \mathscr{C}_\infty \right)$ with $ \nabla \ellp{p}(e) \neq 0$ and $\nabla u_f(e) \neq 0 $, the cardinality of the collection of vertices $\Theta^{-1}(e)$ is bounded by $(2C_0)^2$. By Lemma~\ref{lemma:density-bound-good-event}, we see that, for all $L$ sufficiently large, there are at least $\frac{c |T_L|}{(2C_0)^2}$ edges $e$ in the triangle $T_L$ such that such that $ \nabla \ellp{e_1} (e) \neq 0$ and $\nabla u_f(e) \neq 0$, which implies the claim.
	\end{proof}

	\begin{proof}[Proof of Lemma \ref{lemma:large-gradient-implies-not-in-level-set}]
		We split the proof into several steps. In the first step, we collect several preliminary objects 
		and construct a particular subgraph of the level set of the potential $u_f$ around a good edge, $\mathcal{G}$. This subgraph
		has the important property that every site in $\mathcal{G}$ is connected to the good edge along a path
		upon which the corrected plane $\ellp{p}$
		is strictly increasing. We use this ordering to construct, in Step 2, 
		a rooted block-cut tree of $\mathcal{G}$. 
		
		We seek to use this tree and regularity properties of the corrected plane to find a nearby sensitive edge. 
		Our strategy is to show that if the nearest sensitive edge is far away, then either some topological obstruction
		is violated or the corrected plane grows too quickly. 
		
		To that end, in Step 3, we define the flux of the corrected plane through cut-vertices.
		There we establish a basic identity on the flux through cut-vertices of the tree. 
		In Step 4, we show that the flux bounds the growth of the corrected plane through branches
		of the block-cut tree. 
		
		We use this lower bound in Step 5 to construct an exploration of the block-cut tree which 
		maximizes the growth of the corrected-plane. Finally in Step 6, we show that this exploration 
		must terminate at a nearby sensitive edge.
		\medskip
		
		{\it Step 1: Construction of the set $\mathcal{G}$.} \\
		We fix an integer $r$ such that $\supp f \subseteq Q_r$ (the integer $r$ may be chosen depending only on $f$). We first define the negative line $L_- := \left\{ (-x , 0) \, : \, x \in \N  \right\}$ and define the minimal scale $M_h$ as follows:
		\begin{equation} \label{eq:def-Mh}
			M_h := \inf \left\{ R' > 0 \, : \, \forall x \in \mathscr{C}_\infty \, \mbox{with} \, x_1 \geq |x_2| ~\mbox{and}~ x_1 \geq R', \,  \ellp{e_1}(x) > \sup_{y \in \mathscr{C}_\infty \cap ( Q_r \cup L_-)} \ellp{e_1} (y)  \right\}.
		\end{equation}
		The random variable $M_h $ is almost surely finite due to the facts that the maps $x \mapsto e_1 \cdot x$ converges to $+ \infty$ (linearly) as $|x| \to \infty$ with $x \in T_L$, converges to $-\infty$ (linearly) as $|x| \to \infty$ with $x \in L_-$ and Theorem~\ref{theorem:first-order-corrector}.

		We select the value of the constant $C_\chi$ as follows. Recall the definition of the (sufficiently large) constant $R$ selected at the beginning of Section~\ref{subsubsec:building-the-obstruction}, the constant $c_1$ which appears in Proposition~\ref{prop:longpathsupercritrenormlatt}, and let $C_{\mathrm{branches}} := 2R^3$ (we introduce a specific notation for this constant as it bounds the number of branches of the block-cut tree $\mathcal{G}_{\mathrm{tree}}$ defined below, see Lemma~\ref{lemma:finite-branches}). We then set
		\begin{equation*}
			C_\chi := \frac{100 R^2 \exp \left( C_{\mathrm{branches}} \right)}{c_1},
		\end{equation*}
		and note that the constant $C_\chi$ depends only on $\mathfrak{p}$ and $f$.
		We then select the constant $K_0 \geq 1$ (depending only on $\mathfrak{p}$, $f$, and $C_\chi$, hence only on $\mathfrak{p}$ and $f$) according to Lemma~\ref{lemma:good-event-lower-bound}. We finally define $C_0 := \frac{2 R^2}{c_1} \vee K_0$.
		
		We next select a vertex $x_0 \in \mathscr{C}_\infty$ satisfying the assumptions of (ii) in the statement of Lemma~\ref{lemma:large-gradient-implies-not-in-level-set} (in particular the event $G(x_0, C_\chi,K_0)$ holds) and define $e = (x_0 , x_0 + e_1) \in E\left( \mathscr{C}_\infty \right).$ We will prove that there exists an edge $e' \in B_{C_0}(x_0)$ such that 
		\begin{equation} \label{eq:sensitivityprop}
			\nabla \ellp{e_1} (e') \neq 0 ~~\mbox{and}~~ \nabla u_f(e') \neq 0.
		\end{equation}
		To this end, we first know that, by the definition of the event $G(x_0, C_\chi,K_0)$, we have the identity $\ellp{e_1}(x_0 + e_1) - \ellp{e_1}(x_0) \geq C_\chi > 0$. We thus first check if $u_f(x_0+e_1) - u_f(x_0) \neq 0$ or if $u_f(x_0+e_1) - u_f(x_0) = 0$. In the first case, the edge $e = (x_0 ,x_0 +  e_1)$ satisfies the conclusion of the lemma and we are done. In the second case, we denote by $a := u_f(x_0+e_1) = u_f(x_0)$ the common value. We then consider the subset $\mathcal{G}$ of the level set $L_a$ which is composed of vertices $y \in L_a$ which are connected by a path $\gamma \subseteq L_a$ to $x_0$ along which the corrected plane is strictly increasing. This set of vertices of the graph is formally defined as follows
		\begin{multline*}
			\mathcal{G} := \{ y \in L_a \, : \, \exists n \geq 1, \, \exists \gamma: [ 0 , \ldots, n ] \to L_a, \, \gamma(0) = x_0, \gamma(1) = x_0+e_1, \gamma(N) = y \\ ~\mbox{and}~ \forall i \in [ 1 , \ldots, n-1], \, \ellp{e_1} (\gamma(i+1)) - \ellp{e_1} (\gamma(i)) > 0 \}.
		\end{multline*}
		Note that with this definition, the vertex $x_0$ is not included in $\mathcal{G}.$ The edge set of the graph is defined to be the set of edges in the infinite cluster along which the corrected plane is not constant. It is formally defined as follows
		\begin{equation*}
			E \left( \mathcal{G} \right) := \left\{ (x , y) \in E \left( \mathscr{C}_\infty \right) \, : x, y\in \mathcal{G}, \, \ellp{e_1}(x) \neq \ellp{e_1}(y) \right\}.
		\end{equation*}
		We first collect a few properties of the graph $\mathcal{G}$:
		\begin{itemize}
			\item The graph $\mathcal{G}$ is connected and may be unbounded (in that case, we must have $a = 0$);
			\item For any vertex $z \in \mathscr{C}_\infty \setminus \mathcal{G}$ which is a neighbor of a vertex in $y \in \mathcal{G}$, we have the dichotomy: either $z \notin L_a$ or $\ellp{e_1}(z) \leq \ellp{e_1} (y)$;
			\item The set $Q_r$ does not intersect the set $\mathcal{G}$ nor any bounded connected component of $\Z^2 \setminus \mathcal{G}$. Indeed, if it were the case, then the set $\mathcal{G}$ would have to intersect either the set $Q_r$ or the line $L_-$. This behavior is ruled out by the definition of the minimal scale $M_h$, which implies that 
			$$\ellp{e_1}(x_0) > \sup_{y \in \mathscr{C}_\infty \cap ( Q_r \cup L_-)} \ellp{e_1} (y),$$
			together with the observation that, for any $x \in \mathcal{G}$, $\ellp{e_1}(x) \geq \ellp{e_1}(x_0)$ (this last point is a direct consequence of the definition of the graph $\mathcal{G}$);
			\item Any finite connected component of $\mathscr{C}_\infty \setminus \mathcal{G}$ is included in the level set $L_a$ (this is obtained by observing that the function $u_f$ is harmonic on these connnected components, and that their boundary is included in the set $\mathcal{G}$ hence in the level set $L_a$);
			\item Any biconnected component of the graph $\mathcal{G}$ is finite (this result is a consequence of Proposition~\ref{prop:boundary-size-of-set} and the point above);
		\end{itemize}
		\medskip
		
		{\it Step 2: Construction of the block-cut tree of $\mathcal{G}$.} \\ 
		Denote the block-cut tree of the graph $\mathcal{G}$ by $\mathcal{G}_{\mathrm{tree}, 0}$. From the properties mentioned above, all the vertices of the block-cut tree $\mathcal{G}_{\mathrm{tree},0}$ have a finite degree (since all the biconnected components of $\mathcal{G}$ are finite). We will generically denote by $[x]$ the biconnected components of $\mathcal{G}$ (which are thus vertices of $\mathcal{G}_{\mathrm{tree},0}$) and by $x_{\mathrm{cut}}$ the cut-vertices of~$\mathcal{G}$ (which are also vertices of $\mathcal{G}_{\mathrm{tree},0}$). The biconnected component of $\mathcal{G}$ containing the initial vertex $x_0+e_1$ is defined to be the root of the tree $\mathcal{G}_{\mathrm{tree},0}$ (and we choose arbitrarily if $x_0 + e_1$ is a cut-vertex and belongs more than one biconnected component). 
		
		We next claim that, any leaf of this tree is a biconnected component of the graph $\mathcal{G}$, and that on the boundary of this biconnected component, there is an edge $e$ satisfying the property~\eqref{eq:sensitivityprop}. Indeed if we denote this biconnected component by $[x]$, we may consider the maximum of the corrected plane on $[x]$ and denote this maximum by $y \in [x]$. The maximum exists (but is not necessarily unique) since the biconnected component is finite, it is also distinct from the cut-vertex which is the parent of $[x]$ in the block-cut tree $\mathcal{G}_{\mathrm{tree},0}$ (since the definition of $\mathcal{G}$ implies that the corrected plane attains its minimum over the biconnected component $[x]$ exactly at this cutvertex). Using the definition of the graph $\mathcal{G}$, the fact that the corrected plane is harmonic on the percolation cluster and the maximum principle, we see that there exists a vertex $z \in \mathscr{C}_\infty$ such that $z \sim y$ with $\ellp{e_1} (z) > \ellp{e_1} (y)$. Then either $z \notin L_a$ (in that case, the edge $e = (y , z)$ satisfies~\eqref{eq:sensitivityprop}), or $z \in L_a$, in that case, the definition of the graph $\mathcal{G}$ implies that $z \in \mathcal{G}$. Since $z \notin [x]$ and since $[x]$ is a leaf, this implies that there is a loop in the tree $\mathcal{G}_{\mathrm{tree},0}$ which is a contradiction.
		
		We next modify the block-cut tree $\mathcal{G}_{\mathrm{tree},0}$ by removing some vertices according to the following procedure: for any biconnected component $[x] \in \mathcal{G}_{\mathrm{tree},0}$, if there exists a pair of vertices $z \in \mathscr{C}_\infty \setminus L_a$ and $y \in [x]$ such that $z \sim y$ and $\ellp{e_1}(z) \neq \ellp{e_1}(y)$, then we erase all the descendants of $[x]$ in the tree $\mathcal{G}_{\mathrm{tree},0}$. We denote the graph obtained through this procedure by $\mathcal{G}_{\mathrm{tree}}$. We note that the tree $\mathcal{G}_{\mathrm{tree}}$ satisfies the same property as the tree $\mathcal{G}_{\mathrm{tree},0}$: any leaf of $\mathcal{G}_{\mathrm{tree}}$ is a biconnected component and on the boundary of this biconnected component, there is an edge $e$ satisfying the sensitivity property~\eqref{eq:sensitivityprop}.
		\medskip
		
		{\it Step 3: Flux through cut-vertices.} \\ 
		For each cut-vertex $x_{\mathrm{cut}} \in \mathcal{G}_{\mathrm{tree}}$, we define the incoming and outgoing fluxes of the corrected plane as follows. Given a cut-vertex $x_{\mathrm{cut}}$, we denote by $[x_0] \in \mathcal{G}_{\mathrm{tree}}$ the biconnected component which is the parent of $x_{\mathrm{cut}} \in \mathcal{G}_{\mathrm{tree}}$ in the biconnected tree $\mathcal{G}_{\mathrm{tree}}$. We define the incoming flux of the cut-vertex $x_{\mathrm{cut}}$ as follows
		\begin{equation*}
			i_{\mathrm{in}}(x_{\mathrm{cut}}) := \sum_{\substack{y \in [x_0]  \\ y \sim x_{\mathrm{cut}}}} (\ellp{e_1}(x_{\mathrm{cut}}) - \ellp{e_1}(y)). 
		\end{equation*}
		We similarly define the outgoing flux of the corrected plane at the cut-vertex $x_{\mathrm{cut}}$: if we let $[x_1], \ldots, [x_n]$ be the biconnected components which are the children of  $x_{\mathrm{cut}}$ in $\mathcal{G}_{\mathrm{tree}}$, then, for any $i \in \{ 1 , \ldots , n\}$, we define
		\begin{equation*}
			i_{\mathrm{out}}(x_{\mathrm{cut}}, [x_i]) := \sum_{\substack{ y \in [x_i] \\ y \sim x_{\mathrm{cut}}}} (\ellp{e_1}(y) - \ellp{e_1}(x_{\mathrm{cut}})).
		\end{equation*}
		In this step we prove the following result: if we consider a biconnected component $[x] \in \mathcal{G}_{\mathrm{tree}}$ (which is not the root nor a leaf), denote by $x_{\mathrm{cut},0}$ the cut-vertex which is its parent, and by $x_{\mathrm{cut},1} , \ldots, x_{\mathrm{cut},n}$ the cut-vertices which are its children, then we have the inequalities
		\begin{equation} \label{eq:fluxincreasesthroughbiconnectedcomponents}
			\sum_{i = 1}^n i_{\mathrm{in}}(x_{\mathrm{cut} , i}) \geq i_{\mathrm{out}}(x_{\mathrm{cut} , 0}, [x]) 
		\end{equation}
		and if we denote by $[x], [x_1], \ldots, [x_m]$ the children of $x_{\mathrm{cut},0}$ (since we are in two dimensions, we have that $m \leq 3$), then we have
		\begin{equation}\label{eq:fluxincreasesthroughbiconnectedcomponents2}
			i_{\mathrm{out}}(x_{\mathrm{cut} , 0}, [x]) + \sum_{i=1}^m i_{\mathrm{out}}(x_{\mathrm{cut} , 0}, [x_i]) \geq i_{\mathrm{in}}(x_{\mathrm{cut} , 0}).
		\end{equation}
		To prove~\eqref{eq:fluxincreasesthroughbiconnectedcomponents}, we denote by $X := [x] \setminus \{ x_{\mathrm{cut},n}, x_{\mathrm{cut},1} , \ldots, x_{\mathrm{cut},n} \}$ the biconnected component after removing the cut-vertices. We then use the harmonicity of the corrected plane on the infinite cluster, perform a discrete integration by parts and obtain
		\begin{equation} \label{eq:flux0}
			0 = \sum_{y \in X} \Delta_{\mathscr{C}_\infty} \ellp{e_1}(y) = \sum_{\substack{z \in \mathscr{C}_\infty \setminus  X, y \in [x] \\ z \sim y}} \left( \ellp{e_1}(z) - \ellp{e_1}(y) \right).
		\end{equation}
		The sum on the right-hand side can be decomposed into two terms, depending on whether the vertex $z$ belongs to the graph $\mathcal{G}$ or not. In the case where the vertex $z$ belongs to the graph $\mathcal{G}$, we use the definitions of the incoming flux of a cut-vertex and obtain the identity
		\begin{equation} \label{eq:flux1}
			\sum_{\substack{z \in \mathcal{G} \setminus X, y \in X \\ z \sim y}} \left( \ellp{e_1}(z) - \ellp{e_1}(y) \right) = \sum_{i = 1}^n i_{\mathrm{in}}(x_{\mathrm{cut} , i}) - i_{\mathrm{out}}(x_{\mathrm{cut} , 0}, [x]) .
		\end{equation}
		For the pair of vertices $z$ and $y$ satisfying $z \sim y$, $y \in [x]$, and $z \notin \mathcal{G}$, we observe that, since $[x]$ is not a leaf of the tree $\mathcal{G}_{\mathrm{tree}}$, we must that $\ellp{e_1} (z) \leq \ellp{e_1}(y)$. Indeed, if we had $\ellp{e_1} (z) > \ellp{e_1}(y)$ then either $z \in L_a$ and the definition of the graph $\mathcal{G}$ would imply $z \in \mathcal{G}$ which is a contradiction, or $z \notin L_a$ and in that case the edge $(y,z)$ would satisfy~\eqref{eq:sensitivityprop}, and $[x]$ would thus be a leaf of $\mathcal{G}_{\mathrm{tree}}$ (by construction), which is also a contradiction. Collecting the previous observations, we see that
		\begin{equation} \label{eq:flux2}
			\sup_{\substack{z \in \mathscr{C}_\infty \setminus \mathcal{G} , y \in [x] \\ z \sim y}} \left( \ellp{e_1}(z) - \ellp{e_1}(y) \right) \leq 0.
		\end{equation}
		Combining~\eqref{eq:flux0},~\eqref{eq:flux1} and~\eqref{eq:flux2} completes the proof of~\eqref{eq:fluxincreasesthroughbiconnectedcomponents}. 
		Similarly, we use the definitions of the incoming and outgoing flux and the inequality~\eqref{eq:flux2} to obtain 
		\begin{align*}
			0 = \Delta_{\mathscr{C}_\infty} \ellp{e_1}(x_{\mathrm{cut} , 0}) & =  i_{\mathrm{out}}(x_{\mathrm{cut} , 0} , [x] ) - i_{\mathrm{in}}(x_{\mathrm{cut} , 0}) + \sum_{\substack{z \in \mathscr{C}_\infty \setminus \mathcal{G} \\ 
					z \sim x_{\mathrm{cut} , 0}}} \left( \ellp{e_1}(z) - \ellp{e_1}(x_{\mathrm{cut} , 0}) \right) \\
			& \leq i_{\mathrm{out}}(x_{\mathrm{cut} , 0}, [x]) - i_{\mathrm{in}}(x_{\mathrm{cut} , 0}),
		\end{align*}
		to prove \eqref{eq:fluxincreasesthroughbiconnectedcomponents2}, completing this step. 
		\medskip
		
		{\it Step 4: Regularity of the block-cut tree $\mathcal{G}_{\mathrm{tree}}$.} \\
		In this step, we prove the following properties on the block-cut tree $\mathcal{G}_{\mathrm{tree}}$:
		\begin{enumerate}
			\item[(i)] The number of (potentially infinite) branches of the tree $\mathcal{G}_{\mathrm{tree}}$ is bounded by a constant $C_{\mathrm{branches}}$ depending only on the support of $f$.
			\item[(ii)] If we let $[x]$ be a biconnected component of $\mathcal{G}_{\mathrm{tree}}$ (which is neither the root nor one of the leaves), denote by $x_{\mathrm{cut} , 0}$ its parent in $\mathcal{G}_{\mathrm{tree}}$ and $x_{\mathrm{cut} , 1}$ one of its children in $\mathcal{G}_{\mathrm{tree}}$ 
			then we have the inequality
			\begin{equation*}
				\ellp{e_1} (x_{\mathrm{cut} , 1}) \geq \ellp{e_1} (x_{\mathrm{cut} , 0}) + \frac{i_{\mathrm{in}}(x_{\mathrm{cut} , 1}) }{4}.
			\end{equation*}
		\end{enumerate}
		We defer the proof of (i) to Lemma \ref{lemma:finite-branches} below. To prove (ii), fix a biconnected component  $[x] \in \mathcal{G}_{\mathrm{tree}}$ (which is neither the root nor a leaf), let $x_{\mathrm{cut} , 0}$ be its parent in $\mathcal{G}_{\mathrm{tree}}$ and $x_{\mathrm{cut} , 1}$ one of its children. Using the definition of the incoming flux $i_{\mathrm{in}}(x_{\mathrm{cut} , 1})$ and the observation that any vertex has at most $4$ neighbors (in $2$ dimensions), we deduce that there exists $y \in [x]$ with $y \sim x_{\mathrm{cut} , 1}$ such that
		\begin{equation*}
			\ellp{e_1}(x_{\mathrm{cut} , 1}) - \ellp{e_1} (y) \geq \frac{i_{\mathrm{in}}(x_{\mathrm{cut} , 1}) }{4}.
		\end{equation*}
		Next using the definition of the graph $\mathcal{G}$, we know that there exists a path $\gamma$ going from the initial vertex $0$ to the vertex $y$ along which the corrected plane increases. Since $x_{\mathrm{cut}, 0}$ is a cut-vertex of the graph $\mathcal{G}$, the path $\gamma$ must pass through $x_{\mathrm{cut}, 0}$ in order to arrive to $y$. Since the corrected plane is increasing along $\gamma$, we deduce that
		\begin{equation} \label{eq:increasevaluefluxcorr}
			\ellp{e_1}(x_{\mathrm{cut} , 1}) \geq \ellp{e_1} (y) + \frac{i_{\mathrm{in}}(x_{\mathrm{cut} , 1}) }{4} \geq \ellp{e_1} (x_{\mathrm{cut}, 0}) + \frac{i_{\mathrm{in}}(x_{\mathrm{cut} , 1}) }{4},
		\end{equation}
		which is the desired inequality.

		{\it Step 5: Exploring a path of large flux in the block-cut tree} \\
		We have now established all the preliminary ingredients to construct a short path in the block-cut tree which ends at a sensitive edge. This path is constructed
		via an exploration of the block-cut tree, which we show (in the next step) has to end quickly. 
		
		Denote by $[x_{\mathrm{root}}]$ the root of $\mathcal{G}_{\mathrm{tree}}$ (\ie, the biconnected component of $\mathcal{G}_{\mathrm{tree}}$ containing the edge $x +  e_1$). We then let $\mathcal{L}$ be the collections of all the leafs of the tree $\mathcal{G}_{\mathrm{tree}}$ together with the root $[x_{\mathrm{root}}]$. By point (i) in Step 4, we have the following identity
		\begin{equation} \label{eq:bound.generaltree}
			(\deg ([x_{\mathrm{root}}]) - 1) + \sum_{v \in \mathcal{G}_{\mathrm{tree}} \setminus \mathcal{L}} (\deg (v) - 2) \leq C_{\mathrm{branches}}.
		\end{equation}
		Let us denote by $x_{\mathrm{root}, 1}, \ldots, x_{\mathrm{root}, \deg ([x_{\mathrm{root}}])}$ the cut-vertices which are the children of the root $[x_{\mathrm{root}}]$ in the tree $\mathcal{G}_{\mathrm{tree}}$. Using the assumption $\ellp{e_1} (e_1) - \ellp{e_1} (0) \geq C_\chi$, we have that
		\begin{equation*}
			\sum_{i = 1}^{\deg ([x_{\mathrm{root}}])} i_{\mathrm{in}}(x_{\mathrm{root},i}) \geq C_{\chi}.
		\end{equation*}
		The previous inequality implies that there exists a cutpoint, which we may choose without loss of generality to be $x_{\mathrm{root},1}$ such that
		\begin{equation*}
			i_{\mathrm{in}}(x_{\mathrm{root},1}) \geq \frac{C_{\chi}}{\deg ([x_{\mathrm{root}}])}.
		\end{equation*}
		We then construct a path $[\gamma] \subseteq \mathcal{G}_{\mathrm{tree}}$ in the tree $\mathcal{G}_{\mathrm{tree}}$ so as to maximize the flux of the corrected plane along the path. It is defined according to the following iterative procedure:
		\begin{itemize}
			\item The path $[\gamma]$ starts from the root $[x_{\mathrm{root}}]$ and first visit the cut-vertex $x_{\mathrm{root},1}$;
			\item If $[\gamma]$ visits a cut-vertex $x_{\mathrm{cut}} \in \mathcal{G}_{\mathrm{tree}}$, we then select the biconnected component $[x]$ which maximises the outgoing flux $i_{\mathrm{out}}(x_{\mathrm{cut}} , [x])$ among all the biconnected components which are the children of $x_{\mathrm{cut}}$ (and use an arbitrary criterion to break ties). We then extend the path $[\gamma]$ so that it visits this biconnected component.
			\item If $[\gamma]$ visits a biconnected component  $[x] \in \mathcal{G}_{\mathrm{tree}}$, we select the cut-vertex $x_{\mathrm{cut}}$ which maximises the incoming flux $i_{\mathrm{in}}(x_{\mathrm{cut}})$ among all the cut-vertices which are the children of $[x]$ (and break ties using an arbitrary criterion). We then extend the path $[\gamma]$ so that it visits the cut-vertex $x_{\mathrm{cut}}$. If there is no such cutpoint, then the path $[\gamma]$ has reached a leaf of the tree $\mathcal{G}_{\mathrm{tree}}$ and we stop the iterative construction.
		\end{itemize}
		The path $[\gamma]$ is constructed so as to satisfy the following inequality: for any biconnected component $[x] \in \mathcal{G}_{\mathrm{tree}}$ which is neither the root nor a leaf, if we denote by $x_{\mathrm{cut},0}$ the parent of $[x]$ in $\mathcal{G}_{\mathrm{tree}}$ and $x_{\mathrm{cut},1}$ the child of $[x]$ which belongs to the path $[\gamma]$, then
		\begin{equation*}
			i_{\mathrm{in}} (x_{\mathrm{cut}, 1}) \geq \frac{i_{\mathrm{out}} (x_{\mathrm{cut}, 0}, [x])}{\deg ([x]) - 1}.
		\end{equation*}
		Similarly, for any cut-vertex $x_{\mathrm{cut}} \in \mathcal{G}_{\mathrm{tree}}$, if we denote by $[x]$ the child of $x_{\mathrm{cut}}$ which belongs to the path $[\gamma]$, then we have
		\begin{equation*}
			i_{\mathrm{out}} (x_{\mathrm{cut}} , [x])  \geq \frac{i_{\mathrm{in}}(x_{\mathrm{cut}})}{\deg (x_{\mathrm{cut}}) - 1}.
		\end{equation*}
		Combining the two previous inequalities, we obtain the following (crude) lower bound: for any cut-vertex $x_{\mathrm{cut}} \in [\gamma]$,
		\begin{equation*}
			i_{\mathrm{in}}(x_{\mathrm{cut}}) \geq \frac{C_{\chi}}{\deg([x_{\mathrm{root}}]) \prod_{v \in \mathcal{G}_{\mathrm{tree}} \setminus \mathcal{L}} (\deg(v) - 1)}.
		\end{equation*}
		Using the upper bound $e^{x - 1} \geq x$ valid for any $x \geq 1$ together with the lower bound~\eqref{eq:bound.generaltree}, we deduce that
		\begin{align} \label{eq:lowerboundiin1410}
			i_{\mathrm{in}}(x_{\mathrm{cut}}) & \geq C_{\chi} \exp \left( - ( \deg([x_{\mathrm{root}}]) - 1) - \sum_{v \in \mathcal{G}_{\mathrm{tree}} \setminus \mathcal{L}} (\deg(v) - 2) \right) \\
			& \geq C_{\chi} \exp \left( -C_{\mathrm{branches}} \right). \notag
		\end{align}

		{\it Step 6: The exploration cannot be too long.} \\
		We now conclude by showing the path $[\gamma]$ constructed in the previous step must be short. We distinguish between two cases depending on whether the path $[\gamma]$ is finite or infinite. 
		
		\medskip
		
		\textbf{Case 1: the path $[\gamma]$ is finite.} We let $N \in \N$ be the number of cut-vertices of $[\gamma]$, and denote these cut-vertices by $x_{\mathrm{cut} , 1}, \ldots , x_{\mathrm{cut} , N}$. We then denote by $[x_{\mathrm{root}}], [x_{1}], \ldots, [x_{N}]$ the biconnected components of $[\gamma]$ (since the cut-vertices and the biconnected components alternate along $[\gamma]$, there must be $N+1$ biconnected components). In particular, the path $[\gamma]$ can be written as follows $[\gamma] := ([x_{\mathrm{root}}], x_{\mathrm{cut}, 1}, [x_{1}], \ldots, x_{\mathrm{cut} , N}, [x_{N}])$.

		We then observe that the endpoint $ [x_{N}]$ of $[\gamma]$ must be a leaf of $\mathcal{G}_{\mathrm{tree}}$ (otherwise the iterative construction would continue). As mentioned above and by construction of the tree $\mathcal{G}_{\mathrm{tree}}$, there exists an edge $e = (y , z) \in E \left( \mathscr{C}_\infty\right)$ (with $y \in [x_N]$ and $z \notin [x_N]$) on the boundary of the biconnected component $[x_{N}]$ such that the sensitivity condition~\eqref{eq:sensitivityprop} is satisfied. We then consider a path $\gamma \subseteq \mathcal{G}$ connecting the edge $(x_0, x_0+e_1)$ to the edge $e$ along which the corrected plane is increasing (the existence of this path is guaranteed by the definition of the graph $\mathcal{G}$). Additionally, the path $\gamma$ must pass through the cut-vertices $x_{\mathrm{cut}, 1}, \ldots, x_{\mathrm{cut} , N}$. We may thus decompose the difference of the corrected plane between the vertices $x_0$ and $y$ as follows
		\begin{align*}
			\ellp{e_1}(y) - \ellp{e_1}(x_0) 
			& \geq (\ellp{e_1}(x_{\mathrm{cut}, 1}) - \ellp{e_1}(x_0)) + (\ellp{e_1}(y) - \ellp{e_1}(x_{\mathrm{cut}, N})) \\
			& \qquad +  \sum_{i = 1}^{N-1} \ellp{e_1}(x_{\mathrm{cut}, i+1}) - \ellp{e_1}(x_{\mathrm{cut}, i}) \\
			&\geq \sum_{i = 1}^{N-1} \ellp{e_1}(x_{\mathrm{cut}, i+1}) - \ellp{e_1}(x_{\mathrm{cut}, i}),
		\end{align*}
		where in the second inequality, we used that the first two terms $(\ellp{e_1}(x_{\mathrm{cut}, 1}) - \ellp{e_1}(x_0))$ and $(\ellp{e_1}(y) - \ellp{e_1}(x_{\mathrm{cut}, N}))$ are nonnegative since the corrected plane is increasing along the path~$\gamma$. We next deduce that
		\begin{align*}
			\ellp{e_1}(y) - \ellp{e_1}(x_0) & \geq \frac{1}{4}\sum_{i=1}^{N-1} i_{\mathrm{in}} (x_{\mathrm{cut}, i})  \quad \mbox{(by ~\eqref{eq:increasevaluefluxcorr})} \\
			& \geq \frac{(N-1)}{4} C_{\chi} \exp \left( -C_{\mathrm{branches}} \right) \quad \mbox{(by ~\eqref{eq:lowerboundiin1410})} .
		\end{align*}
		We next recall the formula for the constant $C_\chi$
		\begin{equation} \label{eq:choiceCchi}
			C_{\chi} := \frac{100 R^2 \exp \left( C_{\mathrm{branches}} \right)}{c_1},
		\end{equation}
		the definition of the constant $K_0$ and the one of the event $G(x_0 , C_\chi, K_0)$. We will show the following implication
		\begin{equation} \label{eq:comparingNandepsionl}  \mbox{the good event } G(x_0, C_\chi, K_0) \mbox{ holds and } |y - x_0| \geq C_0 ~ \implies ~  (N-1) \geq \frac{c_1 |y-x_0|}{8 R^2},
		\end{equation}
		which, together with the previous inequality allows us to conclude. Indeed, by the previous inequality and the definition of the good event $G(x_0 , C_\chi, K_0)$, we obtain that, if $|y-x_0| \geq C_0$, then
		\begin{equation*}
			2 |y-x_0| \geq  \ellp{e_1}(y) - \ellp{e_1}(x_0) \geq \frac{(N-1)}{4} C_{\chi} \exp \left( -C_{\mathrm{branches}} \right) \geq \frac{c_1 |y-x_0|}{32 R^2} C_{\chi} \exp \left( -C_{\mathrm{branches}} \right).
		\end{equation*}
		Using the definition of $C_\chi$, this implies
		\begin{equation*}
			2 |y-x_0|  \geq 3 |y-x_0|,
		\end{equation*}
		which is a contradiction (as $|y-x_0| \geq 1$). We have thus obtained that $|y-x_0| \leq C_0$ which is the desired conclusion.
		
		It only remains to prove the lower bound~\eqref{eq:comparingNandepsionl}. We start with the following observation: if $\left| \gamma \right| \geq 4R^2$ and if there exists an index $i \in \{ 1 , \ldots, \left|  \gamma \right| \}$ such that the event $E(\gamma(i))$ holds, then there exists an index $j \in \{ 1 , \ldots, \left|  \gamma \right| \}$ with $|i - j| \leq 4 R^2$ such that $\gamma(i)$ is a cut-vertex of $\mathcal{G}_{\mathrm{tree}}$. To justify this observation, let us assume that the event $E(\gamma(i))$ holds and that $\gamma (i)$ is not a cut-vertex (otherwise the claim is proved). We then denote by $[\gamma(i)]$ the biconnected component containing $\gamma(i)$ (this component is unique since $\gamma(i)$ is not a cut-vertex). If for any index $j \in \{ 1 , \ldots, \left|  \gamma \right| \}$ with $|i - j| \leq 4 R^2$, the vertex $\gamma(j)$ belongs to the biconnected component $[\gamma(i)]$, then this biconnected component would have a diameter larger than $R$ (as in two dimensions, the diameter of any connected set of cardinality larger than $4R^2$ is larger than $R$). Since the event $E(\gamma(i))$ holds, this implies that there are more than $C_f+1$ distinct paths going from the biconnected component $[\gamma(i)]$ to infinity, which yields a contradiction by Proposition~\ref{prop:boundary-size-of-set}.

		Using that the good event $G(x_0, C_\chi, K_0)$ holds, we can apply the definition of the event $G_1(x_0 , K_0)$ stated in~\eqref{eq:defG1} to the path $\gamma$ and obtain that
		\begin{equation*}
			\sum_{i = 1}^{|\gamma|} \indc_{E(\gamma(i))} \geq c_1 |y-x_0|.
		\end{equation*}
		Using the definition of the constant $C_0$, we see that $c_1 |y-x_0| \geq c_1 C_0 \geq 8 R^2$. This implies that $\left| \gamma \right| \geq 8R^2 \geq 4R^2$. We are thus able to use the previous observation to deduce that
		\begin{equation*}
			N \geq \frac{c_1 |y-x_0|}{4 R^2} \geq 2.
		\end{equation*}
		In particular, the previous inequality implies $N - 1 \geq N/2$ (as $N \geq 2$). The proof of the inequality~\eqref{eq:comparingNandepsionl} is complete.
		
		\medskip
		\textbf{Case 2: the path $[\gamma]$ is infinite.}
		In this case, we may find a point $y$ which belongs to the path $\gamma$ such that $|y| \geq C_0$ and $ \ellp{e_1}(y) - \ellp{e_1}(x_0) \geq \frac{(N-1)}{4} C_{\chi} \exp \left( -C_{\mathrm{branches}} \right)$ and then conclude in the exact same way as the finite case. \end{proof}

	\subsubsection{Bounding the number of branches in the block-cut tree}
	
	\begin{figure}
		\centering
		\includegraphics[width=\textwidth]{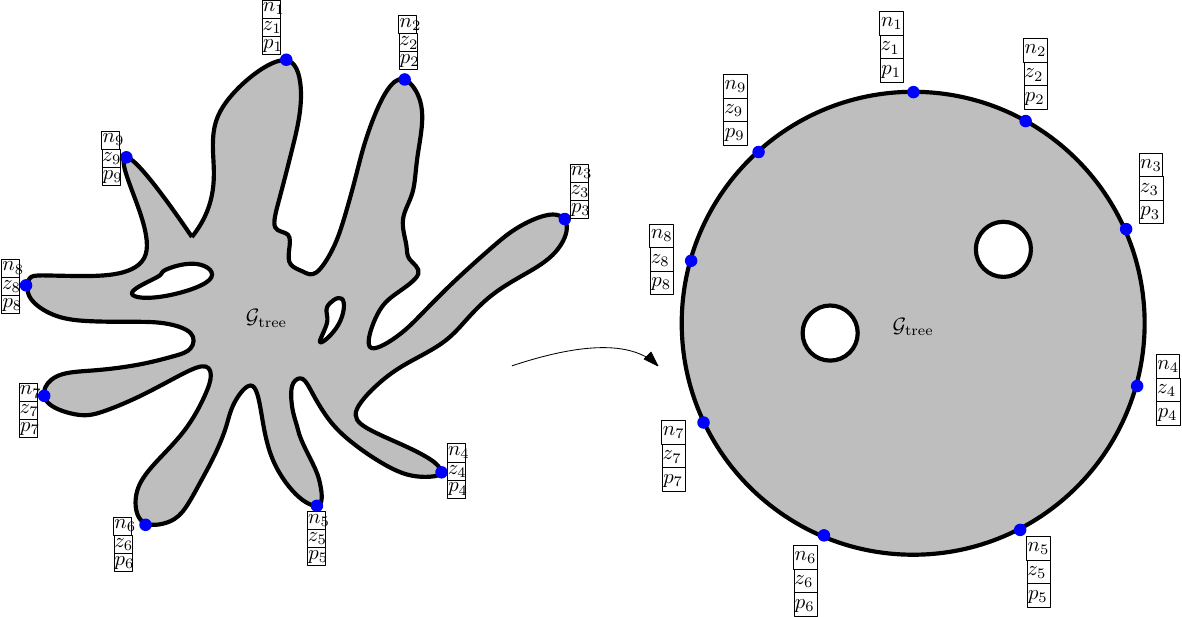} 
		\caption{Transformation of a finite block-cut tree along with points on its leaves.} \label{fig:block-cut-homeomorphism}
	\end{figure}
	
	\begin{figure}
		\centering
		\includegraphics[width=0.45\textwidth]{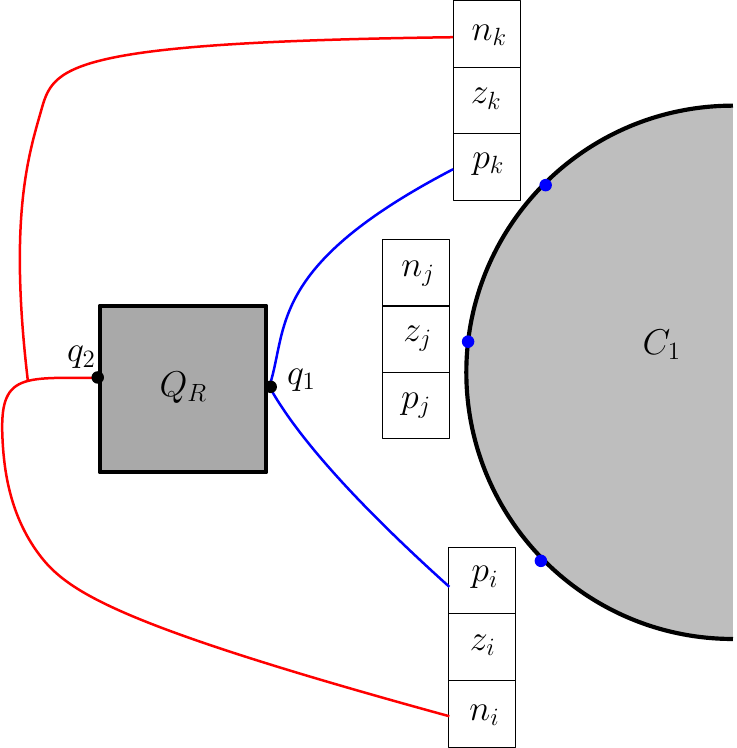} 
		\caption{If there are too many branches, a leaf in the block-cut tree will get trapped.}
		\label{fig:trap-block-cut-tree}
	\end{figure}

	\begin{figure}
		\centering
		\includegraphics[width=0.45\textwidth]{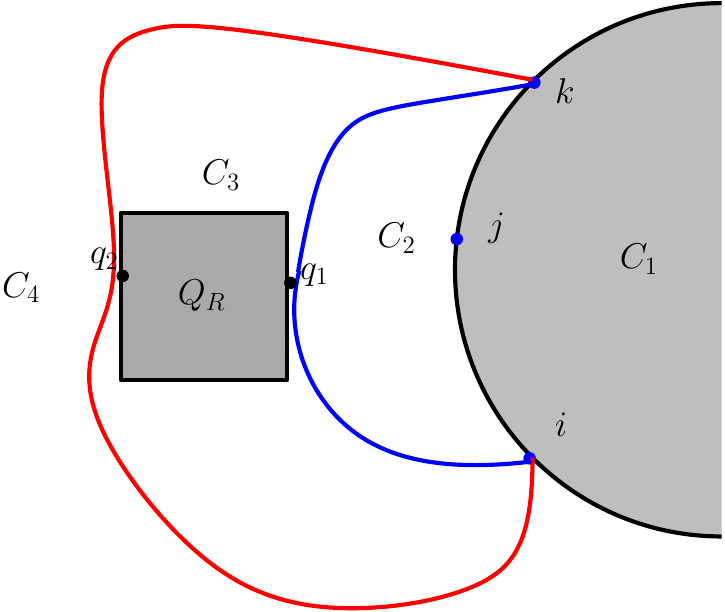} 
		\caption{The support of $f$ must lie in $C_3$, disconnecting it from $j$. }
		\label{fig:trap-block-cut-tree-loops}
	\end{figure}

	We use the objects and notation of the above proof and recall that we set $C_{\mathrm{branches}} := 2 R^3$ and that the constant $R$ is selected sufficiently large (and allowed to depend on $f$) so that $f = 0$ outside the box $Q_{R/2}$.
	
	\begin{lemma} \label{lemma:finite-branches}
		The tree $\mathcal{G}_{\mathrm{tree}}$ has at most $C_{\mathrm{branches}}$ branches.
	\end{lemma}
	\begin{proof}
		By Proposition~\ref{prop:infinite-disjoint-paths-limit}, the tree $\mathcal{G}_{\mathrm{tree}}$ cannot have more than $(\left| \supp f \right| + 2) \leq R^3$ infinite branches. It is thus sufficient to prove that the tree $\mathcal{G}_{\mathrm{tree}}$ cannot have more than $R^3$ finite branches (or equivalently, more than $R^3$ leafs).
		
		Let $Q_{R/2}$ be a box containing the support of $f$. Suppose for sake of contradiction there were more than
		$R^3$ finite branches in $\mathcal{G}_{\mathrm{tree}}$. As observed above, at the end of each finite branch 
		is an edge $(z_i, p_i)$ where $\nabla u_f((z_i, p_i)) \neq 0$. We may assume that $u_f(z_i) = a$ and $u_f(p_i) > a$. 
		Since $z_i$ is not in the support of $f$, $u_f$ is harmonic at $z_i$ and therefore there is a neighbor $n_i \sim z_i$
		for which $u_f(n_i) < a$.
		
		Now, take the tree and prune the infinite branches so that what is left is a block-cut tree 
		with only $R^3$ finite branches, we refer to this tree as $\mathcal{G}_{\mathrm{tree}, f}$. Perform this pruning so that the leaves at the end of the finite branches are still leaves of $\mathcal{G}_{\mathrm{tree}, f}$. We then denote by $\mathcal{G}_{f} \subseteq \Z^2$ the union of all the vertices contained in a biconnected component of $\mathcal{G}_{\mathrm{tree}, f}$.

		We next observe that the vertices $n_i$ and $p_i$ belong to the infinite complementary connected component of $\mathcal{G}_{f}$. 
		Indeed, it was observed above that any finite connected component of $\mathscr{C}_{\infty} \setminus \mathcal{G}$ is included in the level set.
		
		We then move to the continuum in order to apply Jordan curve theorem and let $C_1 \subseteq \R^2$ denote the interior of the complement of the infinite complementary connected component of the set $(\mathcal{G}_{f} + [-1/2 , 1/2]^2) \subseteq \R^2$. By the Jordan curve theorem, Lemma \ref{lemma:jordan-curve-theorem-on-s2}, there is a homeomorphism from $C_1$ to the unit disk -- see Figure \ref{fig:block-cut-homeomorphism}. We use this to order the points on the boundary of $C_1$ according to their angle on the unit disk.  Also, we have that $Q_R$ is disjoint from $C_1$.

		By the pigeonhole principle and Lemma \ref{lemma:poles-sub-super-level-sets}, there are three sites $(n_i, p_i, z_i)$, $(n_k, p_k, z_k)$ and $(n_j, p_j, z_j)$ so that both triples $(p_i, p_k, p_j)$ and $(n_i, n_k, n_j)$ are connected to the points $q_1$ and $q_2$ respectively on $\partial Q_R$. Assume that $i < k < j$.  This leads to either $p_k$ or $n_k$ being disconnected from $Q_R$, which is a contradiction --- see Figure \ref{fig:trap-block-cut-tree}.
		
		In particular, referring to Figure \ref{fig:trap-block-cut-tree-loops}, this construction leads to two loops, 
		one containing $(k, i, j, q_1)$ and another containing $(k, q_2, i)$. By the Jordan curve theorem, these loops 
		correspond to three connected components, $C_2, C_3, C_4$, which together with $C_1$ partition $\R^2$. By considering cases for the location of $Q_R$, as in the end of the proof of Lemma \ref{lemma:no-thick-paths}, we must have $Q_R$ be fully contained in $C_3$; however, this disconnects it from $j$. 
	\end{proof}

	\section{Abelian sandpile} \label{sec:abelian-sandpile}
	A sandpile is a function from the set of vertices of an undirected graph to the nonnegative integers. 
	The value of the function is to be thought of as the number of chips or grains at 
	that vertex. A vertex is {\it unstable} if it has at least as many chips as it has edges, 
	in which case it {\it topples}, giving one chip to each of the vertices it shares an edge with. 
	With appropriate boundary conditions, sandpiles can be {\it stabilized}, that is, unstable vertices topple until every vertex is stable
	and the order in which unstable vertices topple does not affect the final, stable configuration. 
	
	In this section, we consider two aspects of the Abelian sandpile model. First, we briefly exposit, in Section \ref{subsec:sandpile-limit-shape}, the large scale behavior of sandpiles on infinite graphs and explain the connection to Theorem \ref{theorem:integer-valued-linear-growth}. The rest of this section concerns the {\it Abelian sandpile Markov chain}.
	The chain is introduced in Section \ref{subsec:sandpile-chain}. 
	We then discuss the connection to Theorem \ref{theorem:fast-decay-stronger} 
	and also give an argument which shows, despite this result, that the Abelian sandpile Markov chain mixes (up to constants) as slowly on the supercritical percolation cluster as it does on the full lattice. 
	
	This is mostly an expository section and several open questions are presented. The only new material is Theorem \ref{theorem:slow-mixing-percolation} which is a mixing result for the sandpile on the percolation cluster.

	\subsection{Sandpile growth}  \label{subsec:sandpile-limit-shape}
	Condition the origin to be contained in the infinite cluster $\mathscr{C}_{\infty}$ of $\Z^d$. Start with $n$ chips at the origin in the cluster and topple
	unstable vertices until every vertex is stable. Denote the final collection of chips by $s_n: \mathscr{C}_{\infty} \to \{0, \ldots, \deg_{\mathscr{C}_{\infty}} -1 \}$.
	Simulations show that when $n$ is large the support of $s_n$ approximates a Euclidean ball intersected with the cluster --- this has been documented by Sadhu and Dhar \cite{sadhu2011effect}. As we explain further below, Theorem \ref{theorem:integer-valued-linear-growth}
	and Proposition \ref{prop:rational-matrix} provides some evidence that this indeed occurs. 
	\begin{conj} \label{conj:sandpile-limit-shape}
		Suppose $\mathfrak{p} \in (\mathfrak{p}_c(d), 1)$. 
		Almost surely, as $n \to \infty$, the support of the sequence of rescaled sandpiles
		$\overline{s}_n(x) := s_n([n^{1/d} x])$ converges in $\R^d$ to a ball.  
		Moreover, $\overline{s}_n$ converges weakly-* to a function which takes on two values; $\bar{s}: \R^d \to \{0, s_{\mathfrak{p}}\}$
		where $s_{\mathfrak{p}} > 0$ is a constant depending only on $\mathfrak{p}$ and $d$. 
	\end{conj}

	The significance of the previous conjecture is in its contrast with the 
	sandpile on $\Z^d$ (and other periodic lattices) \cite{pegden-smart-sandpile-2013, apollonian-matrix, apollonian-sandpile, bou2021integer, bou2021shape} --- see Figure \ref{fig:sandpile-pics}. In this case, one starts with $n$ chips at the origin on, say, $\Z^d$ on a (possibly random) initial {\it background}
	$\eta: \Z^d \to \Z$,  which can be thought of as an initial arrangement of chips,
	and runs the same toppling dynamics as before. 
	
	It is known in this case (under some mild assumptions on $\eta$, \eg, i.i.d\ and bounded from above by $d$) that the sandpile has a deterministic scaling limit described by the Laplacian of the solution to a fully nonlinear elliptic PDE \cite{bourabee-sandpile-2021}, the {\it sandpile PDE}. 
	The (canonical) example $\eta \equiv 0$ was first considered rigorously in \cite{pegden-smart-sandpile-2013}.
	\begin{figure}
		\begin{center}
			\includegraphics[width=0.3\textwidth]{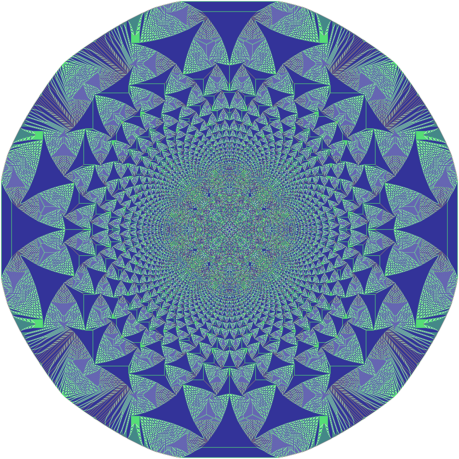}
			\includegraphics[width=0.3\textwidth]{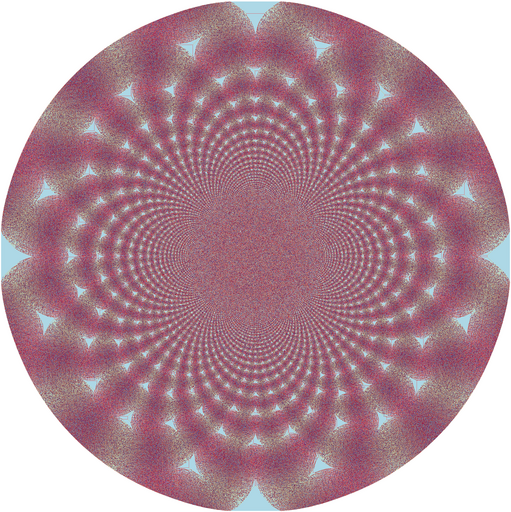}
			\includegraphics[width=0.3\textwidth]{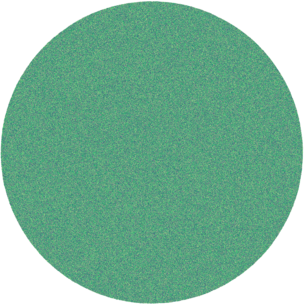}
		\end{center}
		\caption{A comparison of the scaling limit described in Conjecture \ref{conj:sandpile-limit-shape}
			to the sandpile on $\Z^2$. The color of each pixel 
			represents the number of chips in the stable sandpile $s_n$. 
			{\bf Left:} $\Z^2$ with initial background $\eta \equiv 0$.
			{\bf Middle:} $\Z^2$ with i.i.d.\ random initial background $\eta \sim \mbox{Bernoulli}(1/2)$.
			{\bf Right:} Percolation cluster with $\eta \equiv 0$  and $\mathfrak{p} = 0.9$.		
		} \label{fig:sandpile-pics}
	\end{figure}

	The sandpile PDE has a delicate dependence on both the lattice and the distribution of the random initial background, $\eta$. In the case when $\eta \equiv 0$ and the underlying 
	lattice, $\mathbb{L}$, is embedded into $\R^d$, it is characterized
	by a set of {\it allowed Hessians}, the set of $d \times d$ symmetric matrices which have an integer-valued superharmonic representative 
	with that growth at infinity:
	\begin{equation} \label{eq:allowed-hessians}
		\Gamma_{\mathbb{L}} := \{ M \in \symm^d : \exists u: \mathbb{L} \to \Z \mbox{ such that }  \mbox{$u(x) \geq \frac{1}{2} x^T M x + o(|x|^2)$ and $\Delta_{\mathbb{L}} u \leq 0$}\}.
	\end{equation}
	These sets are conjectured to have a rich, fractal description for all periodic graphs and this is known rigorously in two cases,  $\Z^2$ \cite{apollonian-matrix} and the $F$-lattice, \cite{bou2021integer}.
	The set $\Gamma$ provides an explanation for the kaleidoscopic patterns which appear
	in large sandpiles \cite{pegden-smart-stability, apollonian-sandpile}.
	
	Both proofs \cite{apollonian-matrix,bou2021integer} proceed recursively
	and the `base cases' require the existence of integer-valued harmonic functions of quadratic growth, \eg, on $\Z^2$ the function $x \to (x_1^2 - x_2^2)$. 
	The lack of such `base cases' on the cluster, Proposition \ref{prop:rational-matrix}, suggest that $\Gamma_{\mathscr{C}_{\infty}}$ is a trivial set. 
	\begin{conj} \label{conj:sandpile-percolation}
		Suppose $\mathfrak{p} \in (\mathfrak{p}_c(d), 1)$.
		There exists a deterministic constant $s_p(\mathfrak{p}, d)$ such that, almost surely
		\[
		\Gamma_{\mathscr{C}_{\infty}} = \{ M \in \symm^d : \Tr(M) \leq s_p \}. 
		\]
	\end{conj}
	Unlike the lattice, on the percolation cluster we do not yet have a rigorous relationship between $\Gamma_{\mathscr{C}_{\infty}}$ and the large scale properties of sandpiles.
	In fact, it is not known that the scaling limit of the sandpile on the percolation cluster exists. This is because the programs   
	in \cite{pegden-smart-sandpile-2013} and \cite{bourabee-sandpile-2021} presuppose {\it affine-invariance} of the sandpile PDE --- that is, that the sandpile PDE is purely second order. This affine-invariance is an easy consequence of the existence of linear growth integer-valued harmonic functions and hence such functions are an important feature of the proofs in \cite{pegden-smart-sandpile-2013} and \cite{bourabee-sandpile-2021}.
	Consequently, due to Theorem~\ref{theorem:integer-valued-linear-growth}, establishing the following will require new ideas. 
	\begin{conj} \label{conj:sandpile-limit-shape-exists}
		Suppose $\mathfrak{p} \in (\mathfrak{p}_c(d), 1)$. There exists a deterministic $\bar{s}: \R^d \to \R$ such that almost surely, as $n \to \infty$, $\overline{s}_n$ converges weakly-* to $\bar{s}$.
	\end{conj}
	In the remainder of this section we shift our attention from the patterns appearing in sandpiles on infinite graphs to the statistical properties 
	of sandpiles on finite graphs. 
	
	\subsection{Sandpile Markov chain and toppling invariants} \label{subsec:sandpile-chain}
	In order to exposit the cases of both the full lattice and the percolation cluster, take  $\mathfrak{p} \in (\mathfrak{p}_c(d), 1]$.
	We consider the discrete time sandpile Markov chain on $\mathscr{C}_{*}(Q_n)$ and restrict to the event that $\mathscr{C}_{*}(Q_n) \subset \mathscr{C}_{\infty} \cap Q_n$ and $\partial Q_n \cap \mathscr{C}_{*}(Q_n) \neq \emptyset$.  
	Start with a fully saturated sandpile, $s_0 \equiv (\deg -1 )$ on  $\mathscr{C}_{*}(Q_n)$.  Each discrete time step, $s_t \to s_{t+1}$ is as follows: 
	\begin{enumerate}
		\item Pick a site uniformly at random in $\mathscr{C}_{*}(Q_n)$  and add a chip.
		\item Stabilize with dissipating boundary conditions; that is, each time there is a site with at least $\deg_{\mathscr{C}_{\infty}}$ chips on it, the site topples, 
		losing $\deg_{\mathscr{C}_{\infty}}$ chips and giving one chip to each of its neighbors. If the site is on the boundary, \ie, has a neighbor not in $Q_n$, chips are lost across that edge and thus eventually this process stops.
	\end{enumerate}
	The stationary distribution of this Markov chain is uniform over a subset of sandpiles called {\it recurrent sandpiles}.
	Uniform recurrent sandpiles are of interest in the physics literature \cite{dhar-algebraic-aspects}
	and thus it is important to determine how many steps are needed to run the above Markov chain before the resulting sandpile is close to uniform recurrent. 
	
	Techniques from the theory of random walks on finite groups were used in \cite{jerison-levine-pike-2019} to 
	study this problem and it is precisely through this theory that functions with integer-valued Laplacian become relevant. We briefly recall
	what was proved by Jerison, Levine, and Pike and refer the interested reader to \cite{jerison-levine-pike-2019} for more details. 
	
	The set of of recurrent sandpiles has the structure of an Abelian group where the operation is pointwise addition followed by stabilization and the sandpile Markov chain is a random walk on that group. Its eigenvectors are given by a dual group which can be expressed as the additive group of functions
	$\xi: \overline{\mathscr{C}_{*}(Q_n)} \to \R / \Z$ such that $\xi(\partial \mathscr{C}_{*}(Q_n)) \equiv 0$ and $\Delta_{\mathscr{C}_{\infty}} \xi \in \Z$.

	Functions in the dual group represent quantities which are invariant under toppling and were first considered in \cite{dhar-algebraic-aspects}.
	Specifically, if $\Delta_{\mathscr{C}_{\infty}} u \in \Z$, $s_0$ is the initial (possibly unstable) sandpile
	and $s_{\infty}$ is the final, stable sandpile, then an integration by parts shows that
	\begin{equation} \label{eq:toppling-invariant}
		\sum s_0 u = \sum s_{\infty} u \mod \Z.
	\end{equation}
	Along with the application to the Markov chain we discuss in the next section, {\it toppling invariants} were also used in \cite[Theorem 1]{sandpiles-square-lattice} to demonstrate a certain non-universality
	of random sandpiles on $\Z^2$.  We expect Theorem \ref{theorem:fast-decay} to also be useful in this regard.

	\subsection{Slow mixing} \label{subsec:slow-mixing-sandpile}
	The mixing time of the sandpile Markov chain is controlled by the frequencies with eigenvalues
	\begin{equation} \label{eq:eigenvalue-frequency}
		\lambda_{\xi} = \frac{1}{|\mathscr{C}_{*}(Q_n)|} \sum_{x \in \mathscr{C}_{*}(Q_n)} e^{2 \pi i \xi(x)} 
	\end{equation}
	which are close to one. This leads us to {\it multiplicative harmonic} functions, functions  $h: \overline{\mathscr{C}_{*}(Q_n)} \to \C$ such that $h(\partial \mathscr{C}_{*}(Q_n)) \equiv 1$ and 
	\[
	h(v)^{\deg(v)} = \prod_{u \sim v} h(u), \quad \forall v \in \mathscr{C}_{*}(Q_n)
	\]
	and, we see that, by definition, 
	\begin{equation} \label{eq:harmonic-mod-harmonic}
		\mbox{$\Delta_{\mathscr{C}_{\infty}} u \in \Z$ $\iff$ $e^{2 \pi i u}$ is multiplicative harmonic}.
	\end{equation} 
	By associating a frequency with a multiplicative harmonic function as in \eqref{eq:harmonic-mod-harmonic}, 
	each eigenvalue may be indexed by a multiplicative harmonic function $h$,
	\[
	\lambda_{h} = \frac{1}{|\mathscr{C}_{*}(Q_n)|} \sum_{x \in \mathscr{C}_{*}(Q_n)}  h(x).
	\]
	By discrete Fourier analysis, (see, \eg, \cite[Lemma 2.9]{jerison-levine-pike-2019}), after $t$ discrete time steps the $L^2$ distances from the uniform recurrent distribution $U$ are 
	\begin{equation} \label{eq:multiplicative-harmonic-eigenvalue-diff}
		\| P_t - U\|^2_2 = \sum_{h \in \mathcal{H} \setminus \{1\}} | \lambda_h|^{2 t}
	\end{equation}
	where $\mathcal{H}$ is the set of multiplicative harmonic functions on $\mathscr{C}_*(Q_n)$,  $P^t$ is the transition matrix for $t$ steps of the chain, and for two measures $\mu$ and $\nu$ on the sandpile group, $\mathcal{G}$,
	\[
	\| \mu - \nu \|_2 := \left( |\mathcal{G}| \sum_{g \in \mathcal{G}} |\mu(g) - \nu(g)|^2 \right)^{1/2}.
	\]
	Mixing times of sandpiles on general graphs were first studied rigorously in \cite{jerison-levine-pike-2019}
	and a cut-off theorem was proved for sandpiles on the two-dimensional torus in \cite{sandpiles-square-lattice}
	which was later extended to general periodic tiling graphs in \cite{hough-cutoff, hough-spectrum}.  
	In particular, it was shown, via an analysis of multiplicative harmonic functions on growing subsets of $\Z^d$  that the sandpile Markov chain we consider, when $\mathfrak{p} = 1$, has a mixing time of order $n^d \log n$ \cite{hough-cutoff, hough-spectrum}. The extra log factor 
	leads to `slow mixing'. 
	
	\begin{figure}
		\begin{center}
			\includegraphics[width=0.25\textwidth]{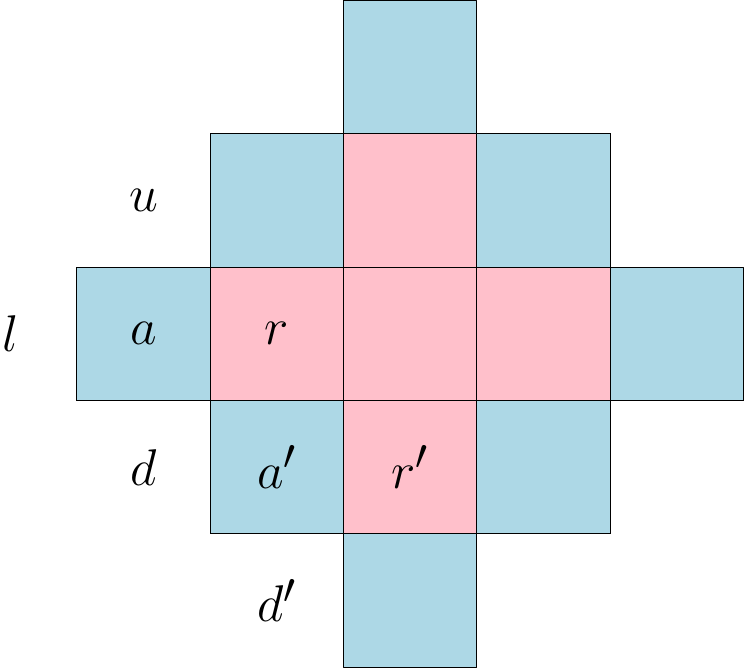}
		\end{center}
		\caption{Visual aid to the proof of Proposition \ref{prop:integer-valued-compact-support}. The larger diamond $D_m$ is in light blue and smaller diamond $D_{m+1}$ is in pink.}\label{fig:recurse-diamond}
	\end{figure}
	
	The sandpile Markov chain on the percolation cluster, \ie, for $\mathfrak{p} \in (\mathfrak{p}_c(d), 1)$ also mixes slowly, but for different reasons which stem from the following.  
	\begin{prop} \label{prop:integer-valued-compact-support}
		If $u: \Z^d \to \R$ has an integer-valued Laplacian and is compactly supported, then $u$ is integer-valued. 
	\end{prop}
	\begin{proof}
		We give the proof in two dimensions. 
		The idea is to induct along diamonds of decreasing size, arguing that at each step, the boundary of the diamond must consist entirely of points where $u$ is integer-valued. 
		Specifically, let $D_1$ be the smallest diamond containing the support of $u$, \ie, some translation of 
		\[
		D(k) : = \{ x \in \Z^2: |x| \leq k \}. 
		\]
		and let $D_1 \supset D_2 \supset \cdots$ be a sequence of diamonds each a translation of $D(k-1), D(k-2), \ldots, \{0\}$
		which exhaust the support of $u$.

		We claim, by induction on $m \geq 1$ that the function $u$ on the outer boundary of each $D_m$ is integer-valued
		\begin{equation} \label{eq:inductive-hypothesis}
			u(\partial D_m) \in \Z .
		\end{equation}
		As $u$ is compactly supported, by definition of $D_1$, $u(\partial D_1) = 0$ and hence the base case is satisfied. 
		Suppose now, by induction, that \eqref{eq:inductive-hypothesis} holds for $m$ and we seek to show it for $m+1$. 
		
		See Figure \ref{fig:recurse-diamond}. The site $a$, at the corner of $D_m$ has three neighbors in $D_{m+1}^c$, $l, u, d$. 
		By evaluating the graph Laplacian at $a$, we see that 
		\[
		-4 a + l + u + d + r \equiv 0 \mod 1
		\]
		and since $a, l,u,d$ are integer-valued, we have 
		\[
		r \equiv 0 \mod 1.
		\]
		We now continue to the site $a'$, we see that 
		\[
		-4 a' + d + r + r' + d' \equiv 0 \mod 1
		\]
		and since $a', d',d ', r$ are integer-valued, we have 
		\[
		r' \equiv 0 \mod 1.
		\]
		We may iterate to see that every site on the boundary of $D_{m+1}$ is integer-valued. 
	\end{proof}
	The previous proposition shows that on $\Z^d$, the only frequencies which can achieve a large non-trivial eigenvalue \eqref{eq:eigenvalue-frequency} are those corresponding to functions which are {\it not} compactly supported.
	However, this fact is false on the percolation cluster and the Markov chain mixes slowly on $\mathscr{C}_{\infty}$ for 
	this reason. See Figure \ref{fig:slow-mixing-gadget}. 
	\begin{figure}
		\begin{center}
			\fbox{\includegraphics[width=0.25\textwidth]{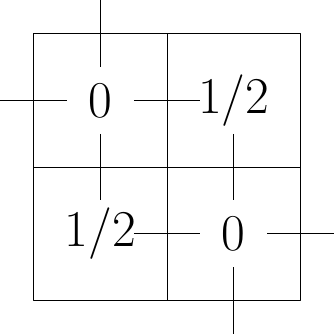}} \qquad
			\fbox{\includegraphics[width=0.25\textwidth]{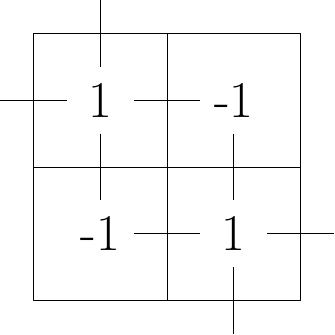}}
		\end{center}
		\caption{On the left, a piece of a non-integer-valued compactly supported function with integer-valued Laplacian on a percolation cluster. 
			The function is identically zero outside of the indicated square.  On the right is the corresponding multiplicative harmonic function as in \eqref{eq:harmonic-mod-harmonic}, which is identically one outside of the square.} \label{fig:slow-mixing-gadget}
	\end{figure}

	\begin{theorem} \label{theorem:slow-mixing-percolation}
		There exists positive constants $c(\mathfrak{p}, d), C(\mathfrak{p},d)$ such that the following holds almost surely.
		For any $\varepsilon > 0$, for all $n$ sufficiently large, 
		\begin{equation} \label{eq:l2-mixing-time-upper-bound}
			\| P^t - U\|_2^2 \leq \varepsilon \quad \mbox{for all $t \geq C n^d \log n$}
		\end{equation}
		and
		\begin{equation} \label{eq:l2-mixing-time-lower-bound}
			\| P^t - U\|_2^2 \geq \varepsilon \quad \mbox{for all $t \leq c n^d \log n$},
		\end{equation}
		where $P^t$ and $U$ are as in \eqref{eq:multiplicative-harmonic-eigenvalue-diff}. 
	\end{theorem}
	\begin{proof}
		First recall that there exists a constant $\theta(\mathfrak{p}) \in (0,1)$ such that, almost surely, 
		\begin{equation} \label{eq:perc-cluster-has-density}
			\lim_{n \to \infty} \frac{|\mathscr{C}_{*}(Q_n)| }{n^d} = 	\lim_{n \to \infty} \frac{|\mathscr{C}_{\infty} \cap Q_n|}{n^d} = \theta(\mathfrak{p}).
		\end{equation}
		By \cite[Theorem 4.3]{jerison-levine-pike-2019}, writing $m := |\mathscr{C}_{*}(Q_n)|$ we have 
		\[
		\| P^t - U\|_2^2 \leq \varepsilon \quad \mbox{for all $t \geq  (5/4) m \log \left( 2 (m-1)(1 + 1/\varepsilon) \right)$}
		\]
		and thus the upper bound \eqref{eq:l2-mixing-time-upper-bound} follows from the previous two displays. 
		
		The lower bound follows by explicitly constructing eigenvectors which obstruct mixing. Specifically, 
		consider a square isomorphic to $[0, 1] \times [0,1] \times \{0\}^{d-2}$ with all edges removed
		except those shown in Figure \ref{fig:slow-mixing-gadget}.  The multiplicative harmonic function, $h: \mathscr{C}_*(Q_n) \to \R$, as illustrated in the right of Figure \ref{fig:slow-mixing-gadget}, has eigenvalue 
		\[
		\frac{1}{m} \sum_{x \in \mathscr{C}_{*}(Q_n)} h(x) =  \frac{1}{m} \left( m - 2 \right) = (1 - \frac{2}{m}). 
		\]
		This prescription of edges appears in the percolation cluster with positive probability 
		and thus, by the ergodic theorem, has positive density in  $\mathscr{C}_{\infty}$. Consequently, the eigenvalue $(1 - \frac{2}{m})$ has multiplicity of order $n^d$. Therefore, by \eqref{eq:multiplicative-harmonic-eigenvalue-diff}
		\[
		\| P_t - U\|^2_2 = \sum_{h \in \mathcal{H} \setminus \{1\}} | \lambda_h|^{2 t}
		\geq c m (1- \frac{2}{m})^{2 t}
		\]
		which implies, by 	\eqref{eq:perc-cluster-has-density}, the desired lower bound \eqref{eq:l2-mixing-time-lower-bound} with a smaller choice of $c$. 
	\end{proof}
	\begin{figure}
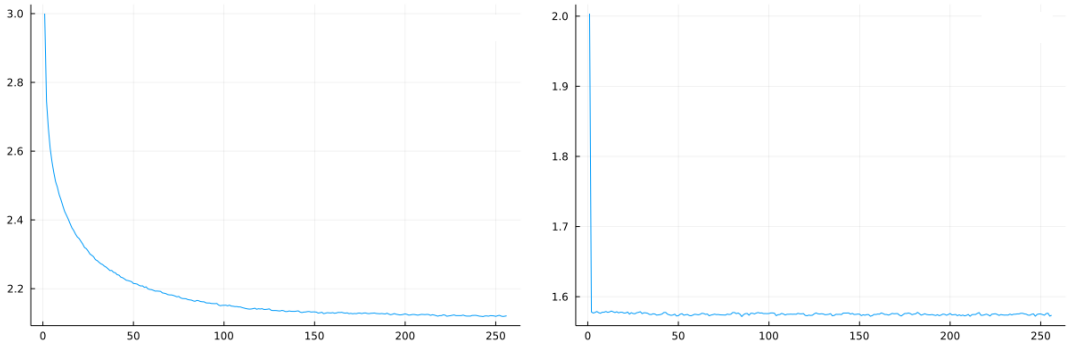

		\begin{center}
			\includegraphics[width=0.43\textwidth]{Figures/p_1.png}
			\includegraphics[width=0.43\textwidth]{Figures/p_75.png}
		\end{center}
		\caption{Density of the sandpile while running the Markov chain on $\Z^2$ (left) and the percolation cluster with $\mathfrak{p} = 3/4$ (right) 
			on $\mathscr{C}_*(Q_n)$ with $n = 10^3$. Each unit on the x-axis denotes an order of $n$ steps of the sandpile Markov chain. The $y$-axis is the average value of the sandpile, $(s_n)_{(\mathscr{C}_*(Q_n))}$. } \label{fig:sandpile-density-figure}
	\end{figure}

	Using \cite[Lemma 27]{sandpiles-square-lattice} and the Cauchy-Schwarz inequality, the $L^2$ bounds in Theorem \ref{theorem:slow-mixing-percolation}
	may be improved to bounds in the total variation metric. This, however, still leaves open the possibility of certain statistics of the sandpile Markov chain on the percolation cluster mixing faster. On $\Z^2$, this phenomena was observed experimentally in \cite{sokolov2015memory}. 
	The simple structure of the toppling invariants as guranteed by Theorem \ref{theorem:fast-decay-stronger} may be useful in making this rigorous on the percolation cluster. In particular, we have shown, at least in two dimensions, that the only eigenvectors which 
	obstruct mixing correspond to finitely supported functions on the cluster. 
	\begin{problem} 
		Show that the density of the sandpile Markov chain on the cluster mixes in time $o(N^d \log N)$, see Figure \ref{fig:sandpile-density-figure}.
	\end{problem} 
	We also leave open the possibility of proving cutoff, as in \cite{sandpiles-square-lattice, hough-cutoff, hough-spectrum}, to future work.

	\bibliographystyle{alpha}
	\bibliography{percolation-harmonic.bib}

\end{document}